\documentclass[11pt,notitlepage]{article}
\usepackage{amsthm,graphicx,placeins,enumitem}
\usepackage[linesnumbered,ruled,norelsize]{algorithm2e}

\usepackage{amsmath}
\usepackage{amssymb}
\usepackage[toc,page]{appendix}
\usepackage{array}
\usepackage[noblocks]{authblk}
\usepackage[english]{babel}
\usepackage{bbm}
\usepackage{booktabs}
\usepackage[font=small]{caption}
\usepackage{changepage}
\usepackage{color,colortbl}

\usepackage{tikz-qtree}

\usepackage{endnotes}
\usepackage{bm}
\usepackage{enumitem}
\usepackage{float}
\usepackage[T1]{fontenc}
\usepackage[hang,flushmargin]{footmisc}
\usepackage{graphicx}
\usepackage{subcaption} 
\usepackage{epstopdf}
\usepackage[left=1in,right=1in,top=1in,bottom=1in,footnotesep=0.8cm]{geometry}
\usepackage{hyperref}
\hypersetup{colorlinks,
	linkcolor=blue,
	filecolor=blue,
	urlcolor=black,
	citecolor=blue}
\allowdisplaybreaks
\usepackage{cleveref}

\usepackage[utf8]{inputenc}
\usepackage{csquotes}
\usepackage{lipsum}
\usepackage{longtable}
\usepackage{makecell}
\usepackage{mathrsfs}
\usepackage{mathtools}
\usepackage{multirow}
\usepackage{mwe}
\usepackage[round]{natbib}
\usepackage{relsize}

\usepackage{romanbar}
\usepackage{setspace}
\usepackage{textcomp}
\usepackage{thmtools,thm-restate}
\usepackage{titlesec}
\usepackage[colorinlistoftodos]{todonotes}
\usepackage{url}
\allowdisplaybreaks

\newtheorem{example}{Example}

\DeclareMathOperator*{\argmax}{arg\,max}

\newcommand{\E}{\mathbf{E}}
\newcommand{\R}{\mathbb{R}}
\newcommand{\Z}{\mathbb{Z}}
\newcommand{\1}{\mathbbm{1}}

\newcommand{\mS}{\mathcal{S}}

\newcommand{\mW}{\mathcal{W}}

\newcommand{\mZ}{\mathcal{Z}}

\newcommand{\overbar}[1]{\mkern 1.5mu\overline{\mkern-0.5mu#1\mkern-0.5mu}\mkern 1.5mu}

\newtheorem{assumption}{Assumption}

\newtheorem{definition}{Definition}
\newtheorem{lemma}{Lemma}

\titleformat*{\section}{\Large\bfseries}
\titleformat*{\subsection}{\large      \bfseries}
\titleformat*{\subsubsection}{\normalsize\bfseries}
\titleformat*{\paragraph}{\normalsize\bfseries}
\titleformat*{\subparagraph}{\normalsize\bfseries}

\titlespacing{\paragraph}{0pt}{0.3\baselineskip}{0.5em}

\providecommand{\keywords}[1]
{
	\small	
	\textbf{\textit{Keywords:}} #1
}

\begin{document}
\doublespacing
	
\title{\Large \textbf{Structured Actor-Critic for Managing Public Health Points-of-Dispensing}}
\author[1]{Yijia Wang}
\author[1]{Daniel R. Jiang}
\affil[1]{Department of Industrial Engineering, University of Pittsburgh}
\affil[ ]{\textit {\{yiw94,drjiang\}@pitt.edu}}
\date{}

\maketitle

\defcitealias{pennsylvania2019}{NFRP, 2019}
\defcitealias{county2017}{DOD, 2017}
\defcitealias{allegheny2010}{Allegheny, 2010}
\defcitealias{dhhr2018}{West Virginia DHHR, 2018}
\defcitealias{michigan2017}{Michigan DHHS, 2017}
\defcitealias{novel2009emergence}{H1N1 Virus Investigation Team, 2009}

\defcitealias{cdc2009H1N1}{CDC, 2009a}
\defcitealias{cdc2009H1N1vaccine}{CDC, 2009b}
\defcitealias{cdc2012VFC}{CDC, 2012}
\defcitealias{cdc2014VFC}{CDC, 2014}
\defcitealias{cdc2018hepatitisA}{CDC, 2019}
\defcitealias{cdc2021how}{CDC, 2021a}
\defcitealias{cdc2021risk}{CDC, 2021b}
\defcitealias{cdc2021trends}{CDC, 2021c}
\defcitealias{cdc2021understanding}{CDC, 2021d}

\defcitealias{kentucky2018Hepatitis}{Kentucky DPH, 2018}
\defcitealias{ohio2019hepatitis}{Ohio Department of Health, 2019}
\defcitealias{florida2019hepatitis}{Florida Health, 2019}
\defcitealias{michigan2018Hepatitis}{Michigan DHHS, 2018}
\defcitealias{westvirginia2019hepatitis}{West Virginia DHH, 2019}
\defcitealias{tennessee2019hepatitis}{Tennessee Department of Health, 2019}
\defcitealias{cdc2018vaccineStorage}{U.S. HHS and CDC, 2018}

\newcommand{\red}[1]{\textcolor{red}{#1}}
\newcommand{\magenta}[1]{\textcolor{magenta}{#1}}

\begin{abstract}
    \looseness-1 Public health organizations face the problem of dispensing treatments (i.e., vaccines, antibiotics, and others) to groups of affected populations through ``points-of-dispensing'' (PODs) during emergency situations, typically in the presence of complexities like {demand stochasticity, heterogenous utilities (e.g., for vaccine distribution, certain segments of the population may need to be prioritized), and limited storage.}
	We formulate a {hierarchical Markov decision process (MDP) model with two levels of decisions (and decision-makers): the upper-level decisions come from an inventory planner that ``controls'' a lower-level dynamic problem, which optimizes dispensing decisions that take into consideration the heterogeneous utility functions of the random set of PODs.}
	We then derive structural properties of the MDP model and propose an approximate dynamic programming (ADP) algorithm that leverages structure in both the \emph{policy} and the \emph{value} space (state-dependent basestocks and concavity, respectively).
	The algorithm can be considered an \emph{actor-critic} method; to our knowledge, this paper is the first to jointly exploit policy and value structure within an actor-critic framework. We prove that the policy and value function approximations each converge to their optimal counterparts with probability one and provide a comprehensive numerical analysis showing improved empirical convergence rates when compared to other ADP techniques. 
	Finally, we show how an aggregation-based version of our algorithm can be applied in a realistic case study for the problem of dispensing naloxone (an overdose reversal drug) via first responders amidst the ongoing opioid crisis.
\end{abstract}


\section{Introduction} \label{sec:intro}

Public health organizations manage ``points-of-dispensing'' (PODs), operated by first responders or first receivers \citep{ablah2010large}, for distributing critical medical supplies during emergency situations (e.g., the ongoing opioid crisis, the COVID-19 pandemic, the 2009 H1N1 influenza pandemic, meningitis outbreaks). In this paper, we consider the hierarchical and sequential problem of optimizing inventory control and making dispensing decisions for multiple PODs. Our problem setting is specifically motivated by the ongoing opioid overdose \emph{harm reduction} efforts of public health organizations in cities across the U.S., where the opioid epidemic was declared a public health emergency in 2017. In particular, our modeling is motivated by the Naloxone for First Responders Program (NFRP), a statewide naloxone distribution {initiative} in Pennsylvania.

\looseness-1 {Our setup contains \emph{hierarchical decisions} in order to model the interplay between two decision-makers: the ``upper-level'' central inventory manager and a ``lower-level'' dispensing coordinator. The NFRP is an example of an organization that operates in this manner through the use of a centralized coordinating entity (CCE) that manages dispensing. Both decision-makers make sequential and non-myopic decisions (at different timescales) and are modeled using Markov decision processes (MDPs).} Another novel point of emphasis for our model is the notion that the effectiveness of the public health intervention can vary across different groups of the affected population \citep{lee2015vaccine, doi:10.1080/02791072.2018.1430409} and across different locations. 
Therefore, instead of modeling demand in a static and homogeneous manner, we consider the case where at each period, new demand information is revealed sequentially as a POD attribute and demand. The dispensing decisions are made according to the arrivals of PODs with a goal of maximizing total utility. {The essential trade-off considered by the dispensing coordinator is: \emph{should we satisfy a lower-priority demand now, or save the inventory for a possible higher-priority demand in the future?}}

The model we develop in this paper, however, is quite general and useful for related problems in public health as well where hierarchical decisions and demand heterogeneity may be an issue {(e.g., vaccine distribution, where certain segments of the population are more susceptible and where higher-level inventory and lower-level dispensing decisions might be made separately).} Other important characteristics of this problem include demand nonstationarity and the potential for limited storage capacity.  

Exact computation of the optimal policy for this model is difficult when the number of states is large, when the stochastic models are unknown, or when demand data is collected slowly over time. The main methodological contribution of the paper addresses these issues through a \emph{structured actor-critic} algorithm; our proposed method exploits structure in both the \emph{policy} and the \emph{value function} and can discover near-optimal policies in a fully data-driven way. Our algorithm uses several gradient updates on each iteration and thus is highly suitable for the situation where data arrives in an ongoing fashion and online updates are desired. In other words, a large batch of historical data is not required for our algorithm and the policy can be learned over time. We now give five examples of public health problems for which our model and algorithm are applicable.

\begin{example}[Opioid Overdose Epidemic] 
    \looseness-1 The rate of opioid overdose deaths tripled between 2000 and 2014 in the United States \citep{rudd2000increases}. More recently, in July 2017, it was estimated that there are 142 American deaths each day due to overdose \citep{christie2017president}. Naloxone is a drug that has the ability to reverse overdoses within seconds to minutes. To save lives amidst the current opioid epidemic, it is critical for naloxone to be widely distributed. Indeed, many harm reduction programs such as NFRP are undertaking the challenge by distributing naloxone free of charge to first responders. {The} NFRP program is run by Pennsylvania Commission on Crime and Delinquency (PCCD), who dispenses naloxone to eligible first responders through centralized, local hubs in each county or region First responders include emergency medical services, law enforcement, fire fighters, public transit drivers and so on.
    One challenge facing these organizations is that the utility of naloxone varies across different types of first responders. \citet{goodloe2014should, rando2015intranasal} emphasize the importance of law enforcement officers, who are ``often a community’s first contact with opioid overdose victims after 9-1-1 services have been summoned.'' 
    The utility of naloxone also varies across regions {due to the varying levels of opioid usage in different populations}. The West Virginia Department of Health and Human Resources (DHHR) purchased about 34,000 doses of naloxone; {in addition to} distributing to the state police, fire departments, and emergency medical services, DHHR additionally planned to distribute 1,000 doses of naloxone to each of the eight high priority counties, including Berkeley, Cabell, Harrison, Kanawha, Mercer, Monongalia, Ohio, and Raleigh \citepalias{dhhr2018}. Therefore, the prioritization of certain ``demand classes'' is an important consideration when naloxone is expensive or when quantities are limited; see, e.g., \citet{cohn2017baltimore} for a report on rationing practices in Baltimore.
\end{example}

\begin{example}[COVID-19]
    By the end of February 2021, COVID-19 has resulted in 28,409,727 cases and 511,903 deaths \citepalias{cdc2021trends}. Compared with 5-17 age group, the rate of death is 1100 times higher in 65-74 age group, 2800 times higher in 75-84 age group, and 7900 times higher in 85 and older age group \citepalias{cdc2021risk}. According to the COVID-19 vaccination recommendations by CDC \citepalias{cdc2021how}, in phase 1a, healthcare personnel and long-term care facility residents are offered vaccination first. After that, the 75 and older age group and the 65-74 age group are vaccinated in phases 1b and 1c.
\end{example}

\begin{example}[Influenza]
    The need for distinct demand classes was also observed for the case of vaccine distribution during the 2009 H1N1 influenza pandemic. The H1N1 influenza virus first emerged in Mexico and California in April 2009 \citep{neumann2009emergence} and the pandemic lasted until August 2010 \citep{world2010h1n1}. Children and young adults were disproportionately affected when compared to older adults \citep{kwan2009spring}: during April 15 and May 5, 2009, among the 642 confirmed infected patients in the U.S. (ranging from 3 months to 81 years old), 60\% were 18 years old or younger \citepalias{novel2009emergence}. The reported H1N1 cases from April 15 to July 24, 2009, show that the infected rate (number of cases per 100,000 population) of 0 to 4 age group is 17.6 times of the infected rate of 65 and older age group, and the rate of 5 to 24 age group is 20.5 times of the rate of 65 and older age group \citepalias{cdc2009H1N1}. 
    The Advisory Committee on Immunization Practice (ACIP) recommended a priority group (about 159 million Americans), in which there was a subset with highest priority (about 62 million Americans) \citep{rambhia2010mass}. Patients aged 65 and older were only considered for vaccination once the demand amongst younger groups were met \citepalias{cdc2009H1N1vaccine}.
\end{example}

\begin{example}[Hepatitis A] 
\looseness-1 Hepatitis A outbreaks began in 2016 and are currently (as of August 2019) ongoing in 29 states across the U.S \citepalias{cdc2018hepatitisA}. Recent data from August 16, 2019 shows  4837 cases (60 deaths) in Kentucky, 3244 cases (15 deaths) in Ohio, 2740 cases (31 deaths) in Florida, 918 cases (28 deaths) in Michigan, 2540 cases (23 deaths) in West Virginia, and 2219 cases (13 deaths) in Tennessee \citepalias{kentucky2018Hepatitis, ohio2019hepatitis, florida2019hepatitis, michigan2018Hepatitis, westvirginia2019hepatitis, tennessee2019hepatitis}. This outbreak largely affects the homeless, drug users, and their direct contacts \citepalias{cdc2018hepatitisA}. Center for Disease Control (CDC) guidelines suggest that vaccine inventory be conducted monthly to ensure adequate supplies and that the vaccine order decisions take into account projected demand and storage capacity \citepalias{cdc2018vaccineStorage}, two important aspects of our model. The CDC also recommends against overstocking, which presents the risk of wastage and outdated vaccines. 
\end{example}

\begin{example}[Vaccines for Children Program]
    \looseness-1 The measles epidemic in 1989 to 1991 revealed the issue of low vaccination rates among children \citepalias{cdc2014VFC}. Vaccines for Children (VFC) is a program started in 1994 that aims to reduce the economic barriers of vaccination for disadvantaged children  \citep{santoli1999vaccines, zimmerman2001effect, smith2005association}. It supplies free vaccines (including influenza, hepatitis A, hepatitis B, and measles) to registered providers, who in turn provide vaccinations to eligible children \citep{zimmerman2001effect, social2005VaccineforChildren}. Before healthcare providers are enrolled, VFC coordinators perform site visits to ensure proper storage practice \citepalias{cdc2012VFC}. A study in 2012 found, however, that out of 45 VFC providers, 76\% exposed VFC vaccines to inappropriate temperatures, and 16\% kept expired VFC vaccines \citep{levinson2012vaccines}. 
\end{example}

\paragraph{Our Results.} 
The main contributions of this paper are summarized below.
\begin{itemize}
	\item In this paper, we first develop and analyze a hierarchical, finite-horizon Markov decision process (MDP) model that abstracts the above problems into a single framework. The upper-level problem (the ``upper-level MDP'') is an inventory model that controls a lower-level dispensing problem (the ``lower-level MDP''). Here, we consider the setting where the utilities of PODs differ across regions due to the varying intervention effects on patients in different populations. The demand and POD-attribute distributions at each period depend on an information process, which can represent past demand realizations or other external information.
	
	\item \looseness-1 We then analyze the structural properties of the model. The MDP features basestock-like structure in a discrete state setting and discretely-concave value functions; both of these properties depend on the discrete-concavity observed in the lower-level problem. The motivation for a discrete state formulation comes from the naloxone distribution application, where demand quantities are relatively small; {this is not an ideal setting for use of a continuous state approximation}.
	
	\item Next, we propose a new actor-critic algorithm \citep{sutton1998reinforcement, konda2000actor} that exploits the structural properties of the MDP. More specifically, the algorithm tracks both policy and value function approximations (an identifying feature of an ``actor-critic'' method) and utilizes the structure to improve the empirical convergence rate. Moreover, the algorithm is suited for a setting where data arrive continually and the policy is updated over time. This algorithm (and its general idea) is potentially of broader interest, beyond the public health application.
	
	\item Finally, we present a case study for the problem of dispensing naloxone. We show how an aggregation-based version of the algorithm can be applied in a setting with continuous information states. 
	In addition to computing approximations to the optimal inventory management and dispensing strategies, we also conduct a sensitivity analysis to understand the impact of various model parameters. 
\end{itemize}

The paper is organized as follows. A literature review is provided in Section~\ref{sec:lit_review}. We introduce the hierarchical MDP model in Section~\ref{sec:model} and derive its structural properties in Section~\ref{sec:structural}. The proposed \emph{stuctured actor-critic} algorithm is given and discussed in Section~\ref{sec:alg}. In Section~\ref{sec:numerical}, we conduct numerical experiments.
We propose an aggregation-based version of the algorithm in Section~\ref{sec:aggregation} and finally present the naloxone case study in Section~\ref{sec:case_study}.

\section{Literature Review} \label{sec:lit_review}

In this section, we provide a brief review of related literature. 
The upper-level replenishment decisions in this paper are closely related to both lost-sales and perishable inventory models. In the lost-sales case, \citet{nahmias1979simple} constructs simple myopic approximations for three variations of the classical model with lead time. \citet{ha1997inventory} studies a single-item, make-to-stock production model with several demand classes and lost sales and constructs stock-rationing levels for the optimal policy. \citet{mohebbi2003supply} focuses on random supply interruptions in lost-sales inventory systems with positive lead times. \citet{zipkin2008old} finds that the standard base-stock policy performs poorly compared to some other heuristic policies. We also refer readers to \citet{bijvank2011lost} for a detailed review. Our public health application is also somewhat related to the problems studied in perishable inventory models \citep{janssen2016literature}, even though our motivating application does not require us to explicitly model age. 

Related to our hierarchical model is the case of multi-echelon systems, where, for example, an upper echelon (e.g. a central warehouse) replenishes the inventory of a lower echelon (e.g. a retailer) that serves demand \citep{clark1960optimal}. \citet{tan1974optimal} studies the optimal ordering and allocation policies for the upper echelon and \citet{graves1996multiechelon} constructs an allocation policy for the multi-echelon system. In the model of \citet{chen2000staggered}, each retailer is allowed to replenish once from the warehouse during an ordering cycle. \citet{van2007optimal} shows the optimality of base-stock policies and derives newsvendor-type equations for the optimal base-stock levels. \citet{zhou2013multi} studies a multi-product multi-echelon inventory system. \citet{grob2018inventory} aims to optimize the reorder points of both echelons given fixed order quantities. Our model expands upon this literature in that the optimization problems for the two echelons are modeled as two nested MDPs (or a ``hierarchical'' MDP). We show the concavity of the value function of both the upper-level and lower-level, which is then utilized to derive the structured actor-critic algorithm.

Our proposed actor-critic method falls under the class of approximate dynamic programming (ADP) or reinforcement learning (RL) algorithms \citep{bertsekas1996neuro-dynamic, sutton1998reinforcement, powell2007approximate}. Possibly the most well-known RL technique is Q-learning \citep{watkins1989learning}, a model-free approach that uses stochastic approximation (SA) to learn state-action value function (or ``Q-function''). In some cases, convexity of the value function is known a priori and can be exploited; see, e.g., \citet{pereira1991stochastic, powell2004learning, nascimento2009optimal, philpott2008convergence, shapiro2011analysis, nils2013optimizing}. The updates used in the value function approximation part of our algorithm most closely resembles \cite{powell2004learning} and \cite{nascimento2009optimal}. 

Related to the policy function approximation part of our algorithm, \citet{kunnumkal2008using} proposes a stochastic approximation method to compute basestock levels in continuous state inventory problems. Our method also utilizes two types of basestock structure, but it does so in a different way from \citet{kunnumkal2008using} due to our focus on discrete-valued inventory states. The primary feature of an actor-critic algorithm is that it approximates both the policy and value function \citep{werbos1974beyond, witten1977adaptive, werbos1992approximate, konda2000actor}. The ``actor'' is the policy function approximation (for selecting actions) and the ``critic'' represents the value function approximation used to ``criticize'' the actions selected by the actor. The novelty of our method is due to its use of the structure in \emph{both the value function and the policy}. Our experimental results show significant advantages of exploiting this policy-value structure.
Further, differing from most actor-critic methods, we do not use stochastic policy,\footnote{Given a state, a stochastic policy returns an action distribution, while a deterministic policy returns a specific action.} reducing the number of policy parameters to be learned.

In addition, state aggregation is a commonly used method to deal with large dynamic problems \citep{fox1973discretizing, bean1987aggregation, singh1995reinforcement}, including inventory management \citep{schweitzer1985iterative, chen1999applying, chen1999application, mousavi2004stochastic, zaher2014optimal}. Error bounds for these types of approximations can be found in \citet{bertsekas1975convergence}, \citet{ren2002state}, and \citet{van2006performance}. Our results in Section~\ref{sec:case_study} make use of partial aggregation of the state space.

Due to the discrete inventory states used in our model, we make use of the concept of $L^\natural$-convexity (concavity) as a tool in the analysis. This theory was first developed in \citet{fujishige2000notes} for discrete convex analysis and then extended to continuous variables by \citet{murota2000extension}. Closely related concepts are $l$-convexity and submodularity. It turns out that these ideas are useful in understanding the structures of optimal policies in the field of inventory management, as introduced by \citet{lu2005order} in an assemble-to-order multi-item system. \citet{zipkin2008structure} uses $L^\natural$-convexity in some variations of the basic multiperiod lost-sales model with lead time and \citet{huh2010optimal} extend the results to lost-sales serial inventory systems. \citet{pang2012note} use similar ideas to analyze inventory-pricing systems with lead time, and \citet{gong2013optimal} study finite capacity systems with both manufacturing and remanufacturing. See \citet{xin2017convexity} for a survey of applications utilizing the theory of $L^\natural$-convexity.

As for the utility in our model, the quality-adjusted life-year (QALY) is widely used to in healthcare to measure the value of medical interventions. The QALY was originally developed for cost-effectiveness analysis to decide scarce resource allocation across competing healthcare programs \citep{fanshel1970health, torrance1972utility, weinstein1977foundations}, and has been endorsed by the US Panel on Cost-Effectiveness in Health and Medicine as a standardized methodological approach in cost-effectiveness analyses \citep{weinstein1996cost}. The QALY has been used in naloxone distribution research to evaluate the cost-effectiveness of distributing naloxone to heroin users \citep{coffin2013cost}, distributing naloxone to adults at risk of heroin overdose in UK \citep{langham2018cost}, and one-time versus biannual distribution \citep{acharya2020cost}.

\section{Model Formulation} \label{sec:model} 
As discussed above, our MDP model is motivated by the hierarchical structure of public health organizations, such as Pennsylvania's NFRP, which distributes naloxone to a CCE that, in turn, coordinates the dispensing of naloxone to first responders in various counties. We assume that the {central inventory manager} makes \emph{replenish-up-to} and \emph{dispense-down-to} decisions to the central storage periodically. Then, naloxone is distributed to the {dispensing coordinator}, who makes dispensing decisions to sequentially and randomly arriving PODs. Given an initial allotment of inventory, the dispensing decisions to PODs are made with the goal of maximizing cumulative utility\footnote{The trade-off here considers, for example, the number of naloxone kits that should be provided to first responders in a neighborhood with high drug overdose death rate versus the first responders in a neighborhood with low drug overdose death rates.} of the satisfied naloxone demand within the dispensing period. 
The timing of events during each period is as follows: (1) the central inventory manager decides the replenish-up-to and dispense-down-to levels, (2) the dispensing coordinator receives naloxone, and (3) based on POD demands, POD attributes, and the level of available inventory, the dispensing coordinator dispenses naloxone in order to maximize utility. Figure~\ref{fig:hierarchical} gives an illustration of the timing of these events. In this section, we first discuss the lower-level dispensing problem and then illustrate the upper-level inventory control model.

\begin{figure}[ht]
	\centering
	\includegraphics[width=0.98\textwidth]{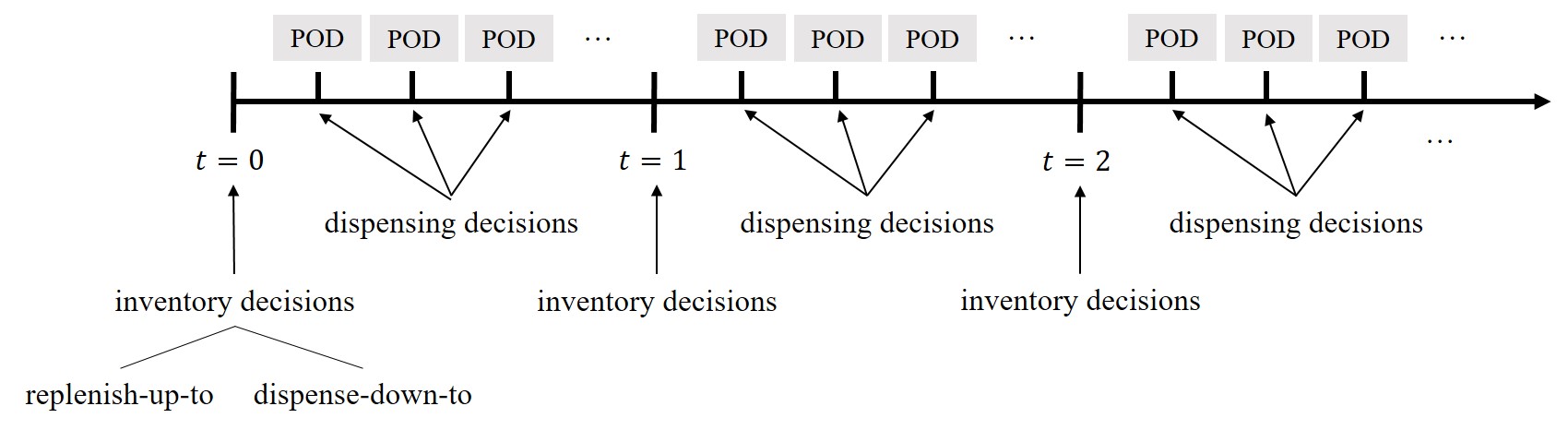}
	\caption{Sequence of Events}
	\label{fig:hierarchical}
\end{figure}

\subsection{The Dispensing MDP}

\looseness-1 {We first discuss the \emph{lower-level} MDP for making dispensing decisions within each period $t$. After the central inventory manager makes the replenish-up-to and dispense-down-to decisions, the dispensing coordinator receives a sequence of POD demands to satisfy, starting with an initial inventory allotment based on the dispense-down-to decision.} The dispensing model contains $n$ sub-periods. 
In sub-period $i$, the arriving POD is represented by an \emph{attribute} $\Xi_{t,i}$ which is interpreted as the arriving POD's attributes. When there is no arriving POD in sub-period i, $\Xi_{t,i}=0$. The distribution of $\Xi_{t,i}$ depends on an exogenous information process $\{W_t\}$ that transitions according to the upper-level timescale (and thus stays fixed at a particular realization $w$ for all sub-periods in the dispensing problem; a full description of this process will be given in Section~\ref{sec:replenishment}). 
Given realizations $w_t$ and $\xi_{t, i}$ of the exogenous information $W_t$ and attribute $\Xi_{t,i}$, we consider an increasing expected utility function $u_{w_t}(\cdot,\xi_{t,i})$, whose argument is the number of inventory units $y_{i}$ dispensed to the arriving POD in sub-period $i$.
For the remainder of this section, we omit the subscript $t$ in for convenience.

These utility functions should be interpreted as parameters specified by the public health organization. The motivation for modeling heterogeneous utilties for the case of naloxone dispensing is primarily due to varying severity of the epidemic across different regions and populations (first responders in regions with more opioid users should have higher priority).
To model this heterogeneity in demand, our model allows region and other related information to be encoded within the attribute $\xi_{i}$, which then determines the utility. 

The dispensing coordinator aims to maximize the total utility subject to the initial inventory allotment $x_0$. In sub-period $i$, given exogenous information $w$, available inventory level $x_i$ and attribute state $\xi_i$ about the arriving POD, a dispensing decision $y_i$ is made.
Let $\{\mu_{w,0},\mu_{w,1},\ldots,\mu_{w,n-1}\}$ be lower-level dispensing policy for exogenous information $w$ and suppose $\mathcal{M}_w$ is the set of all feasible policies that satisfy $\mu_{w,i}(x_i,\xi_i) \leq x_i$.
The objective on the lower-level is given by
\begin{equation*}
 U_{w,0}(x_0, \xi_0) = \max_{\mu_w\in \mathcal{M}_w} \mathbf{E} \left[ \sum_{i=0}^{n-1} u_w \bigl(\mu_{w,i}(X_i,\Xi_i), \Xi_i\bigr) \, \Bigl | \, X_0=x_0, \Xi_0=\xi_0, W=w\right],
\end{equation*}
where the transition follows $X_{i+1} = x_i - \mu_{w,i}(x_i,\xi_i)$. {The optimum is attained by an optimal policy $\mu^*_w$.}
We now write the Bellman optimality equation for the objective:
\begin{equation} \label{eq.disopt}
    \textstyle U_{w,i}(x_i,\xi_i) = \textnormal{max}_{y_i \leq x_i} \; u_w\bigl(y_i,\xi_i\bigr) + \E_w\bigl[ U_{w,i+1}(X_{i+1},\Xi_{i+1}) \bigr]
\end{equation}
for $i=0,1,\ldots,n-1$, and $U_{w,n}(x_{n},\xi_{n}) = 0$, where $\E_w$ is being used as shorthand for the expected value conditioned on $\{W_t=w\}$.

\subsection{The Inventory Control MDP} \label{sec:replenishment}

The sequential inventory control aspect of the model contains $T$ planning periods. In each period, there are two decisions to be made: the replenish-up-to level and the dispense-down-to level. 
In the first period $t=0$, the initial resource level $R_0=0$. In the last period $t=T$, no decision is made and the remaining inventory $R_T$ is either worthless or charged a disposal cost (controlled by a parameter $b\ge 0$). Let $\{W_t\}$ be the aforementioned exogenous information process, which may contain information regarding past POD demands, current disease trends, or other dynamic information related to the public health situation. 
As discussed above, the information state $W_t$ influences the distribution of the attributes $\Xi_{t,i}$ of the arriving PODs for sub-periods $i=1,2,\ldots,n$ of the lower-level problem in period $t$. 
We assume that $W_t$ takes values in a finite set $\mathcal W$ and that it is a Markov process.

Let $R_\textnormal{max}$ be the capacity of the central storage facility. At the end of each period $t$, the central inventory manager makes a replenish-up-to decision $z^\textnormal{rep}_t$ based on the available resource level $R_t \in \{0,1,\ldots,R_\textnormal{max}\}$ and the exogenous information $W_t \in \mathcal W$. After this, the central inventory manager makes a dispense-down-to decision $z^\textnormal{dis}_t$ based on the replenish-up-to decision $z^\textnormal{rep}_t$ and $W_t$. 

We will often refer to particular values of the resource level $R_t$ and exogenous information $W_t$ using the notations $r$ and $w$, respectively. Let $\bar{\mZ}(r) = \{r,r+1,\ldots,R_\textnormal{max}\}$ be the set of feasible replenish-up-to decisions if the current inventory level is $r$, so that $z^\textnormal{rep}_t \in \bar{\mZ}(R_t)$ in period $t$. 
This means the central inventory manager orders $z^\textnormal{rep}_t - R_t$ units of inventory with a per-unit ordering cost $c_w\geq c_\textnormal{min}$ (note that we allow this ordering cost to depend on the exogenous information $w$), {where $c_\textnormal{min}$ is a positive scalar}.

Let $\underline{\mZ}(z^\textnormal{rep}) = \{0,1,\ldots,z^\textnormal{rep}\}$ be the set of feasible dispense-down-to decisions if the current resource level is $z^\textnormal{rep}$, so that $z^\textnormal{dis}_t \in \underline{\mZ}(z^\textnormal{rep}_t)$ in period $t$. This means the central inventory manager delivers $z^\textnormal{rep}_t - z^\textnormal{dis}_t$ units of inventory to the dispensing coordinator, and $z^\textnormal{rep}_t - z^\textnormal{dis}_t$ serves as the initial inventory allotment in the lower-level dispensing MDP problem. 
The transition to the next inventory state $R_{t+1}$ is given by:
\begin{equation} \label{eq.trans}
	\textstyle R_{t+1} = z^\textnormal{dis}_t.
\end{equation}
Each unit of leftover inventory after applying the transition \eqref{eq.trans} is charged a holding cost $h<c_\textnormal{min}$. 

A policy $\{\pi_0, \pi_1, \ldots, \pi_{T-1} \}$ is a sequence of mappings $\pi_t = (\pi_t^\textnormal{rep}, \pi_t^\textnormal{dis})$
from states $(R_t,W_t)$ to replenish-up-to levels and dispense-down-to levels. Let $\Pi$ be the set of all feasible policies that satisfy $\pi_t^\textnormal{rep}(R_t,W_t) \geq R_t$ and $\pi_t^\textnormal{dis}(R_t,W_t) \leq \pi_t^\textnormal{rep}(R_t,W_t)$. Our objective is given by:
\begin{equation*}
    \max_{\pi\in\Pi}\; \mathbf{E} \Biggl[ \sum_{t=0}^{T-1} \bigl( -h R_t - c_{W_t} \bigl(\pi_t^\textnormal{rep}(R_t,W_t)-R_t\bigr) + U_{W_t,0}\bigl(\pi_t^\textnormal{rep}(R_t,W_t) - \pi_t^\textnormal{dis}(R_t, W_t),\, \Xi_{t,0}\bigr) \Biggr] - b\,R_T,
\end{equation*}
where $R_t$ transitions according to \eqref{eq.trans} for $(z^\textnormal{rep}_t,z^\textnormal{dis}_t) = \pi_t(R_t,W_t)$, and the gap between the two decisions of the upper-level problem, $\pi_t^\textnormal{rep}(R_t,W_t) - \pi_t^\textnormal{dis}(R_t, W_t)$, serves as the initial resource level of the lower-level problem. 
We now write a preliminary set of Bellman optimality equations for the objective above. Let $V_T(r,w) = -b\,r$ be the terminal value (note: $b$ is zero if there is no disposal cost). For $t<T$, we have
\begin{equation} \label{eq.prelimbellman}
    \begin{aligned}
	V_t(r,w) = \max_{z^\textnormal{rep}\in\bar{\mZ}(r), z^\textnormal{dis}\in\underline{\mZ}(z^\textnormal{rep})}\, &(c_w-h)\,r - c_w z^\textnormal{rep} + \E_w \bigl[ U_{w,0}\bigl(z^\textnormal{rep}-z^\textnormal{dis}, \Xi_{t,0}\bigr) + V_{t+1}\bigl(z^\textnormal{dis}, W_{t+1}\bigr) \bigr].
	\end{aligned}
\end{equation}

So far, we have considered $z^\textnormal{rep}$ and $z^\textnormal{dis}$ as being made simultaneously in each period, but we can equivalently view the dispense-down-to decision $z^\textnormal{dis}$ to be taken after the replenish-up-to decision $z^\textnormal{rep}$ (this reflects the reality and also is useful for our analysis and algorithm). The set $\underline{\mZ}(z^\textnormal{rep})$ of feasible dispense-down-to decisions is dependent on the replenish-up-to decision $z^\textnormal{rep}$. Therefore, the value function in each period can be broken into two steps:
\begin{equation} \label{eq.bellman_rep}
    \begin{aligned}
	\textstyle V_t^\textnormal{rep}(r,w) = (c_w-h)\,r + \max_{z^\textnormal{rep}\in\bar{\mZ}(r)}\bigl\{ - c_w z^\textnormal{rep} + V_t^\textnormal{dis}(z^\textnormal{rep},w) \bigr\},
	\end{aligned}
\end{equation}
\begin{equation} \label{eq.bellman_dis}
    \begin{aligned}
	\textstyle V_t^\textnormal{dis}(z^\textnormal{rep},w) = \max_{z^\textnormal{dis}\in\underline{\mZ}(z^\textnormal{rep})} \E_w \bigl[U_{w,0} \bigl( z^\textnormal{rep}-z^\textnormal{dis},\Xi_{t,0}\bigr) + V_{t+1}^\textnormal{rep} \bigl(z^\textnormal{dis}, W_{t+1}\bigr) \bigr],
	\end{aligned}
\end{equation}
with $V_T^\textnormal{rep}(r,w) = -b\,r$.
Similarly, there are two postdecision value functions in each period corresponding to the replenish-up-to decision and the dispense-down-to decision respectively:
\begin{equation} \label{eq.postdec_rep}
    \textstyle \tilde{V}_t^\textnormal{rep}(z^\textnormal{rep},w) = - c_w z^\textnormal{rep} + V_t^\textnormal{dis}(z^\textnormal{rep},w),
\end{equation}
\begin{equation} \label{eq.postdec_dis}
    \textstyle \tilde{V}_t^\textnormal{dis}(z^\textnormal{dis},w) = \E_w\bigl[ V_{t+1}^\textnormal{rep}\bigl(z^\textnormal{dis},W_{t+1}\bigr) \bigr].
\end{equation}
The optimal policy can be written as follows
\begin{equation} \label{eq.policy_rep}
	\textstyle \pi_t^{\textnormal{rep},*}(r,w) \in \argmax_{z^\textnormal{rep}\in\bar{\mZ}(r)} \; \tilde{V}_t^\textnormal{rep}(z^\textnormal{rep},w).
\end{equation}
\begin{equation} \label{eq.policy_dis}
	\textstyle \pi_t^{\textnormal{dis},*}(z^{\textnormal{rep}},w) \in \argmax_{z^\textnormal{dis}\in\underline{\mZ}(z^{\textnormal{rep}})} \; \E_w \bigl[ U_{w,0}\bigl(z^\textnormal{rep}-z^\textnormal{dis},\Xi_{t,0}\bigr) \bigr] + \tilde{V}_t^\textnormal{dis}(z^\textnormal{dis},w),
\end{equation}
where, with a slight abuse/reuse of notation, $\pi_t^{\textnormal{dis},*}(z^{\textnormal{rep}},w)$ is the optimal dispense-down-to policy when the replenish-up-to level is $z^{\textnormal{rep}}$.
Combining \eqref{eq.bellman_rep}-\eqref{eq.policy_dis}, we obtain equivalent formulations of the optimality equation written using $\tilde{V}_t^\textnormal{rep}(z^\textnormal{rep},w)$, $\tilde{V}_t^\textnormal{dis}(z^\textnormal{dis},w)$, $\pi_t^{\textnormal{rep},*}(r,w)$, and $\pi_t^{\textnormal{dis},*}(z^\textnormal{rep},w)$:
\begin{equation} \label{eq.postdec_rep_opt}
    \textstyle \tilde{V}_t^\textnormal{rep}(z^\textnormal{rep},w) = - c_w z^\textnormal{rep} + \E_w \bigl[U_{w,0} \bigl( z^\textnormal{rep} - \pi_t^{\textnormal{dis},*}(z^\textnormal{rep}, w), \, \Xi_{t,0} \bigr) \bigr] + \tilde{V}_t^\textnormal{dis} \bigl(\pi_t^{\textnormal{dis},*}(z^\textnormal{rep}, w), w \bigr),
\end{equation}
\begin{equation} \label{eq.postdec_dis_opt}
    \textstyle \tilde{V}_t^\textnormal{dis}(z^\textnormal{dis},w) = \E_w\bigl[ (c_{W_{t+1}} - h) \,z^\textnormal{dis} + \tilde{V}_{t+1}^\textnormal{rep} \bigl(\pi_{t+1}^{\textnormal{rep},*}(z^\textnormal{dis}, W_{t+1}), W_{t+1}\bigr) \bigr],
\end{equation}
with $\textstyle \tilde{V}_{T-1}^\textnormal{dis}(z^\textnormal{dis},w) = -b\, z^\textnormal{dis}$.

Our proposed algorithm will make use of the convenient formulations of $\tilde{V}_t^\textnormal{rep}(z^\textnormal{rep},w)$ and $\tilde{V}_t^\textnormal{dis}(z^\textnormal{dis},w)$ as expectations in \eqref{eq.postdec_rep_opt} and \eqref{eq.postdec_dis_opt}.
These formulations are useful for ADP for two reasons: (1) the maximization is within the expectation, so a data- or sample-driven method is easier to incorporate and (2) knowledge about the policies $\pi^{\textnormal{rep},*}_t$ and $\pi^{\textnormal{dis},*}_t$ can be used within a value function approximation procedure. Indeed, our actor-critic algorithm will make use of the interplay between the greedy policy functions \eqref{eq.policy_rep} and \eqref{eq.policy_dis} and the optimal value functions \eqref{eq.postdec_rep} and \eqref{eq.postdec_dis}.

\section{Structural Properties} \label{sec:structural}

In this section, we analyze the structure properties of the postdecision value functions $\tilde{V}_t^\textnormal{rep}$ and $\tilde{V}_t^\textnormal{dis}$ and the optimal policies $\pi_t^{\textnormal{rep},*}$ and $\pi_t^{\textnormal{dis},*}$. We remind the reader that our model uses discrete inventory states. As opposed to the standard continuous inventory state approximation, this modeling decision was made in order to accomodate the public health setting, where resources are potentially scarce. Our structural analysis makes use the properties of $L^\natural$-concave functions, an approach used often in inventory models \citep{xin2017convexity}.

\begin{definition}[$L^\natural$-concave function]
	\label{def: discrete concave function and L-natural concave function}
	A function $g: \Z^d \rightarrow \R \cup \{+\infty\}$ with $\text{dom}\,g \neq \emptyset$ is $L^\natural$-concave if and only if it satisfies discrete midpoint concavity: 
	\begin{equation} \label{eq.midpoint_concavity}
	    g(p) + g(q) \leq g \left( \left\lceil \dfrac{p+q}{2} \right\rceil \right) + g \left( \left\lfloor \dfrac{p+q}{2} \right\rfloor \right)
	\end{equation}
	for all $p, q \in \Z^d$, where $\lceil \cdot \rceil$ and $\lfloor \cdot \rfloor$ are the \emph{ceiling} and \emph{floor} functions, respectively.
\end{definition}

For the one-dimensional case, $ g: \Z \rightarrow \R $, the condition \eqref{eq.midpoint_concavity} can be reduced to the simpler statement: $g(p) - g(p-1) \ge g(p+1) - g(p)$ for all $p \in \Z$, and $L^\natural$-concavity is equivalent to discrete concavity \citep{murota1998discrete}. Throughout the rest of the paper, we will use \emph{discretely concave} to refer to one-dimensional functions that satisfy this condition.

\begin{assumption} \label{assumption: properties of u}
	For any $w$ and $\xi$, the expected utility function $u_w(x,\xi)$ is discretely concave in $x$. 
\end{assumption}

\begin{restatable}{proposition}{propositionUIsConcave} \label{prop: U is concave in r}
	Suppose Assumption~\ref{assumption: properties of u} is satisfied. Then, for each information state $w$, POD attribute $\xi$, and sub-period $i$, the lower-level value function $U_{w,i}(x,\xi)$ is discretely concave in the inventory state $x$.
\end{restatable}

\begin{restatable}{proposition}{propositionVtildeIsConcaveSeparate} \label{prop: V_tilde_separate is concave in z}
	Suppose Assumption~\ref{assumption: properties of u} is satisfied. Then, the following properties hold:
	\begin{enumerate}
		\item For each $t$ and information state $w$, the postdecision value function $\tilde{V}_t^\textnormal{rep}(z^\textnormal{rep},w)$ is discretely concave in $z^\textnormal{rep}$ and $\tilde{V}_t^\textnormal{dis}(z^\textnormal{dis},w)$ is discretely concave in $z^\textnormal{dis}$.
		\item For each $t$ and state $(r,w)$, the optimal policy $\pi_t^{\textnormal{rep},*}(r,w)$ can be written as a series of state-dependent, discrete basestock policies, with thresholds $l^\textnormal{rep}_t(w) \in \{0,1,\ldots,R_\textnormal{max}\}$:
		\begin{equation*} \label{eq.z_t-basing-on-s_t}
		    \pi_t^{\textnormal{rep},*}(r,w) = \max\{r,l^\textnormal{rep}_t(w)\}.
		\end{equation*}
		It is optimal to replenish the inventory level as close as possible to $l^\textnormal{rep}_t(w)$.
		\item For each $t$ and state $(z^\textnormal{rep},w)$, the optimal policy $\pi_t^{\textnormal{dis},*}(z^\textnormal{rep},w)$ can be written as a series of state-dependent, discrete basestock policies, with thresholds $l^\textnormal{dis}_t(z^\textnormal{rep},w) \in \{0,1,\ldots,R_\textnormal{max}\}$:
		\begin{equation*} \label{eq.z_t-basing-on-z_t}
		    \pi_t^{\textnormal{dis},*}(z^\textnormal{rep},w) = \min\{z^\textnormal{rep}, l^\textnormal{dis}_t(z^\textnormal{rep},w)\}.
		\end{equation*}
	\end{enumerate}
\end{restatable}
\begin{proof}
	See Appendix~\ref{appendix: proof property outer_separate} for the proof of Part 1. Parts 2 and 3 then follow directly from \eqref{eq.policy_rep} and \eqref{eq.policy_dis} respectively.
\end{proof}
We remark that the state-dependency of the replenish-up-to thresholds $l^\textnormal{rep}_t(w)$ in Proposition~\ref{prop: V_tilde_separate is concave in z} refers only to the exogenous information state $W_t$, while the dispense-down-to thresholds $l^\textnormal{dis}_t(z^\textnormal{rep},w)$ are dependent on both the inventory and information states $(z^\textnormal{rep},w)$. In the former case, if $r<l^\textnormal{rep}_t(w)$, it is optimal to replenish up to $l^\textnormal{rep}_t(w)$, while if $r_t \geq l^\textnormal{rep}_t(w)$, it is optimal not to replenish. The quantity ordered is given by $\pi_t^{*,\textnormal{rep}}(r,w)-r$. In the latter case,
if $z^\textnormal{rep}> l^\textnormal{dis}_t(z^\textnormal{rep},w)$, it is optimal to dispense down to $l^\textnormal{dis}_t(z^\textnormal{rep},w)$, while if $z^\textnormal{rep} \leq l^\textnormal{dis}_t(z^\textnormal{rep},w)$, it is optimal to dispense down to $z^\textnormal{rep}$.

For algorithmic reasons, we define $v^\textnormal{rep}_t(z^\textnormal{rep},w) = \Delta \tilde{V}^\textnormal{rep}_t(z^\textnormal{rep},w)$ and $v^\textnormal{dis}_t(z^\textnormal{dis},w) = \Delta \tilde{V}^\textnormal{dis}_t(z^\textnormal{dis},w)$ to be the ``slopes'' of postdecision state values $\tilde{V}^\textnormal{rep}_t(z^\textnormal{rep},w)$ and $\tilde{V}^\textnormal{dis}_t(z^\textnormal{dis},w)$ respectively, where $\Delta \tilde{V}^\textnormal{rep}_t(z^\textnormal{rep},w) = \tilde{V}^\textnormal{rep}_t(z^\textnormal{rep},w) - \tilde{V}^\textnormal{rep}_t(z^\textnormal{rep}-1,w)$, $\Delta \tilde{V}^\textnormal{dis}_t(z^\textnormal{dis},w) = \tilde{V}^\textnormal{dis}_t(z^\textnormal{dis},w) - \tilde{V}^\textnormal{dis}_t(z^\textnormal{dis}-1,w)$, and $\tilde{V}^\textnormal{rep}_t(-1,w) = \tilde{V}^\textnormal{dis}_t(-1,w) \equiv 0$. It holds that $\tilde{V}^\textnormal{rep}_t(z^\textnormal{rep},w) = \sum_{z'=0}^{z^\textnormal{rep}} v^\textnormal{rep}_t(z',w)$, where $v^\textnormal{rep}_t(0,w) \equiv \tilde{V}^\textnormal{rep}_t(0,w)$. Proposition~\ref{prop: V_tilde_separate is concave in z} implies that $v^\textnormal{rep}_t(z,w) \geq v^\textnormal{rep}_t(z',w)$ for all $0 < z \leq z'$. The same is true for $\tilde{V}^\textnormal{dis}_t(z^\textnormal{dis},w)$ and $v^\textnormal{dis}_t(z^\textnormal{dis},w)$.

\section{The Structured Actor-Critic Method} \label{sec:alg}

In this section, we focus on the upper-level inventory control and dispensing problem and introduce the structured actor-critic (S-AC) algorithm. 
The goal of the algorithm is to approximate the postdecision value functions $\tilde{V}^\textnormal{rep}$ and $\tilde{V}^\textnormal{dis}$ and the optimal (basestock) policies $\pi^{\textnormal{rep},*}$ and $\pi^{\textnormal{dis},*}$ by exploiting structure for both.
For the lower-level dispensing problem, we use backward induction to solve the dynamic programming exactly, and apply the optimal lower-level dispensing policy $\mu^*_w$ to each of the arrived PODs.

\subsection{Overview of the Main Idea}

Our algorithm is based on the recursive relationship of \eqref{eq.postdec_rep} and \eqref{eq.postdec_dis} and the properties of the problem as described in Proposition~\ref{prop: V_tilde_separate is concave in z}. The basic structure is a time-dependent version of the actor-critic method, which makes use of the interaction between the value approximations and the policy approximations in each iteration. The ``actor'' refers to the policy approxmations $\{\bar{\pi}^{\textnormal{rep},k}\}$ and $\{\bar{\pi}^{\textnormal{dis},k}\}$, and the ``critic''  refers to the value approximations $\{\bar{V}^{\textnormal{rep},k}\}$ and $\{\bar{V}^{\textnormal{dis},k}\}$. If the optimal policy is known, then the postdecision values can be calculated by \eqref{eq.postdec_rep_opt} and \eqref{eq.postdec_dis_opt}; similarly, if the value function is known, the optimal policies can be calculated by \eqref{eq.policy_rep} and \eqref{eq.policy_dis}. The proposed algorithm applies these two relationships in an alternating fashion. 

\begin{figure}[ht]
	\centering
	\includegraphics[width=0.85\textwidth]{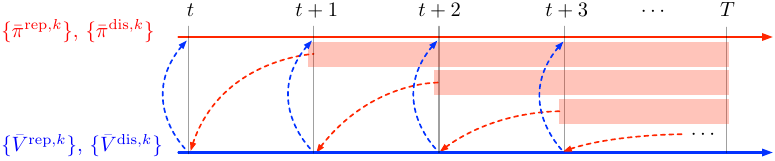}	
	\caption{An illustration of how value and policy functions interact under the S-AC algorithm.}
	\label{fig:sac}
\end{figure}

We represent the replenish-up-to policy by approximate basestock thresholds $\{\overbar{l}^{\textnormal{rep},k}\}$, where $\overbar{l}^{\textnormal{rep},k}_t(w)$ is the approximation to $l^\textnormal{rep}_t(w)$ at iteration $k$. Note that compared to a standard actor-critic implementation which tracks a stochastic policy for each state \citep{sutton1998reinforcement}, this is a significant reduction in the number of parameters needed to be learned. 
We represent the dispense-down-to policy as approximations $\{\bar{\pi}^{\textnormal{dis},k}_t(z^{\textnormal{rep}},w)\}$. 
As for the values, we represent them as approximations $\{\bar{v}^{\textnormal{rep},k}\}$ and $\{\bar{v}^{\textnormal{dis},k}\}$, where $\bar{v}^{\textnormal{rep},k}_t(z^\textnormal{rep},w)$ and $\bar{v}^{\textnormal{dis},k}_t(z^\textnormal{dis},w)$ approximate the discrete slopes $v^\textnormal{rep}_t(z^\textnormal{rep},w) = \Delta \tilde{V}^{\textnormal{rep}}_t(z^\textnormal{rep},w)$ and $v^\textnormal{dis}_t(z^\textnormal{dis},w) = \Delta \tilde{V}^{\textnormal{dis}}_t(z^\textnormal{dis},w)$, respectively. According to Proposition~\ref{prop: V_tilde_separate is concave in z}, if the approximations of the slopes are nonincreasing in $z^\textnormal{rep}$ and $z^\textnormal{dis}$, respectively, then the approximate value function is discretely concave in each of the decisions.

These approximations are iteratively updated via a stochastic approximation method \citep{robbins1951stochastic,kushner2003stochastic}. At each iteration, the algorithm has three steps. In the first step, we observe an exogenous information sequence and the attribute-request vectors for the whole planning horizon. In the second step, we observe the value of the current state under the current policy approximations, subject to the observed {attribute-demand} vectors. This value is used to update the value approximations. Finally, in the third step, we use the \emph{implied basestock threshold} from the latest value function to update our approximate policy. The interactions between the policy and value approximations are shown in Figure~\ref{fig:sac}.

\looseness-1 Throughout the rest of the paper, we use \emph{bar} notation (e.g., $\bar{v}^{\textnormal{rep},k}$ or $\overbar{l}^{\textnormal{rep},k}$) to denote approximations tracked by the algorithm at iteration $k$. On the other hand, we use \emph{hat} notation (e.g., $\hat{V}^{\textnormal{rep},k}_t$ or $\hat{v}^{\textnormal{rep},k}_t$) to denote observed values at iteration $k$ (these are one-time observations used to update the tracked approximations).

\subsection{Algorithm Description}
First, let us give some notation. The observed trajectory of the exogenous information process $\{W_t\}$ at iteration $k$ is denoted $\{w_0^k, w_1^k, \ldots, w_{T-1}^k\}$ and the initial postdecision replenished resource level at period 0 is $z_0^{\textnormal{rep,k}}$. The corresponding attribute $\{\xi_{t,1}^k\}$ observed at iteration $k$ is assumed to follow the conditional distributions given $w_t^k$. Similarly, let $\mathbf{Z}^k_t(w)$ be an independent realization of the process $(W_\tau, \xi_{\tau,1})_{\tau=t}^{T-1}$ conditioned on $W_t=w$. This sequence of realizations is used to obtain an observation of the value of policy approximation starting at $t$ and $W_t=w$ and we denote its elements by
\begin{equation*}
    \mathbf{Z}^k_t(w) = \bigl\{ (\check{w}_\tau^k, \check{\xi}_{\tau,1}^k) : \tau = t,\ldots,T-1 \bigr\},
\end{equation*}
where $\check{w}_t^k = w$. Define $\tilde{\pi}^{\textnormal{rep},k}$ and $\tilde{\pi}^{\textnormal{dis},k}$ as the rounded policies, i.e. $\tilde{\pi}^{\textnormal{rep},k}(r,w) = \texttt{round} [\bar{\pi}^{\textnormal{rep},k}(r,w)]$ for all $(r,w)$, 
$\tilde{\pi}^{\textnormal{dis},k}(z^\textnormal{rep},w) = \texttt{round} [\bar{\pi}^{\textnormal{dis},k}(z^\textnormal{rep},w)]$ for all $(z^\textnormal{rep},w)$, 
where $\texttt{round}[x]$ returns the nearest integer to $x\in\mathbb R$. This is necessary because our approximate thresholds will not be integers. Let $f_t^{\textnormal{rep}} (\tilde{\pi}^{\textnormal{rep},k-1}, \tilde{\pi}^{\textnormal{dis},k-1}; \mathbf{Z}^k_t(w_t), r_t)$ be the Monte Carlo estimates of the replenish-up-to postdecision value starting in period $t$ under the current policy approximations and an initial state $(r_t,w_t)$:
\begin{equation} \label{eq.ft_rep}
    \begin{aligned}
	\textstyle f_t^{\textnormal{rep}} \bigl(\tilde{\pi}^{\textnormal{rep},k-1}, \tilde{\pi}^{\textnormal{dis},k-1}; \mathbf{Z}^k_t(w_t), r_t\bigr)
	=&\textstyle \sum_{\tau=t}^{T-2} \bigl[ -c_{\check{w}_\tau^k} \tilde{z}_\tau^\textnormal{rep} + U_{\check{w}_{\tau}^k,0}^{\mu^*} \bigl( \tilde{z}_\tau^\textnormal{rep} - \tilde{z}_\tau^\textnormal{dis}, \check{\xi}_{\tau,0}^k \bigr) + (c_{\check{w}_{\tau+1}^k} - h) \tilde{z}_\tau^\textnormal{dis} \bigr] \\
	&-c_{\check{w}_{T-1}^k} \tilde{z}_{T-1}^\textnormal{rep} + U_{\check{w}_{T-1}^k}^{\mu^*} \bigl( \tilde{z}_{T-1}^\textnormal{rep} - \tilde{z}_{T-1}^\textnormal{dis}, \check{\xi}_{T-1,0}^k\bigr) - b\,\tilde{z}_{T-1}^\textnormal{dis},
	\end{aligned}
\end{equation}
where for all $\tau\geq t$, $\mu^* = \mu^*_{\check{w}_{\tau}^k}$, $\tilde{z}_\tau^\textnormal{rep} = \tilde{\pi}_\tau^{\textnormal{rep},k-1}(r_\tau,\check{w}_\tau^k)$, 
$\tilde{z}_\tau^\textnormal{dis} = \tilde{\pi}_\tau^{\textnormal{dis},k-1}(\tilde{z}_\tau^\textnormal{rep},\check{w}_\tau^k)$.
Let $f_t^{\textnormal{dis}} (\tilde{\pi}^{\textnormal{rep},k-1}, \tilde{\pi}^{\textnormal{dis},k-1}; \mathbf{Z}^k_t(w_t), z_t^{\textnormal{rep}})$ be the Monte Carlo estimates of the dispense-down-to postdecision value starting in period $t$ under the current policy approximations and an initial state $(z_t^{\textnormal{rep}},w_t)$:
\begin{equation} \label{eq.ft_dis}
    \begin{aligned}
	\textstyle f_t^{\textnormal{dis}} \bigl(\tilde{\pi}^{\textnormal{rep},k-1}, & \tilde{\pi}^{\textnormal{dis},k-1}; \mathbf{Z}^k_t(w_t), z_t^{\textnormal{rep}}\bigr)\\
	&=\textstyle \sum_{\tau=t}^{T-2} \bigl[ (c_{\check{w}_{\tau+1}^k} - h) \tilde{z}_\tau^\textnormal{dis} - c_{\check{w}_{\tau+1}^k} \tilde{z}_{\tau+1}^\textnormal{rep} + U_{\check{w}_{\tau+1}^k,0}^{\mu^*} \bigl( \tilde{z}_{\tau+1}^\textnormal{rep} - \tilde{z}_{\tau+1}^\textnormal{dis}, \check{\xi}_{\tau+1,0}^k \bigr) \bigr] - b\,\tilde{z}_{T-1}^\textnormal{dis},
	\end{aligned}
\end{equation}
where $\tilde{z}_t^\textnormal{dis} = \tilde{\pi}_t^{\textnormal{dis},k-1}(z_t^\textnormal{rep},\check{w}_t^k)$, and for all $\tau\geq t+1$, $\mu^* = \mu^*_{\check{w}_{\tau}^k}$, $\tilde{z}_\tau^\textnormal{rep} = \tilde{\pi}_\tau^{\textnormal{rep},k-1}(r_\tau,\check{w}_\tau^k)$, 
$\tilde{z}_\tau^\textnormal{dis} = \tilde{\pi}_\tau^{\textnormal{dis},k-1}(\tilde{z}_\tau^\textnormal{rep},\check{w}_\tau^k)$.
The replenish-up-to policy is 
\begin{equation*}
    \bar{\pi}_{\tau}^{\textnormal{rep},k}(r_\tau,\check{w}_\tau^k) = \max\{r_\tau, \bar{l}_\tau^{\textnormal{rep},k}(\check{w}_\tau^k)\}.
\end{equation*}
Although there is substantial notation used in defining $f_t^{\textnormal{rep}}$ and $f_t^{\textnormal{dis}}$, we remark that they are simply Monte Carlo observations of the policy's postdecision values respectively corresponding to the replenish-up-to and dispense-down-to decisions.

At each period $t$, to compute the approximate slopes, we use $f_{t}^{\textnormal{dis}}$ to observe values $\hat{V}_t^{\textnormal{rep},k} (z_t^{\textnormal{rep},k}, w_t^k)$ and $\hat{V}_t^{\textnormal{rep},k} (z_t^{\textnormal{rep},k}-1, w_t^k)$, and $f_{t+1}^{\textnormal{rep}}$ to observe values $\hat{V}_t^{\textnormal{dis},k} (z_t^{\textnormal{dis},k}, w_t^k)$ and $\hat{V}_t^{\textnormal{dis},k} (z_t^{\textnormal{dis},k}-1, w_t^k)$, where $f_{t}^{\textnormal{dis}}$ and $f_{t+1}^{\textnormal{rep}}$ are implied by the current policies $\bar{\pi}^{\textnormal{rep},k-1}$ and $\bar{\pi}^{\textnormal{dis},k-1}$; specifically, for $z^\textnormal{rep},z^\textnormal{dis} \ge 0$, the observations $\hat{V}_t^{\textnormal{rep},k}(z^\textnormal{rep},w_t^k)$ and $\hat{V}_t^{\textnormal{dis},k}(z^\textnormal{dis},w_t^k)$ are 
\begin{equation} \label{eq.V-and-f-rep}
	\begin{aligned}
	\hat{V}_t^{\textnormal{rep},k}(z^\textnormal{rep},w_t^k) =& -c_{w_t^k} z^\textnormal{rep} + U_{\check{w}_{t}^k,0}^{\mu^*} \bigl( z^\textnormal{rep} - \tilde{\pi}_t^{\textnormal{dis},k-1}(z^\textnormal{rep},w_t^k), \check{\xi}_{t,0}^k \bigr) \\&+ f_{t}^{\textnormal{dis}} \bigl(\tilde{\pi}^{\textnormal{rep},k-1}, \tilde{\pi}^{\textnormal{dis},k-1}; \mathbf{Z}^k_{t}(w_{t}), z^\textnormal{rep}\bigr),
	\end{aligned}
\end{equation}
and
\begin{equation} \label{eq.V-and-f-dis}
	\begin{aligned}
	\hat{V}_t^{\textnormal{dis},k}(z^\textnormal{dis},w_t^k) =& (c_{w_{t+1}} - h)z^\textnormal{dis} + f_{t+1}^{\textnormal{rep}} \bigl(\tilde{\pi}^{\textnormal{rep},k-1}, \tilde{\pi}^{\textnormal{dis},k-1}; \mathbf{Z}^k_{t+1}(w_{t+1}), z^\textnormal{dis}\bigr),
	\end{aligned}
\end{equation}
where $w_{t+1}$ is sampled from the distribution $W_{t+1}\, |\, W_t=w_t^k$. 
The approximate slopes $\hat{v}_t^{\textnormal{rep},k}$ and $\hat{v}_t^{\textnormal{dis},k}$ are given by: 
\begin{equation} \label{eq.v_rep_hat}
    \hat{v}_t^{\textnormal{rep},k} = \hat{V}_t^{\textnormal{rep},k}(z_t^{\textnormal{rep},k},w_t^k) - \hat{V}_t^{\textnormal{rep},k}(z_t^{\textnormal{rep},k}-1,w_t^k),
\end{equation}
\begin{equation} \label{eq.v_dis_hat}
    \hat{v}_t^{\textnormal{dis},k} = \hat{V}_t^{\textnormal{dis},k}(z_t^{\textnormal{dis},k},w_t^k) - \hat{V}_t^{\textnormal{dis},k}(z_t^{\textnormal{dis},k}-1,w_t^k),
\end{equation}
where we define $\hat{V}_t^{\textnormal{rep},k}(-1,w_t^k) = \hat{V}_t^{\textnormal{dis},k}(-1,w_t^k) \equiv 0$. By doing so, the value assigned to $\hat{v}_t^{\textnormal{rep},k}$ when $z_t^{\textnormal{rep},k}=0$ is actually $\hat{V}_t^{\textnormal{rep},k}(0,w_t^k)$. This also applies to $\hat{v}_t^{\textnormal{dis},k}$. We now summarize the structured actor-critic method; the full details of the approach are given in Algorithm~\ref{alg:sac}. 

\begin{figure}[ht]
	\centering
	\includegraphics[width=0.99\textwidth]{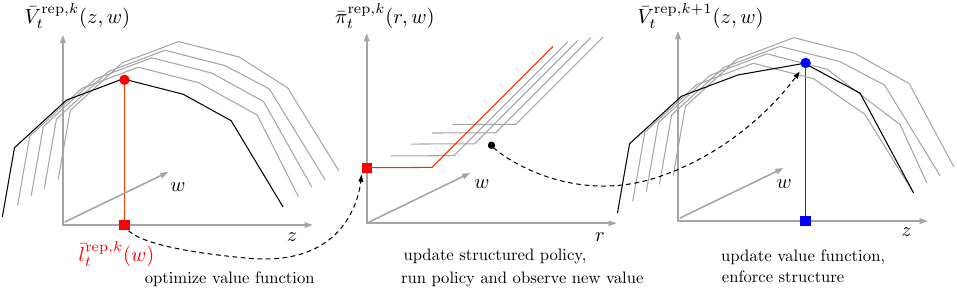}
	\caption{An illustration of the sequence of updates used in the S-AC algorithm.}
	\label{fig:value_policy}
\end{figure}

\IncMargin{1em}
\begin{algorithm}
	\SetKwInput{Input}{Input}\SetKwInput{Output}{Output}
	\Input{
	    Lower-level optimal policy $\mu^*$ (learned from backward dynamic programming). Initial policy estimates $\overbar{l}^{\textnormal{rep},0}$ and $\bar{\pi}^{\textnormal{dis},0}$, and value estimates $\bar{v}^{\textnormal{rep},0}$ and $\bar{v}^{\textnormal{dis},0}$ (nonincreasing in $z^\textnormal{rep}$ and $z^\textnormal{dis}$ respectively).
		Stepsize rules $\tilde{\alpha}_t^k$ and $\tilde{\beta}_t^k$ for all $t,k$.}
		\medskip
	\Output{Approximations $\overbar{l}^{\textnormal{rep},k}$, $\bar{\pi}^{\textnormal{dis},k}$, $\bar{v}^{\textnormal{rep},k}$, and $\bar{v}^{\textnormal{dis},k}$.}
	\BlankLine
	\For{$k=1,2,\ldots$}{
		\vspace{0.5em}
		Sample initial states $z_0^{\textnormal{rep},k}$ and $z_0^{\textnormal{dis},k}$. \label{alg:sac0}\\
		\vspace{0.5em}
		\For{$t=0,1,\ldots,T-1$}{\label{forins}
			\vspace{0.5em}
			Observe $w_t^k$ and $\xi_{t,1}^k$, then observe $\hat{v}_t^{\textnormal{rep},k}$ and $\hat{v}_t^{\textnormal{dis},k}$ according to \eqref{eq.v_rep_hat} and \eqref{eq.v_dis_hat} respectively. \label{alg:sac1}\\
			\vspace{0.5em}
			Perform SA step: \label{alg:sac2}
			\begin{equation*}
			    \begin{aligned}
			    \textstyle \tilde{v}_t^{\textnormal{rep},k}(z^\textnormal{rep},w) &= \bigl(1-\alpha_t^k(z^\textnormal{rep},w)\bigr)\,\bar{v}_t^{\textnormal{rep},k-1}(z^\textnormal{rep},w) + \alpha_t^k(z^\textnormal{rep},w)\, \hat{v}_t^{\textnormal{rep},k}, \\
			    \textstyle \tilde{v}_t^{\textnormal{dis},k}(z^\textnormal{dis},w) &= \bigl(1-\alpha_t^k(z^\textnormal{dis},w)\bigr)\,\bar{v}_t^{\textnormal{dis},k-1}(z^\textnormal{dis},w) + \alpha_t^k(z^\textnormal{dis},w)\, \hat{v}_t^{\textnormal{dis},k}.
			    \end{aligned}
			\end{equation*}
			\\
			\vspace{0.5em}
			Perform the concavity projection operation \eqref{eq.concavity-projection}: \label{alg:sac3}
			\begin{equation*}
			    \textstyle \bar{v}_t^{\textnormal{rep},k} = \Pi_{z_t^{\textnormal{rep},k},w_t^k} (\tilde{v}_t^{\textnormal{rep},k}), \;\;\;\; \bar{v}_t^{\textnormal{dis},k} = \Pi_{z_t^{\textnormal{dis},k},w_t^k} (\tilde{v}_t^{\textnormal{dis},k}).
			\end{equation*}
			\\
			\vspace{0.5em}
			Observe and update the replenish-up-to threshold: \label{alg:sac4}
			\begin{equation*}
			    \hat{l}_t^{\textnormal{rep},k} =\textstyle \argmax_{z^\textnormal{rep} \in \bar{\mZ}(0)} \sum_{j=0}^{z^\textnormal{rep}} \bar{v}_t^{\textnormal{rep},k}\bigl(j,w_t^k\bigr),
			\end{equation*}
			\begin{equation*}
			    \overbar{l}_t^{\textnormal{rep},k}(w) = \bigl(1-\beta_t^k(w)\bigr)\,\overbar{l}_t^{\textnormal{rep},k-1}(w) + \beta_t^k(w)\,\hat{l}_t^{\textnormal{rep},k}.
			\end{equation*}
			\\
			\vspace{0.5em}
			Observe and update the dispense-down-to policy: \label{alg:sac5}
			\\
			\vspace{0.5em}
			\For{$z_t^{\textnormal{rep}} = 0,1,\ldots, R_{\textnormal{max}}$ \label{alg:sac6}}{
			\begin{equation*}
			    \hat{\pi}_t^{\textnormal{dis}} =\textstyle \argmax_{z^\textnormal{dis} \in \underline{\mZ}(z_t^{\textnormal{rep}})} U_{w_t^k,0}^{\mu^*}\bigl(z_t^{\textnormal{rep}} - z^\textnormal{dis}, \xi_{t,0}^k\bigr) + \sum_{j=0}^{z^\textnormal{dis}} \bar{v}_t^{\textnormal{dis},k}\bigl(j,w_t^k\bigr),
			\end{equation*}
			\begin{equation*}
			    \bar{\pi}_t^{\textnormal{dis},k} (z^{\textnormal{rep}},w) = \bigl(1-\alpha^k(z^{\textnormal{rep}},w)\bigr)\,\bar{\pi}_t^{\textnormal{dis},k-1}(z^{\textnormal{rep}},w) + \alpha^k(z^{\textnormal{rep}},w)\,\hat{\pi}_t^{\textnormal{dis}}.
			\end{equation*}
			}\label{alg:sac7}
			\vspace{0.5em}
			If $t<T-1$, take $z_{t+1}^{\textnormal{rep},k}$ and $z_{t+1}^{\textnormal{dis},k}$ according to the $\epsilon$-greedy exploration policy.
		}
	}
	\caption{Structured Actor-Critic Method}
	\label{alg:sac}
\end{algorithm}\DecMargin{1em}

\begin{itemize}
	\item The inputs of Algorithm~\ref{alg:sac} are a random initial basestock policy $\overbar{l}^{\textnormal{rep},0}$, and concave, piecewise linear value function approximations $\bar{v}^{\textnormal{rep},0}$ and $\bar{v}^{\textnormal{dis},0}$.
	
	\item Each iteration $k$ consists of a loop through the time periods $t$.
	
	\item At period $t$, the approximate slopes are updated in Lines~\ref{alg:sac1}--\ref{alg:sac3}. Based on $z_t^{\textnormal{rep},k}$, $z_t^{\textnormal{dis},k}$ and $\mathbf{Z}_t^k(w_t^k)$, we first observe the sequences of the predecision resource $\{r_{t+1}, r_{t+2}, \ldots, r_T\}$ and the postdecision resources $\{z_t^{\textnormal{rep},k}, z_{t+1}^{\textnormal{rep}}, \ldots, z_{T-1}^{\textnormal{rep}}\}$ and $\{z_t^{\textnormal{dis},k}, z_{t+1}^{\textnormal{dis}}, \ldots, z_{T-1}^{\textnormal{dis}}\}$. These are computed according to \eqref{eq.trans}, and the equations $z_\tau^{\textnormal{rep}} = \tilde{\pi}_\tau^{\textnormal{rep},k-1}(r_\tau, w_\tau^k)$, and $z_\tau^\textnormal{dis} = \argmax_{z^\textnormal{dis} \in \underline{\mZ}(\tilde{z}_\tau^\textnormal{rep})} U_{\check{w}_{\tau}^k,0}^{\mu^*} \bigl( \tilde{z}_\tau^\textnormal{rep} - z^\textnormal{dis}, \check{\xi}_{\tau,0}^k \bigr) + \bar{V}_{\tau}^{\textnormal{dis},k-1} \bigl(z^\textnormal{dis}, \check{w}_\tau^k\bigr)$ for all $\tau \ge t+1$. In the following illustration, let us take the value slope and policy corresponding to the replenish-up-to decision as an example, those corresponding to the dispense-down-to decision are similar.
	
	\item The observation of the slope $\hat{v}_t^{\textnormal{rep},k}$ \emph{implied by the policies} $\tilde{\pi}^{\textnormal{rep},k-1}$ and $\tilde{\pi}^{\textnormal{dis},k-1}$ is computed using \eqref{eq.V-and-f-rep} and \eqref{eq.v_rep_hat} and used to calculate the smoothed slopes $\tilde{v}_t^{\textnormal{rep},k}(z^{\textnormal{rep}},w)$ in Line~\ref{alg:sac2}, where $\alpha_t^k(z^\textnormal{rep},w) = \tilde{\alpha}_t^k \1\{z^{\textnormal{rep}} = z_t^{\textnormal{rep},k}\} \1\{w=w_t^k\}$. Thus, only the state $(z_t^{\textnormal{rep},k},w_t^k)$ is updated. 
	
	\item A concavity projection operation in Line~\ref{alg:sac3} is performed on the slopes $\tilde{v}_t^{\textnormal{rep},k}$, resulting in a new set of slopes $\Pi_{z_t^{\textnormal{rep},k}, w_t^k} (\tilde{v}_t^{\textnormal{rep},k})$, in order to avoid violation of concavity. The component of $\Pi_{z_t^{\textnormal{rep},k}, w_t^k} (\tilde{v}_t^{\textnormal{rep},k})$ at state $(z^{\textnormal{rep}},w)$ is
	\begin{equation} \label{eq.concavity-projection}
	    \begin{aligned}
    	&\Pi_{z_t^{\textnormal{rep},k}, w_t^k} (\tilde{v}_t^{\textnormal{rep},k})[z^{\textnormal{rep}},w] \\
    	&= \begin{cases}
    	\tilde{v}_t^{\textnormal{rep},k} ( z_t^{\textnormal{rep},k}, w_t^k ) & \text{if}\;\; w = w_t^k,\; z^{\textnormal{rep}} < z_t^{\textnormal{rep},k},\; \tilde{v}_t^{\textnormal{rep},k}(z^{\textnormal{rep}},w) < \tilde{v}_t^{\textnormal{rep},k} (z_t^{\textnormal{rep},k}, w_t^k) \\
    	& \text{or}\; w = w_t^k,\; z^{\textnormal{rep}} > z_t^{\textnormal{rep},k},\; \tilde{v}_t^{\textnormal{rep},k}(z^{\textnormal{rep}},w) > \tilde{v}_t^{\textnormal{rep},k} (z_t^{\textnormal{rep},k}, w_t^k), \\
    	\tilde{v}_t^{\textnormal{rep},k}(z^{\textnormal{rep}},w) & \text{otherwise}.
    	\end{cases}
    	\end{aligned}
	\end{equation}
	
	\item The approximate replenish-up-to policy is updated in Lines~\ref{alg:sac4}. The observation $\hat{l}_t^{\textnormal{rep},k}$ is the maximum point of $\bar{V}_t^{\textnormal{rep},k}(\cdot,w_t^k)$ inside the set $\mZ(0)$, which is the implied replenish-up-to basestock threshold from the value function approximation. Given the observation, the policy is updated with stepsize $\beta_t^k(w) = \tilde{\beta}_t^k \1\{w=w_t^k\}$. 
	
	\item The approximate dispense-down-to policy is updated in Lines~\ref{alg:sac6}--\ref{alg:sac7}. For each $z^\textnormal{rep}$, we can observe $\hat{\pi}_t^{\textnormal{dis}}$ according to \eqref{eq.policy_dis}. The policy is updated with the observation and stepsize $\alpha_t^k(z^\textnormal{rep},w) = \tilde{\alpha}_t^k \1\{z^{\textnormal{rep}}=z_t^{\textnormal{rep}}\} \1\{w=w_t^k\}$. 
	
	\item Finally, the next replenish-up-to decision follows an $\epsilon$-greedy policy, which is to select $z_{t+1}^{\textnormal{rep},k} = \tilde{\pi}_\tau^{\textnormal{rep},k-1}(r_\tau, w_\tau^k)$ with probability $1-\epsilon$, or take $z_{t+1}^{\textnormal{rep},k}$ randomly from $\mZ(r_{t+1}^k)$ with probability $\epsilon$. In our numerical experiments, $\epsilon$ is chosen to be 0.1.
\end{itemize}

\looseness-1 Figure~\ref{fig:value_policy} illustrates how the replenish-up-to value function and policy approximations interact with each other. The first two panels together show that given a structured value function, its maximizer (red square) is used to update the structured policy. Panels two and three together show that an observation of the current policy's value (blue circle) is in turn used to update the structured value function (where a projection step occurs to enforce structure). The process then repeats with the new maximizer (blue square).

\subsection{Convergence Analysis} \label{sec:convanalysis}

In this section, we give some theoretical assumptions and then state the convergence of Algorithm~\ref{alg:sac}; in particular, the convergence of both the value function approximations $\bar{v}^{\textnormal{rep},k}$ and $\bar{v}^{\textnormal{dis},k}$ and the basestock policies $\overbar{l}^{\textnormal{rep},k}$ and $\bar{\pi}^{\textnormal{dis},k}$. Let $\{\bar{v}_t^{\textnormal{rep},k}\}_{k \geq 0}$ and $\{\bar{v}_t^{\textnormal{dis},k}\}_{k \geq 0}$ be the sequences of slopes, let $\{\overbar{l}_t^{\textnormal{rep},k}\}_{k \geq 0}$ and $\{\bar{\pi}_t^{\textnormal{dis},k}\}_{k \geq 0}$, be the sequences of policies generated by the algorithm. For period $T$, we assume $v_T^{\textnormal{rep}}(z^{\textnormal{rep}},w) = 0$ for all iterations $k\geq 0$ and all possible postdecision states $(z^{\textnormal{rep}},w)$, as we only need to learn the policy and slopes up to period $T-1$. 
We work on a probability space $(\Omega, \mathcal{F}, \mathbf{P})$, where $\mathcal{F} = \sigma\{(r_t^k, z_t^{\textnormal{rep},k}, z_t^{\textnormal{dis},k}, w_t^k, \boldsymbol{\xi}_t^k, \mathbf{D}_t^k, \hat{v}_t^k), \,t\leq T, \, k \ge 0\}$, where $\boldsymbol{\xi}_t^k = (\xi_{t,1}^k, \xi_{t,2}^k, \ldots, \xi_{t,n_k}^k)$, $\mathbf{D}_t^k = (D_{t,1}^k, D_{t,2}^k, \ldots, D_{t,n_k}^k)$. Moreover, we define 
\begin{equation*}
	\mathcal{F}_t^k = \sigma \bigl\{ \{(r_\tau^{k'}, z_\tau^{\textnormal{rep},k'}, z_\tau^{\textnormal{dis},k'}, w_\tau^{k'}, \boldsymbol{\xi}_\tau^{k'}, \mathbf{D}_\tau^{k'}, \hat{v}_\tau^{k'}), {k'}<k, \tau\leq T\} \cup \{(r_\tau^k, z_\tau^{\textnormal{rep},k}, z_\tau^{\textnormal{dis},k}, w_\tau^k, \boldsymbol{\xi}_\tau^k, \mathbf{D}_\tau^k, \hat{v}_t^k), \tau\leq t\} \bigr\}, 
\end{equation*}
for $t\leq T-1$ and $k\geq 1$, with $\mathcal{F}_t^0 = \{\emptyset,\Omega\}$ for all $t \leq T$. Their relationships are $ \mathcal{F}_t^k \subseteq \mathcal{F}_{t+1}^k $ for $t \leq T-1$ and $ \mathcal{F}_T^k \subseteq \mathcal{F}_0^{k+1} $. 

\begin{assumption} \label{assumption: stepsize}
	For any $z$ and $w$, suppose the stepsize sequences $\bigl \{ \alpha_t^k(z^{\textnormal{rep}},w) \bigr \}$, $\bigl\{ \alpha_t^k(z^{\textnormal{dis}},w) \bigr \}$, and $\{\beta_t^k(w)\}$ satisfy the following conditions: 
	\begin{enumerate}[label=(\roman*)] 
		\item For $\textnormal{x}\in\{\textnormal{rep}, \textnormal{dis}\}$,  $\alpha_t^k(z^{\textnormal{x}},w) = \tilde{\alpha}_t^k \1\{z^{\textnormal{x}}=z_t^{\textnormal{x},k}\} \1\{w=w_t^k\} $ for some $ \tilde{\alpha}_t^k \in \R $ that is $\mathcal{F}_t^k$-measurable,
		\item $\beta_t^k(w) = \tilde{\beta}_t^k \1\{w=w_t^k\}$ for some $\tilde{\beta}_t^k \in \R$ that is $\mathcal{F}_t^k$-measurable,
		\item For $\textnormal{x}\in\{\textnormal{rep}, \textnormal{dis}\}$, $\sum_{k = 0}^{\infty} \alpha_t^k(z^{\textnormal{x}},w) = \infty$, $\sum_{k = 0}^{\infty} \, \bigl(\alpha_t^k(z^{\textnormal{x}},w)\bigr)^2 < \infty$ almost surely, 
		\item $\sum_{k = 0}^{\infty} \beta_t^k(w) = \infty$, $\sum_{k = 0}^{\infty} \, \bigl(\beta_t^k(w)\bigr)^2 < \infty$ almost surely. 
	\end{enumerate}
\end{assumption}

Assumption~\ref{assumption: stepsize}(\romannum{1}) and (\romannum{2}) ensures that only the slope and threshold for the observed state is updated in Line~\ref{alg:sac2} of Algorithm~\ref{alg:sac}; the ones corresponding to unobserved states are kept the same until the projection step. Parts (\romannum{3}) and (\romannum{4}) are standard conditions on the stepsize. To keep the convergence results clean, we also assume the state-dependent basestock thresholds are unique (this assumption can be easily relaxed).

\begin{assumption} \label{assumption:unique basestock}
There is a unique optimal solution to $\max_{z \in\mZ(0)} \; \tilde{V}_t^{\textnormal{rep}}(z,w)$, which implies that there is a single optimal replenishment basestock threshold for each $w$. The unique optimal solution assumption also applies to $\tilde{V}_t^{\textnormal{dis}}$.
\end{assumption}

Assumptions (\ref{assumption: properties of u})-(\ref{assumption:unique basestock}) are used for the next two results. The primary novel aspect of our analysis is to connect the approximate policies with the approximate value functions through the structural properties of the problem. Before stating the main convergence result, Theorem~\ref{thm: v and z converge}, we introduce a lemma that illustrates the crucial mechanism for convergence. 

\begin{restatable}{lemma}{lemmavtConverge} 
	\label{lemma: convergence of vt}
	The following hold:
	\begin{enumerate}
	    \item For any fixed period $t$, suppose that the policies $\bar{\pi}_{\tau}^{\textnormal{rep},k} \rightarrow \pi_\tau^{\textnormal{rep}}$ almost surely for $\tau \geq t+1$, and $\bar{\pi}_{\tau}^{\textnormal{dis},k} \rightarrow \pi_\tau^{\textnormal{dis}}$ almost surely for $\tau \geq t$. Then it holds that $\bar{v}_t^{\textnormal{rep},k}(z^{\textnormal{rep}},w) \rightarrow v_t^{\textnormal{rep}}(z^{\textnormal{rep}},w)$ almost surely.
	    \item For any fixed period $t$, suppose that the policies $\bar{\pi}_{\tau}^{\textnormal{rep},k} \rightarrow \pi_\tau^{\textnormal{rep}}$ and $\bar{\pi}_{\tau}^{\textnormal{dis},k} \rightarrow \pi_\tau^{\textnormal{dis}}$ almost surely for $\tau \geq t+1$. Then it holds that $\bar{v}_t^{\textnormal{dis},k}(z^{\textnormal{dis}},w) \rightarrow v_t^{\textnormal{dis}}(z^{\textnormal{dis}},w)$ almost surely.
	\end{enumerate}

\end{restatable}

\begin{proof}[Sketch of Proof]
    Let us show part (1) of the lemma. The proof for part (2) is similar. We first construct two deterministic sequences $\{G^m\}$ and $\{I^m\}$ such that $G^0 = v^{\textnormal{rep}} + v_\text{max}^{\textnormal{rep}}$ and $I^0 = v^{\textnormal{rep}} - v_\text{max}^{\textnormal{rep}}$ with
	\begin{equation*}
		G^{m+1} = \frac{G^m + v^{\textnormal{rep}}}{2}\quad \text{and} \quad I^{m+1} = \frac{I^m + v^{\textnormal{rep}}}{2},
	\end{equation*}
	where $|v_t^{\textnormal{rep}}(z^{\textnormal{rep}},w)| \le v_\text{max}^{\textnormal{rep}}$ for all $t$, $z^{\textnormal{rep}}$, and $w$. These sequences have been previously used in \cite{bertsekas1996neuro-dynamic}. Lemma~\ref{lemma: convergence of vt} is proved if we have
	\begin{equation} \label{eq.bounds}
	    I_t^m(z^{\textnormal{rep}},w) \leq \bar{v}_t^{\textnormal{rep},k-1}(z^{\textnormal{rep}},w) \leq G_t^m(z^{\textnormal{rep}},w),
	\end{equation}
	for any $m$ and sufficiently large $k$. The proof proceeds by showing the following. 
	\begin{enumerate}
	    \item Define noise terms $\epsilon_t^{k}(z_t^{\textnormal{rep},k},w_t^k) = \E \bigl[\hat{v}_t^{\textnormal{rep},k}\bigr] - v_t^{\textnormal{rep}}(z_t^{\textnormal{rep},k}, w_t^k)$ and $\varepsilon_t^{k}(z_t^{\textnormal{rep},k}, w_t^k) = \hat{v}_t^{\textnormal{rep},k} - \E\bigl[\hat{v}_t^{\textnormal{rep},k}\bigr]$. Recall that $\hat{v}_t^{\textnormal{rep},k} = \hat{V}_t^{\textnormal{rep},k}(z_t^{\textnormal{rep},k},w_t^k) - \hat{V}_t^{\textnormal{rep},k}(z_t^{\textnormal{rep},k}-1,w_t^k)$, where
	    \begin{equation*}
        	\begin{aligned}
        	\hat{V}_t^{\textnormal{rep},k}(z^\textnormal{rep},w_t^k) =& - c_{w_t^k} z^\textnormal{rep} + U_{\check{w}_{t}^k,0}^{\mu^*} \bigl(z^\textnormal{rep} - \tilde{\pi}_t^{\textnormal{dis},k-1}(z^\textnormal{rep},w_t^k), \check{\xi}_{t,0}^k \bigr) \\&+ f_{t}^{\textnormal{dis}} \bigl(\tilde{\pi}^{\textnormal{rep},k-1}, \tilde{\pi}^{\textnormal{dis},k-1}; \mathbf{Z}^k_{t}(w_{t}), z^\textnormal{rep}\bigr).
        	\end{aligned}
        \end{equation*}
	    From the assumption that $\bar{\pi}_{\tau}^{\textnormal{rep},k} \rightarrow \pi_\tau^{\textnormal{rep}}$ and $\bar{\pi}_{\tau}^{\textnormal{dis},k} \rightarrow \pi_\tau^{\textnormal{dis}}$ almost surely for all $\tau \geq t+1$, and the fact that $f_{t}^{\textnormal{dis}} \bigl(\tilde{\pi}^{\textnormal{rep},k-1}, \tilde{\pi}^{\textnormal{dis},k-1}; \mathbf{Z}^k_{t}(w), z^\textnormal{rep}\bigr)$ depends on the replenish-up-to policies for periods $t+1$ onward and the dispense-down-to policies for periods $t$ onward, we conclude that 
	    \[\mathbf{E}_w \bigl[f_{t}^{\textnormal{dis}} \bigl(\tilde{\pi}^{\textnormal{rep},k-1}, \tilde{\pi}^{\textnormal{dis},k-1}; \mathbf{Z}_{t}^k(w), z^\textnormal{rep}\bigr) \bigr] \rightarrow \tilde{V}_{t}^{\textnormal{dis}}\bigl(\pi_{t}^{\textnormal{dis},*}(z^\textnormal{rep},w),w\bigr)
	    \]
	    almost surely.
	    Therefore, $\epsilon_t^{k}(z_t^{\textnormal{rep},k},w_t^k)$ converges to zero almost surely and $\varepsilon_t^{k}(z_t^{\textnormal{rep},k}, w_t^k)$ is unbiased.
	    
	    \item We partition the state space $\mS$ into two parts: (1) states $ (z^{\textnormal{rep}},w) \in \mS^-_t $ and (2) states $ (z^{\textnormal{rep}},w)\in \mS \setminus \mS_t^- $, where $\mS_t^-$ is a random set of states that are increased by the projection operator \eqref{eq.concavity-projection} on finitely many iterations $k$. The proof considers each partition separately to show \eqref{eq.bounds}. For states $(z^{\textnormal{rep}},w) \in \mS^-_t$, we show by forward induction on $m$ the existence of a finite index $\tilde{K}_t^m$ such that \eqref{eq.bounds} holds for all iterations $k \geq \tilde{K}_t^m$. The proof utilizes stochastic sequences related to the noise terms and stochastic ``bounding'' sequences. For any state $(z^{\textnormal{rep}},w) \in \mS \setminus \mS^-_t$ and a fixed $m$, by Lemma 6.4 of \cite{nascimento2009optimal}, we show the existence of a state-dependent random index $\hat{K}_t^{m}(z^{\textnormal{rep}},w)$ such that \eqref{eq.bounds} holds for all $k \ge \hat{K}_t^{m}(z^{\textnormal{rep}},w)$.
	\end{enumerate}
	See Appendix~\ref{appendix. proof of convergence of vt} for the full details of the proof.
\end{proof}

Lemma~\ref{lemma: convergence of vt} implies the convergence of the approximate slopes $\bar{v}^k$ to the true slopes $v$ as long as the policy approximation converges correctly.

\begin{restatable}{theorem}{thmvandzconverge}
	\label{thm: v and z converge}
	For $\textnormal{x}\in \{\textnormal{rep}, \textnormal{dis}\}$, the slope approximation $\bar{v}_t^{\textnormal{x},k}(z^{\textnormal{x}},w)$ converges to the slope of the postdecision value function $v_t^{\textnormal{x}}(z^{\textnormal{x}},w)$ almost surely for all $(z^{\textnormal{x}},w)$ and $t$; the policy approximations $\bar{\pi}_t^{\textnormal{rep},k}(r,w)$ and $\bar{\pi}_t^{\textnormal{dis},k}(z^{\textnormal{rep}},w)$ respectively converge to the optimal policies $\pi_t^{\textnormal{rep}}(r,w)$ and $\pi_t^{\textnormal{dis}}(z^{\textnormal{rep}},w)$ almost surely for all $r$, $z^{\textnormal{rep}}$, $w$ and $t$.
\end{restatable}
\begin{proof}[Sketch of Proof]
	The proof depends inductively on Lemma~\ref{lemma: convergence of vt}. Given its result for period $t$, we can then argue the convergence of policy approximations $\bar{\pi}_t^{\textnormal{rep},k}(r,w)$ and $\bar{\pi}_t^{\textnormal{dis},k}(z^{\textnormal{rep}},w)$. This allows us to re-apply Lemma~\ref{lemma: convergence of vt} on period $t-1$. The details are given in Appendix~\ref{appendix: proof theorem}.
\end{proof}

\section{Numerical Experiments} \label{sec:numerical}

In this section, we test the performance of our algorithm empirically and compare its convergence rate with other ADP algorithms on a common set of several benchmark problems with different state space sizes. Specifically, we compare with SPAR, a standard actor-critic method with a linear architecture, a policy gradient method with a linear architecture, and tabular Q-learning. We begin by giving a brief description of these algorithms.

\begin{itemize}
	\item The multi-stage version of SPAR, introduced in \cite{nascimento2009optimal}, takes advantage of the concavity of the value function and uses the temporal difference to update slopes without a policy approximation. More specifically, in order to generate observations $\hat{V}_t^{\textnormal{rep},k}$ and $\hat{V}_t^{\textnormal{dis},k}$, instead of using \eqref{eq.V-and-f-rep} and \eqref{eq.V-and-f-dis}, SPAR uses
	\begin{equation*}
	    \textstyle \hat{V}_t^{\textnormal{rep},k} (z^\textnormal{rep}, w_t^k) = -c_{w_t^k} z^\textnormal{rep} + \max_{z^\textnormal{dis} \leq z^\textnormal{rep}} \bigl\{U_{w_t^k,0}^{\mu^*} \bigl( z^\textnormal{rep} - z^\textnormal{dis}, \xi_{t,0}^{k} \bigr) + \bar{V}_t^{\textnormal{dis},k-1}\bigl(z^\textnormal{dis}, w_t^k\bigr) \bigr\},
	\end{equation*}
	and
	\begin{equation*}
	    \textstyle \hat{V}_t^{\textnormal{dis},k} (z^\textnormal{dis}, w_t^k) = (c_{w_{t+1}} - h) z^\textnormal{dis} + \max_{z^\textnormal{rep} \geq z^\textnormal{dis}} \bar{V}_{t+1}^{\textnormal{rep},k-1}\bigl(z^\textnormal{rep}, w_{t+1}\bigr) 
	\end{equation*}
	respectively.
	Although the original specification of SPAR does not use an exploration policy, we implemented $\epsilon$-greedy with exploration rate $0.1$ for improved performance.
	
	\item We implement an actor-critic (AC) method \citep{sutton1998reinforcement} based on a linear approximation architecture for both the policy and value approximations. In both cases, the basis functions are chosen to be Gaussian radial basis functions (RBFs). The ``critic'' approximates the value function using a weighted sum of RBF basis functions. The ``actor'' is a stochastic policy with a parameter $h_t(r,w; z^{\textnormal{rep}}, z^{\textnormal{dis}})$ for each state-action pair $(r,w; z^{\textnormal{rep}}, z^{\textnormal{dis}})$, and is also approximated using a weighted sum of RBFs, which indicate the tendency of selecting action $(z^{\textnormal{rep}}, z^{\textnormal{dis}})$ in state $(r,w)$. The associated stochastic policy is obtained through a softmax function, so that the probability of taking action $(z^{\textnormal{rep}}, z^{\textnormal{dis}})$ in state $(r,w)$ is $\pi_t(z^{\textnormal{rep}}, z^{\textnormal{dis}}\,|\,r,w) = e^{h(r,w; z^{\textnormal{rep}}, z^{\textnormal{dis}})}/\sum_{(z_1,z_2)} e^{h(r,w; z_1,z_2)}.$
	Detailed steps of the method are shown in Appendix~\ref{sec:alg_ac}. 
	
	\item \looseness-1 Our policy gradient (PG) method \citep{williams1992simple,sutton2000policy} updates the stochastic policy in each iteration. We adopt the Monte-Carlo policy gradient method where the policy approximation follows the same softmax policy as in the AC algorithm above. There is no value function and the policy parameters are updated using a sampled cumulative reward from $t$ to $T$.
	
	\item The previous two algorithms use linear architectures for generalization. We also compare to the widely-used Q-learning (QL) algorithm \citep{watkins1989learning}, which is called \emph{tabular} because each state-action pair is updated independently (structured actor-critic and SPAR lie in-between these two extremes as they generalize by enforcing structure). Q-learning aims to learn the state-action value function: 
	\begin{equation*}
		Q_t(r,w;z^{\textnormal{rep}},z^{\textnormal{dis}}) = (c_w-h)r - c_w z^{\textnormal{rep}} + \E_w\bigl[ U_{w_t,0}^{\mu^*} (z^{\textnormal{rep}} - z^{\textnormal{dis}}, \Xi_{t,0}) + V_{t+1} \left(z^{\textnormal{dis}},W_{t+1}\right) \bigr]. 
	\end{equation*}
	Our implementation is a standard finite-horizon version of the algorithm that uses an $\epsilon$-greedy exploration policy at a rate of $0.1$.
\end{itemize}
Optimal benchmarks used to determine the effectiveness of the five algorithms were computed using standard backward dynamic programming (BDP). All computations in this paper were performed using \texttt{Python}~3.5.

\subsection{Benchmark Instances and Parameters} 
\label{sec:benchmark_instances}
\looseness-1 {We consider 10 PODs in these synthetic benchmark instances. Each POD has a randomly generated attribute ranging between 0 and 1 representing its priority, which is reflected in the utility function $u$. Let the stochastic utility function be $\tilde{u}\bigl(\min(y_i,D_i), \xi_i\bigr)$, with expectation $u_w(x_i,\xi_i) = \E_w \bigl[ \tilde{u} \bigl(\min(\mu_i(x_i,\xi_i), D_i), \xi_i\bigr) \bigr]$, where $\mu_i$ and $D_i$ are respectively the policy and the amount of demand in sub-period $i$.\footnote{In reality, the utility and demand might be revealed several periods later. For modeling purposes, we assume that they are revealed by the end of the current period in this section.} 
Let $\tilde{u}\bigl(z, \xi_i\bigr)$ be nondecreasing and discretely concave in $z$, then $\tilde{u}\bigl(\min(y_i,D_i), \xi_i\bigr)$ is $L^\natural$-concave in $y_i$ based on Lemma 2 in \citet{chen2014coordinating}, the structural properties are kept for this stochastic utility function.\footnote{Specifically, we generate the stochastic utility function $\tilde{u}(z,\xi)$ by generating its unit utility function $\Delta \tilde{u}(z,\xi) = \tilde{u}(z,\xi) - \tilde{u}(z-1,\xi)$ as follows: $\Delta \tilde{u}(1,\xi) = 100 \,(5\, \xi^3 + 1)$, $\Delta \tilde{u}(z,\xi) = \Delta \tilde{u}(z-1,\xi) - 10\, (5\, \xi^4)$.}}
For each exogenous information realization $w$, we randomly generated 10 different patterns of the arriving POD sequences.\footnote{A pattern of the arriving POD sequences was generated from randomly sampling ten elements from a pool which contains all the PODs and some empty elements. The number of the empty elements is dependent on $w$. For example, in the case of $|\mathcal{W}|=3$, the numbers of the empty elements are 5, 10, and 15 for $w=1$, $2$, and $3$ respectively. The utility of an empty element is 0.}

Our interpretation of the stochastic process $\{W_t\}$ is a signal of the total demand\footnote{An example for $\{W_t\}$ is the national trends of the particular public health situation, which may suggest higher demands in the region-of-interest.} for period $t$. For benchmarking purposes, we use the model $W_{t+1} = \varphi_t W_t + \hat{W}_{t+1}$, where $\varphi_t$ is deterministic and $\hat{W}_{t+1}$ is an independent noise term that follows a mean zero discretized normal distribution with standard deviation $\sigma_{t+1}$. In this paper, a continuously distributed random variable $X$ is discretized to $X_\textnormal{disc}$ with $\mathbf{P}(X_\textnormal{disc} = x) = \mathbf{P}(X\leq x) - \mathbf{P}(X \leq x-1)$. Given a demand signal $W_t = w_t$, the realized demand $D_{t,i}$ is a discretized normal distribution with mean $d_i(w_t)$ and standard deviation $\tilde{\sigma}_t=3$ for $i=1,2,\ldots,n$. All of the means above were generated randomly. 

We created 25 benchmark problem instances by varying the sizes of the state, action, and outcome spaces (i.e., number of possible values of the exogenous information). Specifically, we consider problem instances with 21, 31, 41, 51, and 61 inventory levels and 3, 6, 9, 12, and 15 information states; these are the columns and rows shown in Tables~\ref{table.performance_iter} and \ref{table.performance_time}. The sizes of the action spaces corresponding to inventory level sizes 21, 31, 41, 51, and 61 are respectively 231, 496, 861, 1326, and 1891. The time horizon for each instance is $T=10$ and the cost parameters are $b=0$, $h=5$, $c_w\in[10, 50]$, $\E[c_w]=30$.

\renewcommand\arraystretch{1.2}
\begin{table}[ht]
\centering
\footnotesize
\caption{Performance (\% optimality) at iterations 500 and 1000.}
\label{table.performance_iter}
\begin{tabular}{cc|ccccc|ccccc}
\toprule

    & & \multicolumn{5}{c|}{At iteration 500} & \multicolumn{5}{c}{At iteration 1000} \\
    $R_\textnormal{max}$ & $|\mathcal{W}|$ & 3 & 6 & 9 & 12 & 15 & 3 & 6 & 9 & 12 & 15 \\\hline
    \multirow{5}{*}{20} & AC            & 97.20          & 97.68          & 98.01          & 97.41          & 96.88          & 98.86          & 99.03          & 98.50          & 98.38          & 97.60          \\
                    & PG            & 73.04          & 76.02          & 72.35          & 76.64          & 74.29          & 77.94          & 79.12          & 73.35          & 79.16          & 75.38          \\
                    & QL            & 30.02          & 33.86          & 28.36          & 27.85          & 35.53          & 32.60          & 35.91          & 31.75          & 31.20          & 37.63          \\
                    & S-AC (ours) & \textbf{99.76} & \textbf{99.26} & \textbf{98.33} & \textbf{97.68} & \textbf{97.45} & \textbf{99.83} & \textbf{99.57} & \textbf{99.00} & \textbf{98.48} & \textbf{98.50} \\
                    & SPAR          & 97.82          & 95.11          & 95.10          & 94.69          & 92.36          & 96.95          & 97.55          & 93.80          & 94.33          & 95.87          \\\hline
\multirow{5}{*}{30} & AC            & 97.21          & 96.40          & 95.75          & 95.17          & 94.91          & 97.65          & 97.13          & 96.40          & 96.31          & 95.27          \\
                    & PG            & 69.97          & 72.24          & 76.48          & 73.36          & 78.19          & 76.07          & 74.15          & 76.91          & 81.04          & 78.30          \\
                    & QL            & 38.26          & 34.09          & 28.84          & 27.47          & 34.21          & 40.35          & 37.14          & 35.43          & 33.99          & 37.78          \\
                    & S-AC (ours) & \textbf{99.58} & \textbf{99.36} & \textbf{98.53} & \textbf{97.70} & \textbf{97.61} & \textbf{99.83} & \textbf{99.67} & \textbf{99.18} & \textbf{98.67} & \textbf{98.60} \\
                    & SPAR          & 97.85          & 97.94          & 92.57          & 95.11          & 92.58          & 98.62          & 97.88          & 95.24          & 95.12          & 94.46          \\\hline
\multirow{5}{*}{40} & AC            & 96.30          & 95.16          & 91.63          & 93.24          & 92.15          & 96.70          & 96.05          & 92.56          & 93.94          & 92.54          \\
                    & PG            & 72.95          & 77.04          & 75.57          & 73.92          & 78.39          & 76.51          & 77.78          & 75.90          & 75.39          & 79.15          \\
                    & QL            & 39.65          & 35.40          & 26.71          & 24.70          & 32.36          & 42.20          & 40.57          & 35.20          & 33.44          & 37.63          \\
                    & S-AC (ours) & \textbf{99.45} & \textbf{99.35} & \textbf{97.95} & \textbf{97.86} & \textbf{97.50} & \textbf{99.65} & \textbf{99.61} & \textbf{98.90} & \textbf{98.53} & \textbf{98.43} \\
                    & SPAR          & 97.46          & 96.08          & 93.50          & 93.79          & 93.74          & 96.79          & 96.62          & 95.33          & 93.81          & 92.07          \\\hline
\multirow{5}{*}{50} & AC            & 90.96          & 90.56          & 86.47          & 88.00          & 88.02          & 91.65          & 91.76          & 87.18          & 89.03          & 89.73          \\
                    & PG            & 72.06          & 70.67          & 66.57          & 69.34          & 76.81          & 73.95          & 74.75          & 67.71          & 71.71          & 78.05          \\
                    & QL            & 41.63          & 36.35          & 26.26          & 22.00          & 29.87          & 47.68          & 42.72          & 35.22          & 31.90          & 35.60          \\
                    & S-AC (ours) & \textbf{99.42} & \textbf{99.15} & \textbf{97.49} & \textbf{97.47} & \textbf{97.09} & \textbf{99.52} & \textbf{99.46} & \textbf{98.18} & \textbf{98.25} & \textbf{98.01} \\
                    & SPAR          & 95.54          & 96.65          & 90.91          & 91.27          & 94.02          & 97.39          & 96.30          & 92.21          & 94.35          & 90.38          \\\hline
\multirow{5}{*}{60} & AC            & 91.30          & 91.16          & 86.67          & 87.21          & 88.23          & 92.27          & 92.32          & 88.32          & 87.85          & 90.17          \\
                    & PG            & 74.15          & 71.73          & 61.39          & 63.98          & 68.90          & 76.80          & 71.55          & 65.77          & 66.25          & 70.86          \\
                    & QL            & 42.03          & 33.88          & 22.51          & 19.95          & 27.38          & 47.13          & 41.37          & 33.10          & 31.44          & 33.40          \\
                    & S-AC (ours) & \textbf{99.16} & \textbf{99.00} & \textbf{96.50} & \textbf{97.08} & \textbf{96.70} & \textbf{99.25} & \textbf{99.20} & \textbf{96.81} & \textbf{97.52} & \textbf{97.00} \\
                    & SPAR          & 96.40          & 95.55          & 91.52          & 93.63          & 90.73          & 95.89          & 95.51          & 94.18          & 92.64          & 91.51\\
                    \bottomrule
\end{tabular}
\end{table}

\renewcommand\arraystretch{1.2}
\begin{table}[ht]
\centering
\footnotesize
\caption{Performance (\% optimality) after 5 and 10 seconds of CPU time.}
\label{table.performance_time}
\begin{tabular}{cc|ccccc|ccccc}
\toprule
    & & \multicolumn{5}{c|}{CPU time = 5s} & \multicolumn{5}{c}{CPU time = 10s} \\
    $R_\textnormal{max}$ & $|\mathcal{W}|$ & 3 & 6 & 9 & 12 & 15 & 3 & 6 & 9 & 12 & 15 \\\hline
    \multirow{5}{*}{20} & AC            & 91.97          & 93.91          & 92.29          & 93.21          & 89.98          & 94.80          & 95.93          & 94.87          & 95.32          & 94.32          \\
                    & PG            & 65.49          & 68.53          & 66.84          & 68.06          & 71.71          & 67.85          & 72.33          & 71.14          & 71.83          & 73.27          \\
                    & QL            & 32.60          & 35.91          & 31.75          & 31.20          & 37.63          & 32.60          & 35.91          & 31.75          & 31.20          & 37.63          \\
                    & S-AC (ours) & \textbf{99.79} & \textbf{99.52} & \textbf{98.71} & \textbf{98.20} & \textbf{98.03} & \textbf{99.83} & \textbf{99.57} & \textbf{99.00} & \textbf{98.48} & \textbf{98.50} \\
                    & SPAR          & 97.84          & 96.93          & 94.19          & 92.20          & 92.83          & 96.95          & 97.55          & 93.80          & 94.33          & 95.87          \\\hline
\multirow{5}{*}{30} & AC            & 89.78          & 89.61          & 88.64          & 88.31          & 88.74          & 93.20          & 91.90          & 92.62          & 91.26          & 91.77          \\
                    & PG            & 68.84          & 64.37          & 71.91          & 67.15          & 74.40          & 67.55          & 68.01          & 73.35          & 65.14          & 74.44          \\
                    & QL            & 40.35          & 37.14          & 35.43          & 33.99          & 37.78          & 40.35          & 37.14          & 35.43          & 33.99          & 37.78          \\
                    & S-AC (ours) & \textbf{99.62} & \textbf{99.39} & \textbf{98.00} & \textbf{97.37} & \textbf{97.40} & \textbf{99.80} & \textbf{99.67} & \textbf{99.01} & \textbf{98.35} & \textbf{98.35} \\
                    & SPAR          & 95.97          & 97.42          & 94.97          & 94.85          & 94.96          & 97.53          & 97.88          & 92.86          & 94.90          & 95.37          \\\hline
\multirow{5}{*}{40} & AC            & 86.28          & 89.01          & 83.92          & 86.91          & 85.08          & 92.58          & 92.27          & 89.01          & 89.08          & 88.65          \\
                    & PG            & 68.08          & 63.96          & 72.25          & 61.87          & 73.68          & 69.73          & 68.89          & 71.53          & 67.61          & 73.01          \\
                    & QL            & 42.20          & 40.57          & 35.20          & 33.44          & 37.63          & 42.20          & 40.57          & 35.20          & 33.44          & 37.63          \\
                    & S-AC (ours) & \textbf{99.43} & \textbf{98.94} & \textbf{96.85} & \textbf{96.58} & \textbf{95.63} & \textbf{99.58} & \textbf{99.41} & \textbf{98.32} & \textbf{97.94} & \textbf{97.57} \\
                    & SPAR          & 97.73          & 96.98          & 92.88          & 92.47          & 92.49          & 96.60          & 96.28          & 95.10          & 92.54          & 92.97          \\\hline
\multirow{5}{*}{50} & AC            & 86.79          & 81.48          & 79.76          & 81.90          & 80.34          & 88.10          & 88.05          & 82.45          & 84.92          & 83.29          \\
                    & PG            & 67.47          & 68.04          & 63.74          & 60.86          & 69.41          & 66.57          & 65.26          & 65.34          & 65.01          & 71.29          \\
                    & QL            & 47.60          & 42.72          & 35.22          & 31.34          & 35.60          & 47.68          & 42.72          & 35.22          & 31.90          & 35.60          \\
                    & S-AC (ours) & \textbf{98.92} & \textbf{98.51} & \textbf{95.10} & \textbf{94.60} & \textbf{93.92} & \textbf{99.36} & \textbf{99.06} & \textbf{97.01} & \textbf{96.83} & \textbf{96.33} \\
                    & SPAR          & 96.85          & 96.62          & 93.17          & 88.63          & 91.35          & 95.02          & 96.27          & 94.17          & 92.61          & 92.27          \\\hline
\multirow{5}{*}{60} & AC            & 77.05          & 76.73          & 73.10          & 74.15          & 76.31          & 78.78          & 80.14          & 76.33          & 76.38          & 78.48          \\
                    & PG            & 66.10          & 58.84          & 56.40          & 59.41          & 64.81          & 69.42          & 59.67          & 57.36          & 60.11          & 65.73          \\
                    & QL            & 44.75          & 38.60          & 26.78          & 26.41          & 31.62          & 47.13          & 41.37          & 33.10          & 31.44          & 33.40          \\
                    & S-AC (ours) & \textbf{98.65} & \textbf{98.02} & \textbf{93.17} & \textbf{93.20} & 92.31          & \textbf{99.05} & \textbf{98.75} & \textbf{95.41} & \textbf{96.16} & \textbf{95.11} \\
                    & SPAR          & 95.45          & 95.26          & 90.50          & 92.43          & \textbf{92.41} & 95.90          & 96.15          & 93.45          & 93.11          & 93.02 \\
                    \bottomrule
\end{tabular}
\end{table}

\subsection{Optimality Gap of Approximate Policies}
To estimate the value $V^{\tilde{\pi}^k}_0$ of an approximate policy $\tilde{\pi}^k$, {we averaged the value of initial states $(r,w)$ drawn from a uniform distribution}, where the value $V^{\tilde{\pi}^k}_0(r,w)$ is obtained from 100 Monte Carlo simulations following policy $\tilde{\pi}^k$. To evaluate the approximate policy learned from an ADP algorithm, we {run 10 independent replications of the algorithm} and average the performance of the learned approximate policy in each replication. Denote $\bar{V}^{\tilde{\pi}^k}_0$ the evaluation of the approximate policy learned from an algorithm. The percentage of optimality is the ratio of $\bar{V}^{\tilde{\pi}^k}_0$ to $V_0$, where the optimal value function $V_0$ is computed using BDP. 

Tables~\ref{table.performance_iter} and \ref{table.performance_time} show the percentage of optimality of each algorithm at specific iterations and CPU times, across all problem instances. {In almost all instances and comparison points, S-AC outperforms the baseline algorithms. AC is the most competitive baseline with respect to the number of iterations and SPAR is the most competitive when CPU time is of primary interest.} Within the same number of iterations and CPU times, the performance of all the ADP algorithms becomes worse as the size of the problem increases; however, S-AC seems to be less sensitive than the others to problem size. Let us compare the percentage of optimality of the instance with $R_\textnormal{max}=20$ and $|\mathcal{W}|=3$ and the instance with $R_\textnormal{max}=60$ and $|\mathcal{W}|=15$ at iteration 1,000. The performance of AC, S-AC, and SPAR on the latter large instance is respectively 7.2, 2.8, and 6.5 percentage worse than the performance on the smaller instance. For the same instance at CPU time 10 seconds, the performance of the three algorithms on the larger instance is respectively 14.7, 4.7, and 4.9 percentage points worse than the performance on the smaller instance.

To further illustrate the performance of each ADP algorithm, we show the convergence curves of three instances with different sizes. Let us consider three problem instances: (1) $R_\textnormal{max}=20$, $|\mathcal{W}|=3$, (2) $R_\textnormal{max}=40$, $|\mathcal{W}|=9$, and (3) $R_\textnormal{max}=60$, $|\mathcal{W}|=15$. Figure~\ref{fig.numerical.iter} shows the rate of convergence of the ADP algorithms considered in this paper as a function of the number of iterations, while Figure~\ref{fig.numerical.time} shows the rate of convergence as a function of the computation time. {We plot ``log regret'' (log of the suboptimality from 100\%) to help improve the visualization.}

The policy approximations used in AC and PG are parameterized as stochastic policies initialized to take uniformly random actions in each state. This exploration helps to generate relatively high value in early iterations. AC and PG are very competitive with our S-AC algorithm when comparing performance with respect to the iteration count. However, this comes at a computational cost: although stochasticity encourages exploration, Figure~\ref{fig.numerical.time} shows that each iteration is particularly time-consuming when compared to deterministic policies. 

\begin{figure}[h]
	\centering
	\begin{subfigure}[ht]{0.29\textwidth}
		\includegraphics[width=1\textwidth]{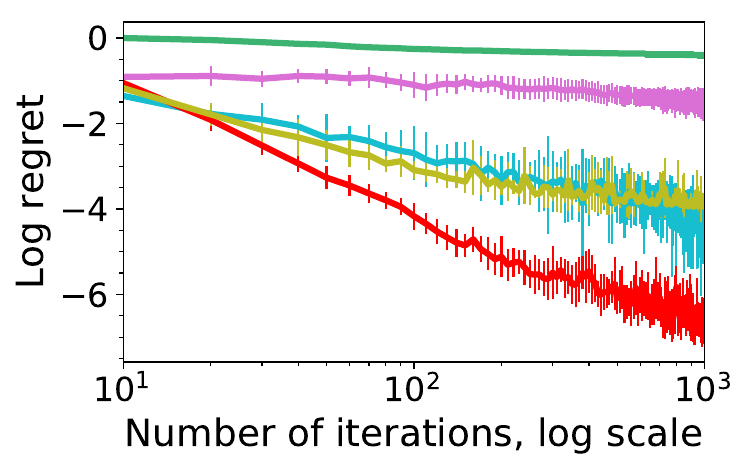}
		\caption{$R_{\textnormal{max}}=20, |\mathcal{W}|=3$}
		\label{fig.numerical.iter.w3r20}
	\end{subfigure}
	\begin{subfigure}[ht]{0.29\textwidth}
		\includegraphics[width=1\textwidth]{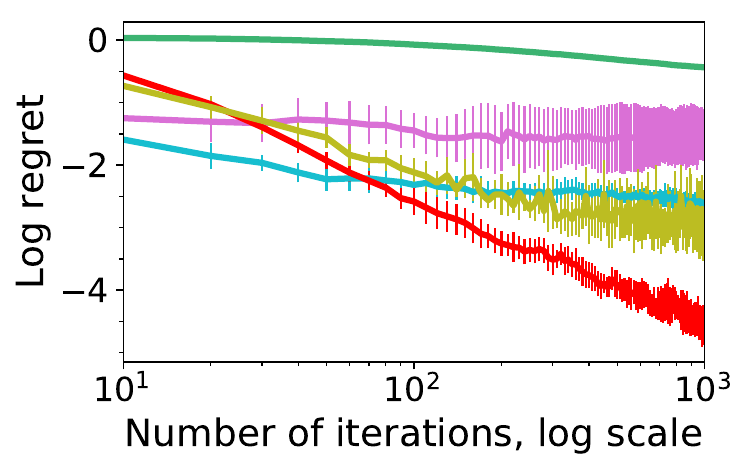}
		\caption{$R_{\textnormal{max}}=40, |\mathcal{W}|=9$}
		\label{fig.numerical.iter.w9r40}
	\end{subfigure}
	\begin{subfigure}[ht]{0.4\textwidth}
		\includegraphics[width=1\textwidth]{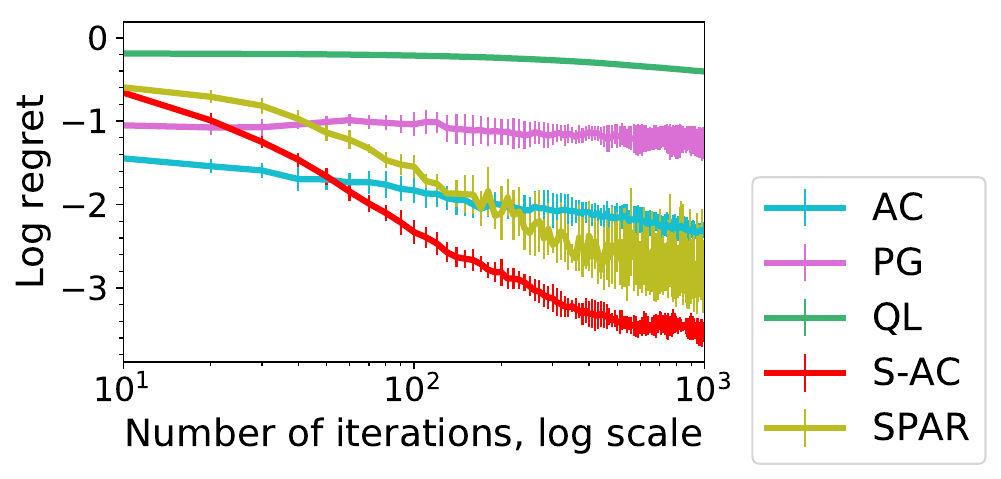}
		\caption{$R_{\textnormal{max}}=60, |\mathcal{W}|=15$}
		\label{fig.numerical.iter.w15r60}
	\end{subfigure}
	\caption{Comparison of ADP algorithms with respect to iteration number.}
	\label{fig.numerical.iter}
\end{figure}

\begin{figure}[h]
	\begin{subfigure}[ht]{0.29\textwidth}
		\includegraphics[width=1\textwidth]{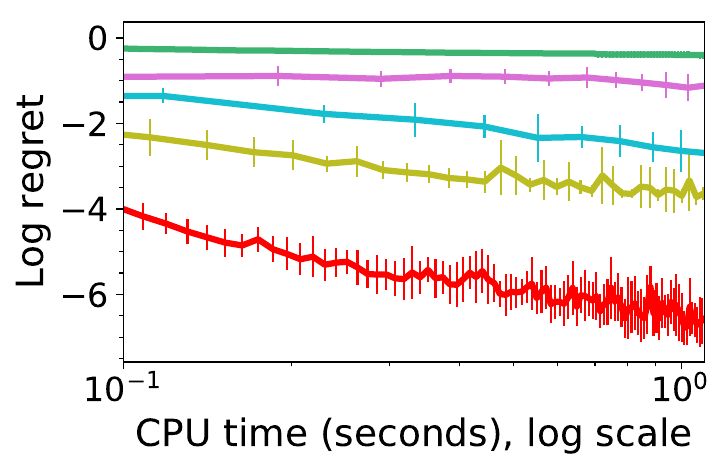}
		\caption{$R_{\textnormal{max}}=20, |\mathcal{W}|=3$}
		\label{fig.numerical.iter.w3r20}
	\end{subfigure}
	\begin{subfigure}[ht]{0.29\textwidth}
		\includegraphics[width=1\textwidth]{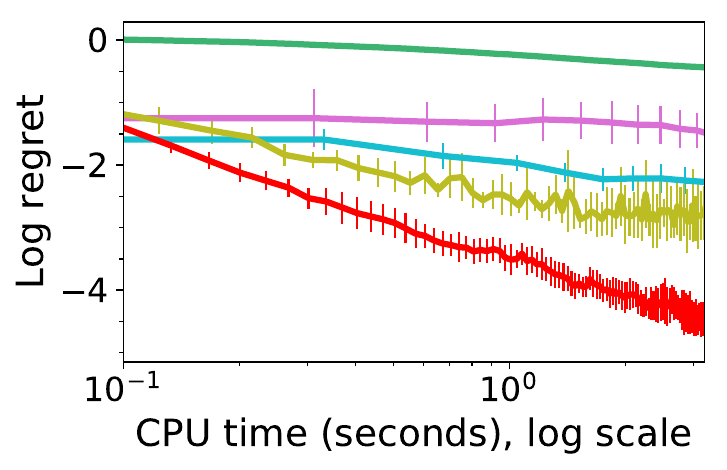}
		\caption{$R_{\textnormal{max}}=40, |\mathcal{W}|=9$}
		\label{fig.numerical.iter.w9r40}
	\end{subfigure}
	\begin{subfigure}[ht]{0.4\textwidth}
		\includegraphics[width=1\textwidth]{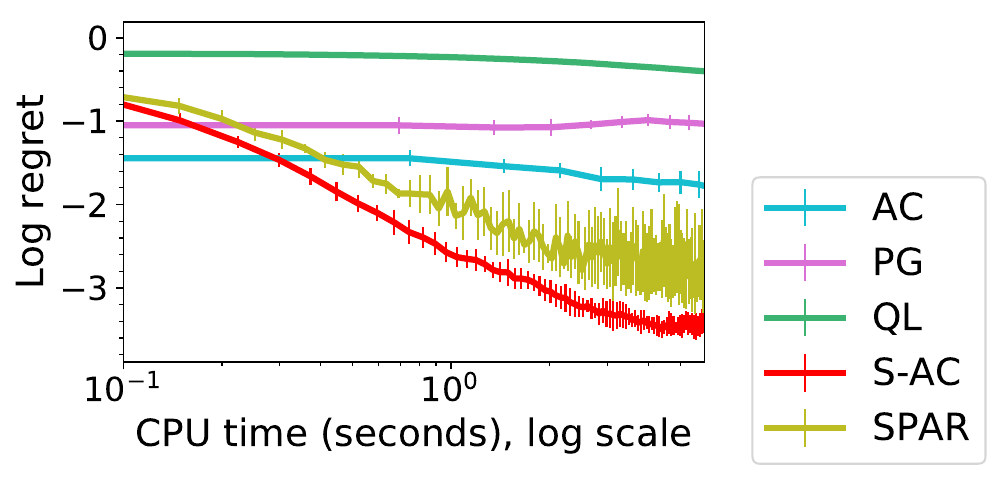}
		\caption{$R_{\textnormal{max}}=60, |\mathcal{W}|=15$}
		\label{fig.numerical.iter.w15r60}
	\end{subfigure}
	\caption{Comparison of ADP algorithms with respect to CPU time.}
	\label{fig.numerical.time}
\end{figure}

\begin{figure}[h]
	\centering
	\begin{subfigure}[ht]{0.28\textwidth}
		\includegraphics[width=1\textwidth]{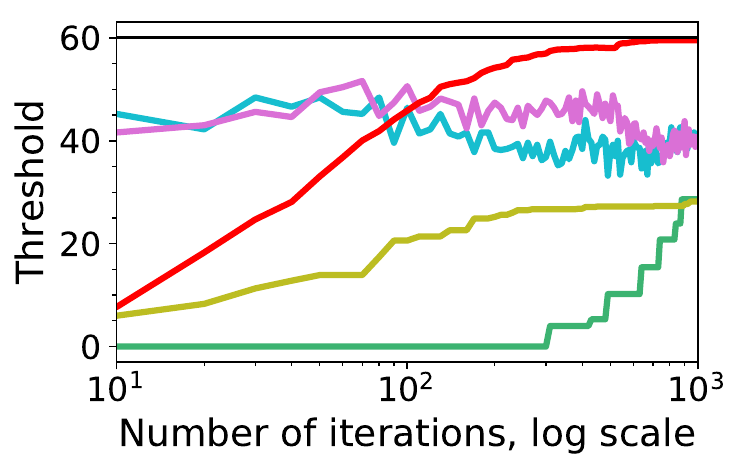}
		\caption{$w_0=1$}
		\label{fig.numerical.bs.w9r60.1}
	\end{subfigure}\;\;\;
	\begin{subfigure}[ht]{0.28\textwidth}
		\includegraphics[width=1\textwidth]{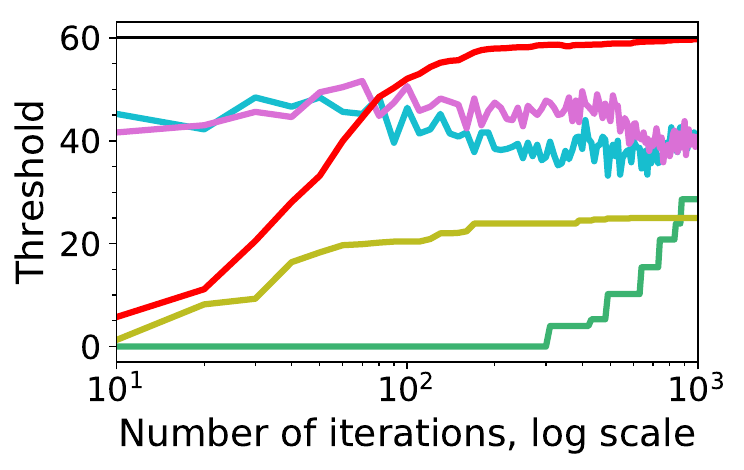}
		\caption{$w_0=4$}
		\label{fig.numerical.bs.w9r60.2}
	\end{subfigure}\;\;\;
	\begin{subfigure}[ht]{0.38\textwidth}
		\includegraphics[width=1\textwidth]{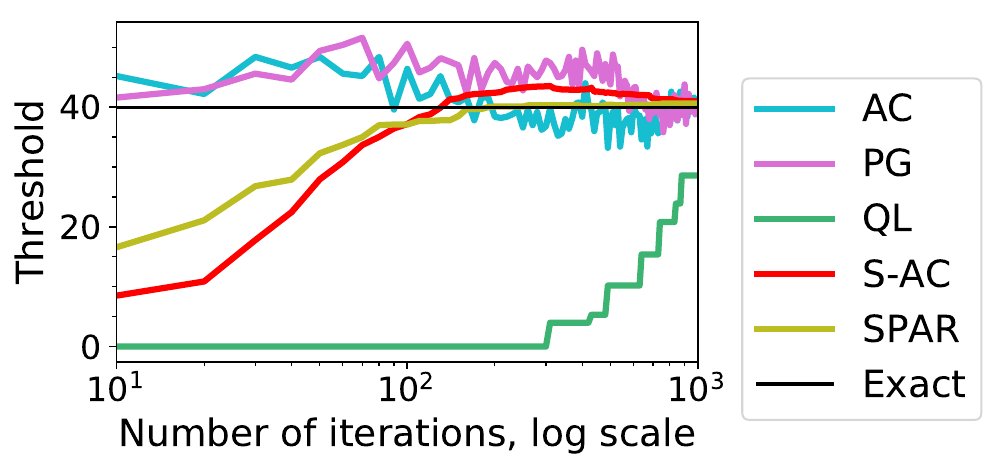}
		\caption{$w_0=8$}
		\label{fig.numerical.bs.w9r60.3}
	\end{subfigure}
	\caption{Convergence of replenish-up-to thresholds at $t=0$ for the $R_{\textnormal{max}} = 60, |\mathcal{W}| = 9$ instance.}
	\label{fig.numerical.bs.w9r60}
\end{figure}

\begin{figure}[h]
	\centering
	\begin{subfigure}[ht]{0.28\textwidth}
		\includegraphics[width=1\textwidth]{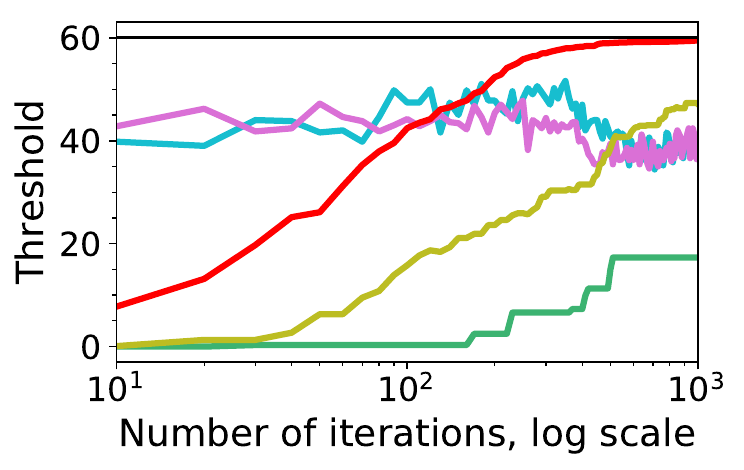}
		\caption{$w_0=2$}
		\label{fig.numerical.bs.w12r60.1}
	\end{subfigure}\;\;\;
	\begin{subfigure}[ht]{0.28\textwidth}
		\includegraphics[width=1\textwidth]{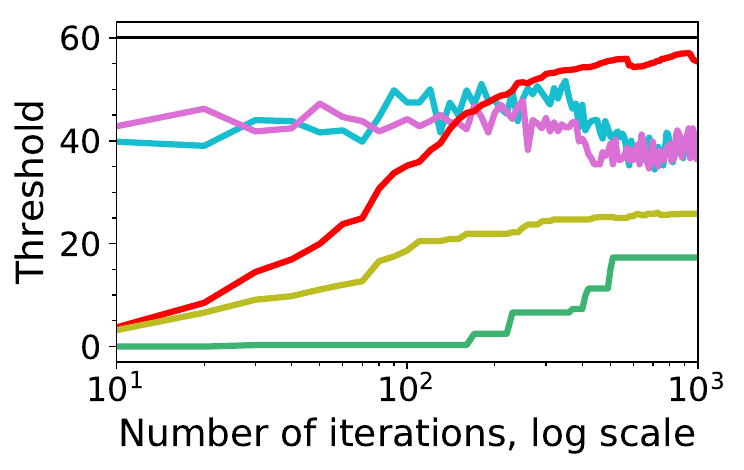}
		\caption{$w_0=6$}
		\label{fig.numerical.bs.w12r60.2}
	\end{subfigure}\;\;\;
	\begin{subfigure}[ht]{0.38\textwidth}
		\includegraphics[width=1\textwidth]{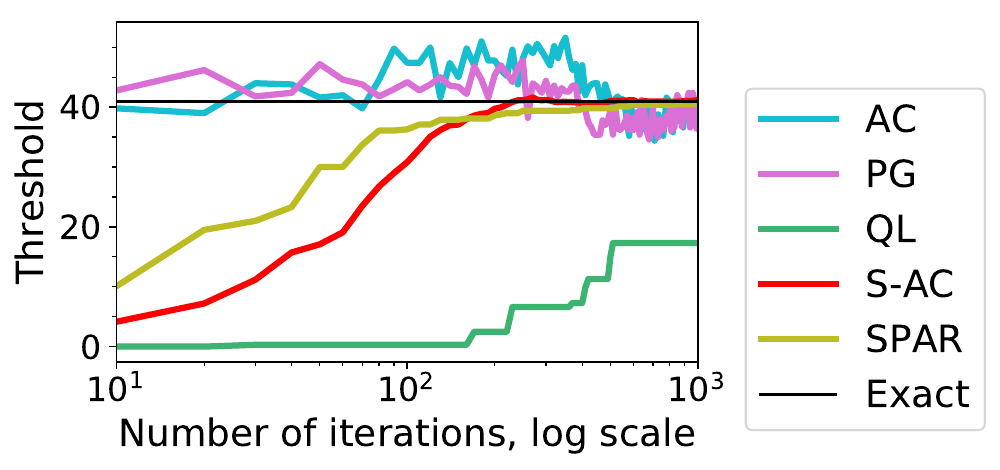}
		\caption{$w_0=10$}
		\label{fig.numerical.bs.w12r60.3}
	\end{subfigure}
	\caption{Convergence of replenish-up-to thresholds at $t=0$ for the $R_{\textnormal{max}} = 60, |\mathcal{W}| = 12$ instance.}
	\label{fig.numerical.bs.w12r60}
\end{figure}

\begin{figure}[h]
	\centering
	\begin{subfigure}[ht]{0.28\textwidth}
		\includegraphics[width=1\textwidth]{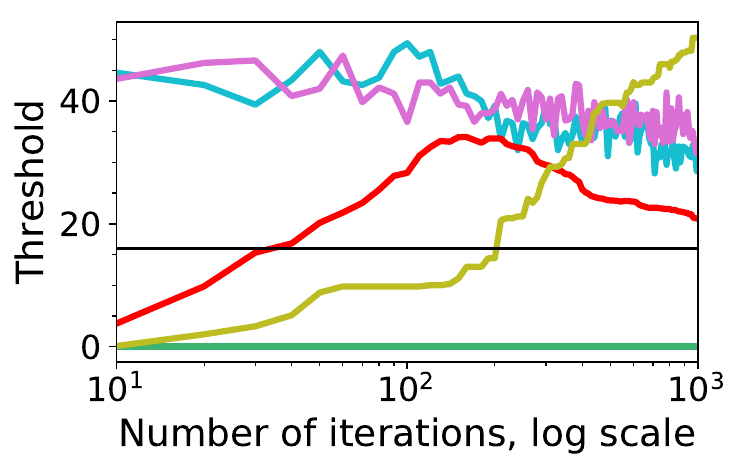}
		\caption{$w_0=2$}
		\label{fig.numerical.bs.w15r60.2}
	\end{subfigure}\;\;\;
	\begin{subfigure}[ht]{0.28\textwidth}
		\includegraphics[width=1\textwidth]{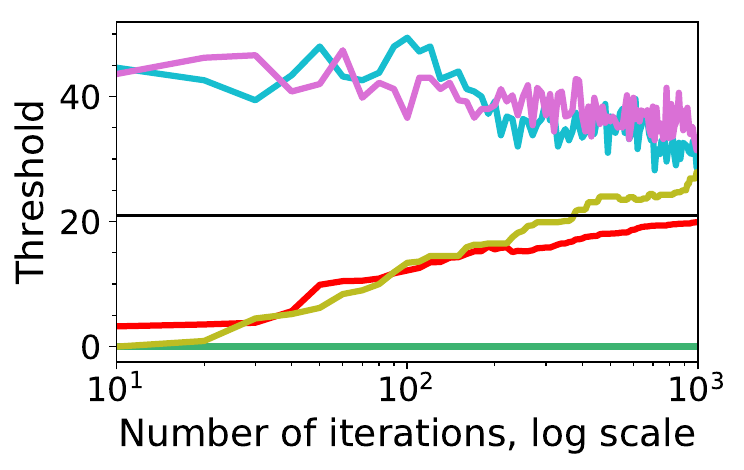}
		\caption{$w_0=7$}
		\label{fig.numerical.bs.w15r60.7}
	\end{subfigure}\;\;\;
	\begin{subfigure}[ht]{0.38\textwidth}
		\includegraphics[width=1\textwidth]{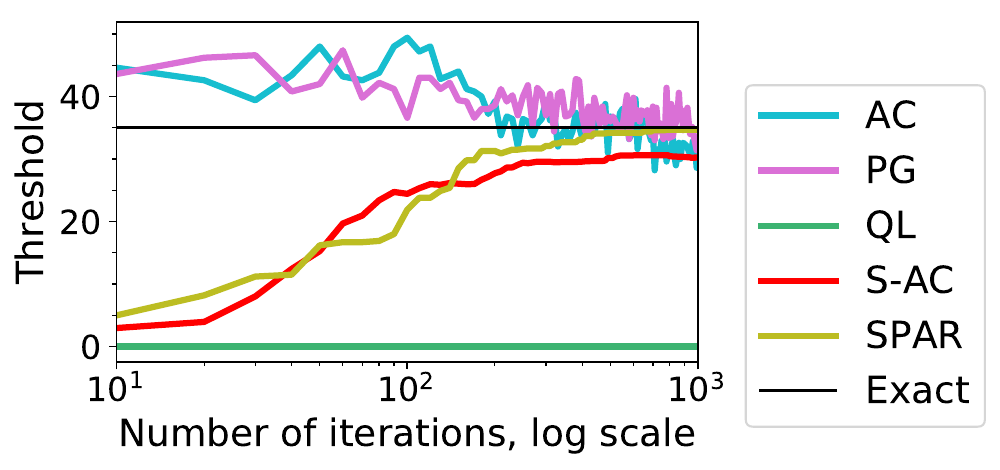}
		\caption{$w_0=12$}
		\label{fig.numerical.bs.w15r60.12}
	\end{subfigure}
	\caption{Convergence of replenish-up-to thresholds at $t=0$ for the $R_{\textnormal{max}} = 60, |\mathcal{W}| = 15$ instance.}
	\label{fig.numerical.bs.w15r60}
\end{figure}

\subsection{Convergence of Implied Basestock Thresholds}
Next, we are interested in examining how the implied replenish-up-to thresholds evolve as each algorithm progresses. The thresholds of AC and PG are selected as the actions with highest probabilities for state $r=0$ and the thresholds of SPAR and QL correspond to the greedy policy with respect to the value function and state-action value function approximations. In this part, we take three problem instances as examples, whose storage capacities are all $R_\textnormal{max}=60$, and exogenous information spaces are $|\mathcal{W}|=9$, $|\mathcal{W}|=12$ and $|\mathcal{W}|=15$ respectively. Figures~\ref{fig.numerical.bs.w9r60} to \ref{fig.numerical.bs.w15r60} show the convergence of approximate replenish-up-to threshold levels $\bar{l}^{\textnormal{rep},k}$ as well as the optimal levels $l^{\textnormal{rep}}$ (denoted ``Exact'' in the plots) for three different exogenous information states $w_0$ at period $t=0$ for the selected problem instances. 

We see that the thresholds generated by S-AC quickly converge to the optimal ones in all instances. Due to the smoothing step of S-AC, the convergence is also observed to be relatively stable. On the other hand, the thresholds of AC, PG, QL, and SPAR tend to either have large gaps to the optimal thresholds or converge in a noisy manner. Stability of the basestock thresholds is particularly useful if S-AC is to be used in an online manner in practice, where drastic changes in the policy from one time period to the next (as observed in the competing algorithms) would be impractical. These results attest to the value of utilizing the structural properties of the policy and value function.

\subsection{Sensitivity Analysis}
\looseness-1 In this section, we study the impact of model parameters. We take the instance with $R_{\textnormal{max}}=50$ and $|\mathcal{W}|=9$ in Section~\ref{sec:benchmark_instances} as the base instance, and vary parameters in the model to evaluate the impact of each parameter. The results are summarized in Table~\ref{table:sensitivity}. Each value in the table is an average of ten replications. For each replication, we take the policy learned by the algorithm at iteration 1,000 and evaluate it by averaging 100 simulations. The first parameter we are interested in is the demand distribution. We consider two types of distribution, normal and uniform distributions. For each type of distribution, we consider two values of the average demand of all PODs in a period, 30 and 50. The table shows that the value is highly influenced by the expected demand, and that with the same expected demand, the type of distribution has relatively little impact on the performance. We are also interested in the impact of the costs in the model. The ordering cost has a much larger impact than the holding cost, and any increase in the ordering cost can significantly reduce the value of the policy. We also note that S-AC finds near-optimal policies in each of these cases.

\renewcommand\arraystretch{1.2}
\begin{table}[ht]
\centering
\footnotesize
\caption{Impact of parameters on ADP algorithms for the $R_{\textnormal{max}} = 50, |\mathcal{W}| = 9$ instance.}
\label{table:sensitivity}
\begin{tabular}{cc|cccccc}
\toprule
Parameter & Value & AC & PG & QL & S-AC & SPAR & Exact \\\hline
\multirow{3}{*}{Mean total demand} & 30, Normal & 19,037 & 16,009 & 7,287  & \textbf{20,313} & 19,077 & 21,332 \\
& 30, Uniform & 18,113 & 15,142 & 8,476  & \textbf{20,865} & 20,098 & 21,332 \\
& 50, Normal & 28,422 & 23,237 & 10,318 & \textbf{29,080} & 28,278 & 29,387 \\
& 50, Uniform & 28,023 & 23,112 & 10,286 & \textbf{29,077} & 28,150 & 29,387 \\\hline
\multirow{3}{*}{Mean ordering cost} & 30 & 30,914 & 25,488 & 15,125 & \textbf{33,532} & 32,671 & 34,647 \\
& 50 & 18,037 & 14,009 & 7,287  & \textbf{20,313} & 19,077 & 20,689 \\
& 70 & 11,257 & 8,660  & 6,032  & \textbf{11,866} & 11,553 & 11,984 \\\hline
\multirow{5}{*}{Holding cost} & 5 & 18,037 & 14,009 & 7,287  & \textbf{20,313} & 19,077 & 20,689 \\
& 20 & 18,402 & 15,064 & 7,189  & \textbf{19,839} & 19,285 & 20,131 \\
& 35 & 17,807 & 14,498 & 5,855  & \textbf{19,381} & 18,784 & 19,592 \\
& 50 & 17,150 & 15,011 & 4,582  & \textbf{18,988} & 18,418 & 19,203 \\
& 65 & 16,575 & 13,708 & 2,954  & \textbf{18,597} & 17,931 & 18,835 \\
\bottomrule
\end{tabular}
\end{table}

\section{Case Study: Naloxone for First Responders in Pennsylvania} \label{sec:case_study}

Our case study is motivated by the need to distribute naloxone (a drug that can reverse overdoses within seconds to minutes) amidst the ongoing opioid overdose crisis, which is affecting communities across the state of Pennsylvania. Our case study makes use a time-series demand model for naloxone, fit using publicly available data from \citet{overdose2021}.
Our model in this section contains a five-dimensional information state $W_t$, which makes the standard version of S-AC intractable. Instead, we leverage an aggregation-based version of S-AC, whose details are introduced in Appendix~\ref{sec:aggregation}. In essence, the method uses clusters of the exogenous information state (via $k$-means clustering) and learns a cluster-dependent policy. When implementing the policy, we use regression to interpolate between clusters. Our experimental results show that this simple extension of S-AC for the case of a continuous and multi-dimensional information state is surprisingly effective.

\subsection{Description of Naloxone for First Responders in Pennsylvania}

The rate of opioid overdose deaths has quadrupled since 1999 \citepalias{cdc2021understanding}, with heroin deaths alone outpacing gun homicides in 2015 \citep{ingraham2016heroin}. Moreover, in 2015, drug overdose deaths in U.S. exceeded the combined mortalities from car accidents and firearms \citep{DEA2015NDTA,DEA2015NDTAannouncement}. By August 2020, the number of deaths from synthetic opioids was 52\% more than the previous year \citep{opioid2021}. There is significant benefit for drug users, family members, community members, law enforcement officers, and medical professionals alike to have training and access to the overdose reversal drug naloxone for use in risky situations (see Pennsylvania’s Act 139).

\begin{table}[ht]
	\setlength\extrarowheight{2pt}
	\centering
	\caption{Parameters used in the NFRP case study.}
	\label{table. naloxone parameters}
	\footnotesize
	\begin{tabular}{p{0.9in} >{\arraybackslash}p{0.5in} p{4.5in}}
		\toprule
		\multicolumn{1}{c}{Parameter} & Value & \multicolumn{1}{c}{Meaning/Explanation} \\
		\hline 
		WTP/unit & \$31,000 & Willingness to pay (WTP) for a unit of naloxone. Product of the next 2 entries. \\ 
		WTP/QALY & \$50,000 & WTP per quality-adjusted life-year (QALY) \citep{kaplan1982health, coffin2013cost}. \\ \addlinespace[3pt]
		QALY/unit & 0.62 & \looseness-1 QALY adjustment factor for lives saved by naloxone. The average of utilities of ``High-risk/low-risk prescription opioid use'' and ``Illicit opioid use'' in \citet{acharya2020cost}. \\
		Ordering cost & \$185.30 & Approximate retail price of an auto-injector form of naloxone \citep{naloxoneprice}. \\
		Treatment cost & \$2,976 & The cost for EMS visit, EMS transport to hospital, and emergency department care \citep{coffin2013cost}.\\
		$R_\textnormal{max}$ & 700 & Capacity of the central storage. \\ 
		$h$ & \$10 & Holding cost.\\
		$b$ & 0 & Disposal cost.\\
		\bottomrule
	\end{tabular}
\end{table}

\looseness-1 In this case study, we consider a somewhat simplified setting of a public health organization modeled after Naloxone for First Responders Program (NFRP), which distributes naloxone through a Centralized Coordination Entity (CCE). We use the top five counties in terms of overdose incidents responded to by emergency medical services (EMS) from publicly available data \citep{overdose2021}, Allegheny County, York County, Bucks County, Dauphin County, and Luzerne County (all of which have incident numbers over 1,000), as the five PODs (first responders) in our case study. 

The parameters of our utility function are based on values found in \citet{kaplan1982health}, \citet{coffin2013cost}, \citet{acharya2020cost}, and \citet{overdose2021}. Since the naloxone dispensed to first responders is used to reverse overdoses, we use willingness to pay (WTP) per unit of naloxone to measure the utility per demand satisfied. Specifically, let the WTP per unit of naloxone minus the treatment cost (EMS visit and related costs) be the unit utility $\Delta u$, similar to the approach taken by \citet{coffin2013cost}. To reflect the different expected demand among counties, we adopt the following expected utility function in the case study: $u_w(y_i,\xi_i) = \Delta u \, \E_w\bigl[\min(y_i,D_i)\bigr]$, where $D_i$ is the demand of POD $i$.\footnote{{The demand is computed as follows: based on data from \citep{overdose2021}, 1-9 doses of naloxone are administrated to reverse an overdose. The demand of POD $i$ at period $t$ equals to a sample of the doses of naloxone needed to reverse $w^i$ incidents, where $w^i$ is the $i$-th element of $w$.}} Further details (ordering cost, capacity of storage, and holding cost) are available in Table~\ref{table. naloxone parameters}.

\begin{figure}[ht]
	\centering
	\begin{tikzpicture}
	\footnotesize
    \tikzset{edge from parent/.style={draw,edge from parent path={(\tikzparentnode.south)-- +(0,-6pt)-| (\tikzchildnode)}}}
    \Tree [.{Inventory Control Center}
    [.{Centralized Coordinating Entity (CCE)}
    [.{Allegheny County} ]
    [.{York County} ]
    [.{Bucks County} ]
    [.{Dauphine County} ]
    [.{Luzerne County} ]
    [.{$\cdots$} ] ] ]
    \end{tikzpicture}
	\caption{The hierarchical system structure used in the case study.}
	\label{fig:case_system}
\end{figure}

\looseness-1 {The system consists of an inventory control center, a dispensing coordinator, and multiple first responders as shown in Figure~\ref{fig:case_system}. Let the time horizon for the case study be $T=12$ months. 
At each period $t$, the inventory control center replenishes the inventory of naloxone after observing the recent incident history, modeled as the county-level incident count of the last period (thus, $W_t \in \mathbb R^5$). The control center then decides the total amount of naloxone to dispense in the current period. This naloxone is delivered to the CCE, who makes lower-level quantity-of-dispensing decisions based on the attribute of the arriving POD $\xi_{i}$, the current available naloxone in stock (the inventory level $x_i$), and the upper-level county-level incident count of the last period $W_t$.}
In the case study, the exogenous information is the incident history, which consists of the number of incidents from the five counties last month. We a vector autoregression (VAR) time-series model with a lag of 1.

Figure~\ref{fig:case_time_series} shows the monthly number of overdose incidents in the five counties from January 1st, 2018 to July 31st, 2020, and 20 sample paths from the VAR(1) model for the next 24 months. The first planning period of the case study is July 2020.
To generate the state aggregation, we sample 10,000 paths of the exogenous information, and use $k$-means clustering to cluster them into 12 clusters. Figure~\ref{fig:case_clustered_w} shows the first three dimensions of the resulting clustering that is then used by S-AC.

\begin{figure}[ht]
	\centering
	\begin{subfigure}[ht]{0.4\textwidth}
		\includegraphics[width=0.95\textwidth]{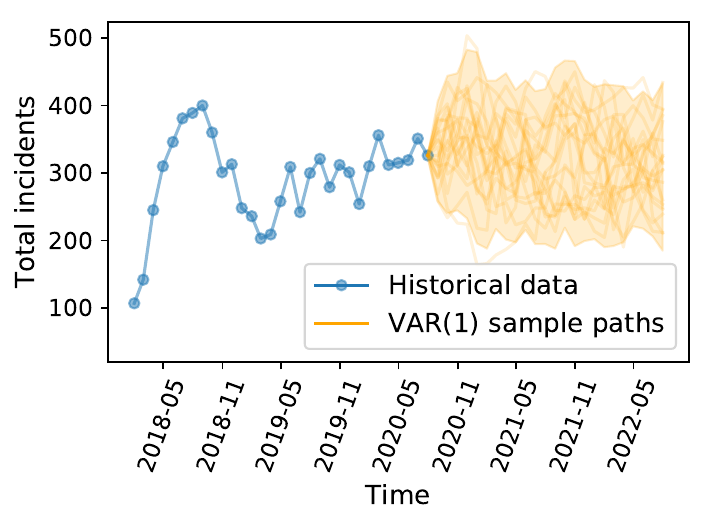}
		\caption{Total number of incidents and predictions.}
		\label{fig:case_time_series}
	\end{subfigure}
	\begin{subfigure}[ht]{0.46\textwidth}
	    \vspace{-8pt}
		\includegraphics[width=0.95\textwidth]{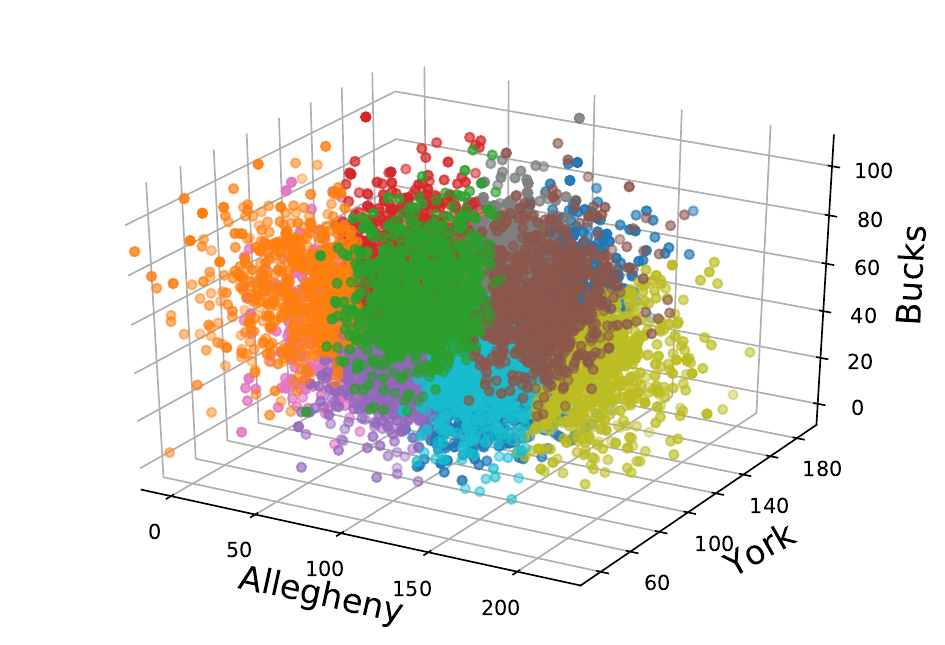}
		\caption{Three dimensions of $k$-means clustering.}
		\label{fig:case_clustered_w}
	\end{subfigure}
	\caption{Total overdose incidents of the five PODs and $k$-means visualization.}
	\label{fig:case.intro}
\end{figure}

\subsection{Performance of the Algorithm}
We denote the aggregate version of our algorithm \texttt{S-AC+DPR}, whose upper-level policies are learned by aggregate S-AC (see Appendix~\ref{sec:aggregation}). The learned cluster-dependent upper-level policies are then interpolated between clusters using Gaussian process regression. The lower-level policies are solved using a discretized DP we then interpolate using linear regression (\texttt{DPR}). In this section, we first study the performance of \texttt{S-AC+DPR} compared with \texttt{AC+DPR} and a suite of heuristic strategies. Next, we illustrate the impact of the various approaches on the lower-level dispensing decisions to each POD, showing some stark differences between the methods. Finally, we show some sensitivity analysis of the cost parameters on the value of the learned policies and heuristics.

\subsubsection{Convergence and Comparison with Heuristics}
\looseness-1 We first describe the heuristic strategies to which we compare our new policy. We make a distinction between the upper-level and lower-level policies and consider two approaches for the upper-level and three approaches for the lower-level, resulting in six combined strategies. On the upper-level, we either take the S-AC policies (\texttt{S-AC}) or always replenish-up-to the expected demand\footnote{The expected demand for a given exogenous information $w$ equals to the sum of the elements of $w$ times the average doses per reverse (1.517), which is computed by averaging the ``dose count'' in the dataset \citet{overdose2021} (excluding the cases without applying naloxone).} and dispense-down-to zero (\texttt{Mean}). On the lower-level, the three strategies are: (1) take the policy trained using dynamic programming and interpolated to the continuous state space by linear regression (\texttt{DPR}), (2) evenly dispense naloxone to the five PODs (\texttt{Even}), and (3) follow the first-come-first-serve rule (\texttt{FCFS}), in which we dispense the expected demand of each POD upon its arrival until all the available resources are dispensed.
We also apply \texttt{AC+DPR} as an alternative ADP method to which we can compare \texttt{S-AC+DPR}. We selected \texttt{AC} because it performs relatively well in Section~\ref{sec:numerical} and is scalable to high-dimensional problems (unlike \texttt{QL} or \texttt{SPAR}).

\begin{figure}[ht]
	\centering
	\includegraphics[width=0.7\textwidth]{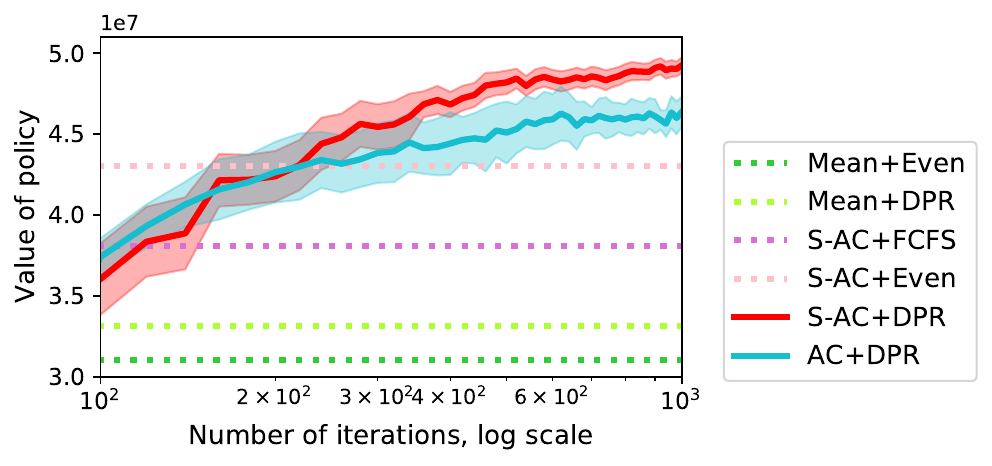}
	\caption{Convergence curve of S-AC and AC compared to performance of heuristics.}
	\label{fig.case.conv}
\end{figure}

\looseness-1 Figure~\ref{fig.case.conv} shows the cumulative performance of the policies over a year, averaged over 100 simulations (the value of policy \texttt{Mean+FCFS} is smaller than $3\mathrm{e}{7}$ and is removed from the plot to better show the results). 
We see that compared with the upper-level heuristic \texttt{Mean}, applying \texttt{S-AC} on the upper-level improves the performance (i.e., compare \texttt{S-AC+DPR} with \texttt{Mean+DPR} and \texttt{S-AC+Even} with \texttt{Mean+Even}). This is due to the ability of the state-dependent basestocks to adapt to dynamic state information.
On the lower-level, we see that \texttt{DPR} outperforms the heuristics \texttt{FCFS} and \texttt{Even} (i.e., compare \texttt{S-AC+DPR} with \texttt{S-AC+FCFS} and \texttt{S-AC+Even}, and \texttt{Mean+DPR} with \texttt{Mean+Even}) significantly. The reason is that the heuristics \texttt{FCFS} dispensing policy is unable to take advantage of the large initial gains in dispensing resources to all of the first responders, and the heuristic \texttt{Even} dispensing policy is unable to adjust the dispensing decision to a POD based on exogenous information.
\begin{figure}[ht]
	\centering
	\includegraphics[width=0.7\textwidth]{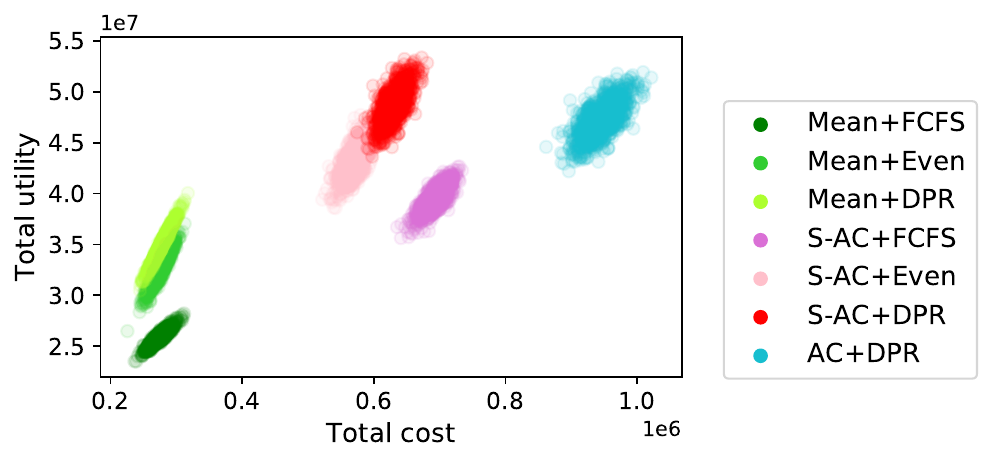}
	\caption{The relationship between total cost and total utility for each method.}
	\label{fig:case.pareto}
\end{figure}

Figure~\ref{fig:case.pareto} shows the total cost vs. total utility for each method that we tested, which helps to illustrate the trade-offs associated with each. The total cost is mostly determined by the upper-level policy (i.e., the scatters of \texttt{Mean+FCFS}, \texttt{Mean+Even} and \texttt{Mean+DPR} are close on the x-axis, and the scatters of \texttt{S-AC+FCFS}, \texttt{S-AC+Even} and \texttt{S-AC+DPR} are close on the x-axis). The upper-level policy \texttt{AC} tends to always replenish the inventory up to a high level, which leads to the highest total cost. The heuristics \texttt{Mean} considers the exogenous information by always replenishing up to the expected demand and dispense all the inventory to the PODs; this approach leads to the lowest total cost. 
The upper-level policy learned by \texttt{S-AC} is able to adapt to the exogenous information state and usually replenishes up to a level that is higher than the expected demand. It also sometimes retains a small portion of the inventory to the next period. 
With the same upper-level policy, although the total cost is similar, the total utility differs when applying different lower-level policy. This observation suggests that by applying a smarter lower-level policy \texttt{DPR}, it is possible to achieve more utility without spending much more cost. Overall, we see that our primary approach \texttt{S-AC+DPR} \emph{attains the highest levels of utility while expending relatively moderate cost}.

\subsubsection{Utilities of Different First Responders}

We now investigate the individual POD (or first responder) utilities achieved under each algorithm. Following the policy obtained after $1,000$ iterations of \texttt{S-AC+DPR}, we get the total utility of each POD during the entire planning horizon. Under our utility function definition and the parameters given in Table \ref{table. naloxone parameters}, the PODs with higher levels of overdose incidents are associated with a higher utilities than PODs fewer incidents. Thus, we expect that a good inventory and dispensing policy will learn to prioritize these high-utility PODs. Figure~\ref{fig:case.hist_sac} shows the histograms of 1,000 simulations for the utilities of the five PODs alongside the historical county-level overdose incidents.

\begin{figure}[ht]
	\centering
	\begin{subfigure}[ht]{0.4\textwidth}
		\includegraphics[width=1\textwidth]{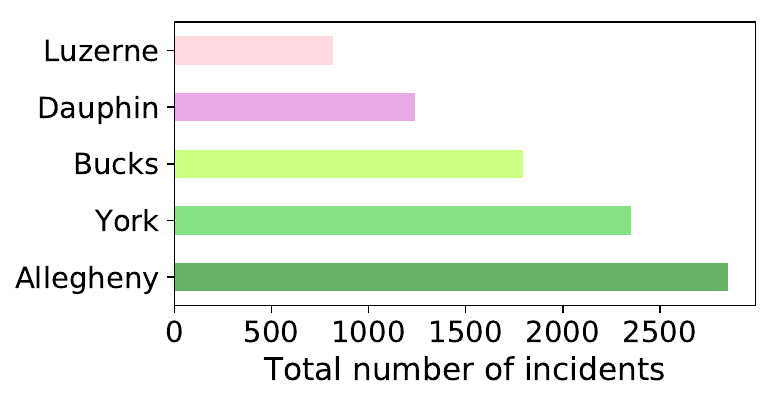}
		\caption{County-level overdose incidents.}
	\end{subfigure}\;\;\;\;
	\begin{subfigure}[ht]{0.5\textwidth}
		\includegraphics[width=1\textwidth]{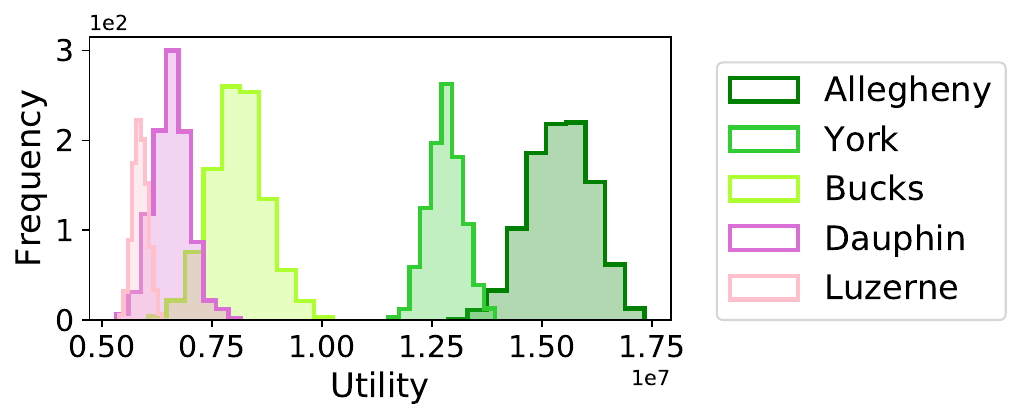}
		\caption{Cumulative utilities achieved by \texttt{S-AC+DPR}.}
	\end{subfigure}
	\caption{Historical overdose incidents by county and the prioritization learned by \texttt{S-AC+DPR}}
	\label{fig:case.hist_sac}
\end{figure}

To investigate how each method prioritizes the different PODs, we show the utilities of each POD generated by each policy in Figure~\ref{fig:case.hist_all}. \texttt{S-AC+DPR} leads to the highest utilities of the first three counties, while \texttt{S-AC+Even} leads to the highest utility of the last county. These two policies perform similarly for the fourth county. Moreover, the two ADP algorithms \texttt{S-AC+DPR} and \texttt{AC+DPR} are in the top three policies for all of the counties' utilities.

\begin{figure}[ht]
	\centering
	\begin{subfigure}[ht]{0.3\textwidth}
		\includegraphics[width=1\textwidth]{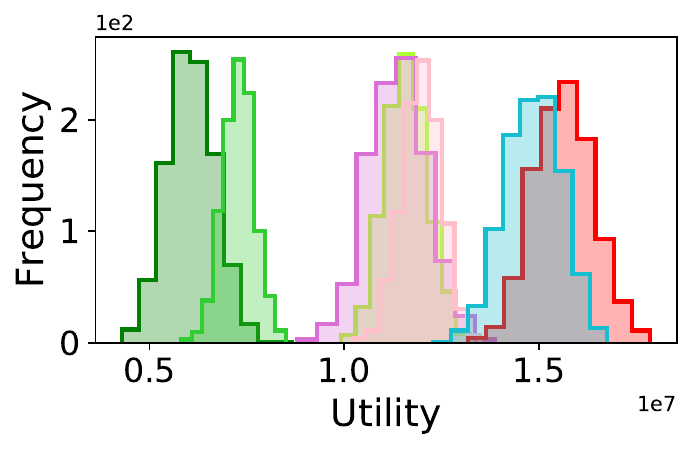}
		\caption{Allegheny}
	\end{subfigure}
	\begin{subfigure}[ht]{0.3\textwidth}
		\includegraphics[width=1\textwidth]{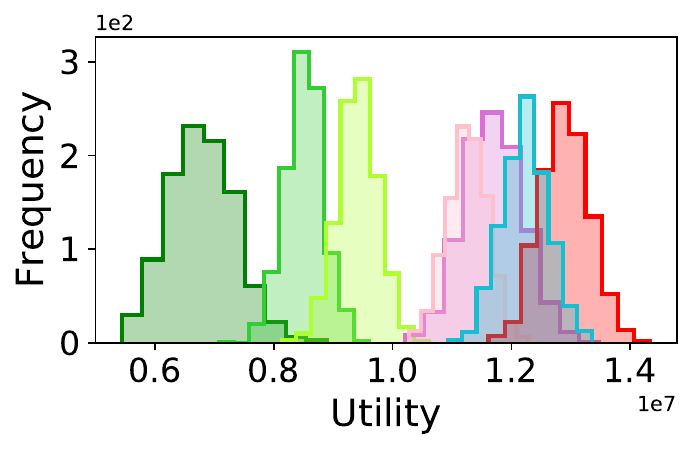}
		\caption{York}
	\end{subfigure}
	\begin{subfigure}[ht]{0.3\textwidth}
		\includegraphics[width=1\textwidth]{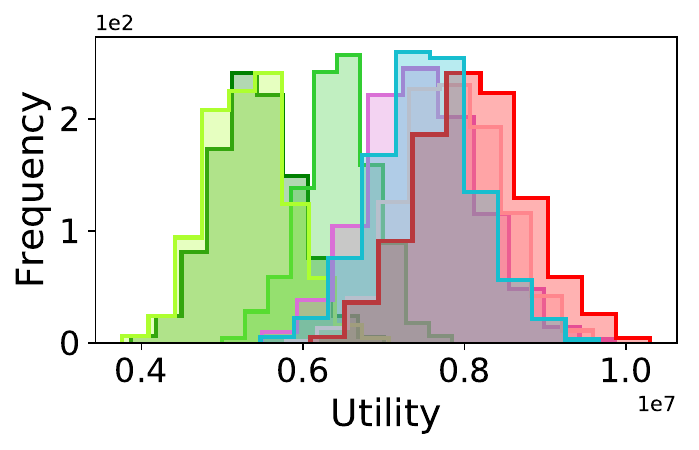}
		\caption{Bucks}
	\end{subfigure}\\
	\begin{subfigure}[ht]{0.3\textwidth}
		\includegraphics[width=1\textwidth]{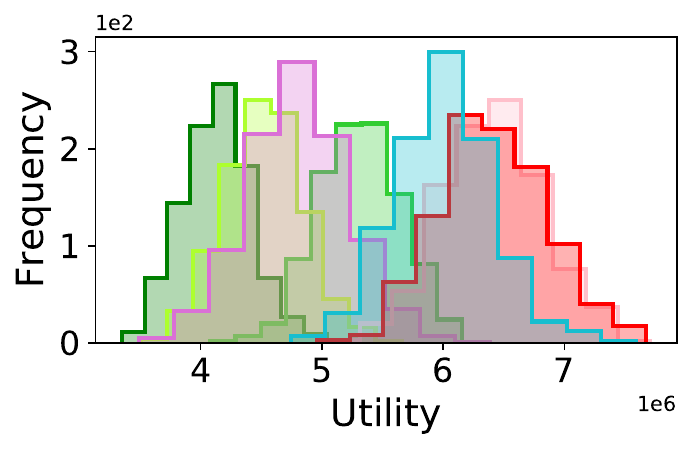}
		\caption{Dauphin}
	\end{subfigure}
	\begin{subfigure}[ht]{0.49\textwidth}
		\includegraphics[width=1\textwidth]{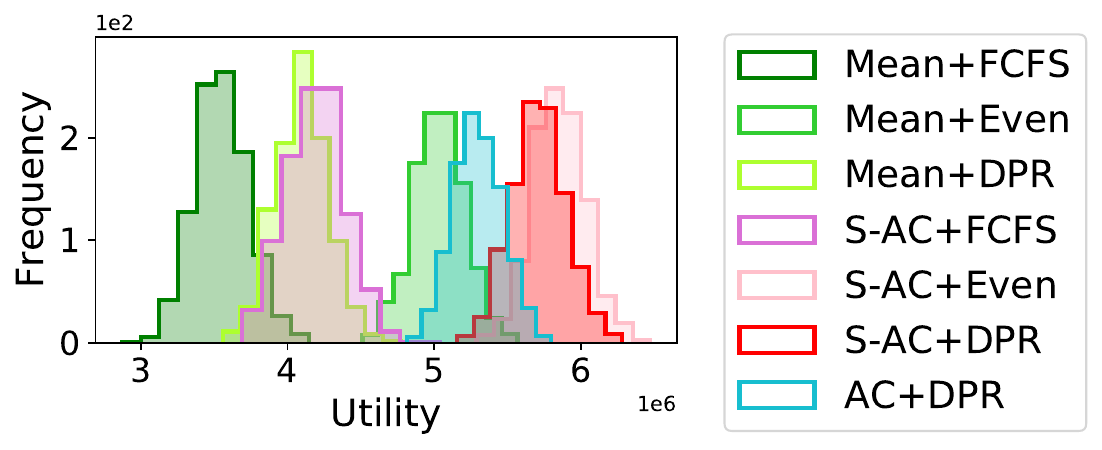}
		\caption{Luzerne}
	\end{subfigure}
	\caption{Comparison of the cumulative utilities for each method.}
	\label{fig:case.hist_all}
\end{figure}

When both levels' policies are heuristics (i.e., \texttt{Mean+FCFS} and \texttt{Mean+Even}), the utility of all PODs are low, with \texttt{Mean+FCFS} leading to the lowest utilities in all cases.
When only the upper-level is a heuristic (i.e., \texttt{Mean+DPR}), the utilities are still not particularly high; in fact, this method ranks in the bottom three policies for all the PODs except for Allegheny.
When the upper-level is \texttt{S-AC} and the lower-level is a heuristic (i.e., \texttt{S-AC+FCFS} and \texttt{S-AC+Even}), the utilities of the top three counties are higher than the utilities by achieved using \texttt{Mean}. The policy \texttt{S-AC+Even} never falls in the bottom three policies;  \texttt{S-AC+FCFS} performs reasonably but is part of the bottom three policies for Dauphin and Luzerne.
In summary, both the upper-level and the lower-level policies play an important in this problem: a properly designed lower-level heuristic can achieve good utility values for some of the PODs; however, intelligent policies on both the upper and lower-levels is necessary to achieve the overall improvement.

\subsubsection{Ordering Cost Sensitivity Analysis} 
Table~\ref{table:case_sensitivity} shows the effect of the ordering cost on the performance (in terms of value achieved) of the various algorithms. The other costs (holding cost and disposal cost) exhibited very minor effects on the value and thus we omitted the results.
Each value in the table is an average of twenty replications of the algorithm, and for each replication of the ADP algorithms, \texttt{S-AC+DPR} and \texttt{AC+DPR}, we take the policy learned by the algorithm at iteration 1,000 and evaluate it by averaging 100 simulations. 
The table shows that \texttt{S-AC+DPR} outperforms the other approaches in all settings. When the ordering cost increases to 5 times (increases from 185 to 925), the value of \texttt{S-AC+DPR} decreases 9.35\%, and when it increases to 20 times (increases from 185 to 3,700), the value decreases 42.48\%. 
These results indicate that the ordering cost of naloxone has a significant influence on the operations of a public health department.

\begin{table}[ht]
\footnotesize
\centering
\caption{Simulated value of the policies on instances with different ordering costs (value in 10 million).}
\label{table:case_sensitivity}
\begin{tabular}{c|ccccccc}
\toprule
Ordering cost & \texttt{Mean+FCFS} & \texttt{Mean+Even} & \texttt{Mean+DPR} & \texttt{S-AC+FCFS} & \texttt{S-AC+Even} & \texttt{S-AC+DPR} & \texttt{AC+DPR} \\\hline
185   & 2.07      & 3.10      & 3.31     & 3.81      & 4.30      & \textbf{4.92} & 4.64 \\
925   & 1.94      & 2.96      & 3.39     & 3.70      & 4.10      & \textbf{4.46} & 4.34 \\
1,850  & 1.75      & 2.77      & 3.04     & 3.12      & 3.56      & \textbf{3.92} & 3.76 \\
3,700  & 1.39      & 2.41      & 2.00     & 2.06      & 2.62      & \textbf{2.83} & 2.56 \\
\bottomrule
\end{tabular}
\end{table}

\subsection{Extensions}
We showed how an aggregation-based version of S-AC along with $k$-means clustering can be used to handle the multi-dimensional continuous features used in the case study. There are also other possible extensions to S-AC that can make it more scalable to high-dimensional problems. For example, shape-constrained deep neural networks \citet{gupta2019incorporate} \citet{gupta2021pender} can handle both monotonicity and concavity via penalization of derivatives during training. In principle, our S-AC algorithm could be extended to use techniques like these, but the same core principles of S-AC would remain intact. We leave these investigations to future work.

\section{Conclusions} \label{sec:conclusions}

In this paper, we formulate a hierarchical MDP model for the sequential problem of optimizing inventory control and making dispensing decisions for a public health organization. We propose a novel, provably convergent actor-critic algorithm that utilizes problem structure in both the policy and value approximations (state-dependent basestock structure for the policy and concavity for the value functions).
Although the algorithm is developed in the setting of our specific MDP, the general paradigm of a structured actor-critic algorithm is likely to be of broader methodological interest. Numerical experiments show that high-quality policies can be obtained in a small number of iterations and that the convergence of the policy is significantly less noisy when compared to competing algorithms. Lastly, we propose an aggregation-based version of our algorithm and provide a case study for the problem of dispensing naloxone to first responders.

\section*{Acknowledgements}
The authors thank Mohamed Kashkoush for invaluable assistance with data collection and analysis, and Hawre Jalal for providing the background of the case study. This research was supported by a Central Research Development Fund grant from the University of Pittsburgh.

\clearpage

\begin{appendices}
\section{Proofs} \label{appendix: proofs}
In this section, we give the proofs of results from the main paper: Proposition~\ref{prop: V_tilde_separate is concave in z}, Lemma~\ref{lemma: convergence of vt}, and Theorem~\ref{thm: v and z converge}.

\subsection{Proof of Proposition~\ref{prop: U is concave in r}} 
\label{appendix: proof property inner}
	
We prove the $L^\natural$-concavity of the $Q$-value function of the lower-level problem by backward induction. Note that if this is true, then the discrete concavity of $U_{w,i}(x,\xi)$ in $x$ follows by Lemma 2 of \citet{chen2014coordinating}. Let $J_{w,i}(x,\xi,y)$ be the $Q$-value for a given state-action pair $(x,\xi,y)$ at period $i$:
\begin{equation*}
    J_{w,i}(x,\xi,y) = u_w\bigl(y,\xi\bigr) + \E_w\bigl[ U_{w,i+1}(X_{i+1},\Xi_{i+1})\bigr].
\end{equation*}
The base case is $J_{w,n+1}(x,\xi,y) = b\,x$, which is $L^\natural$-concave in $(x,y)$. The induction hypothesis is that $J_{w,i+1}(x,\xi,y)$ is $L^\natural$-concave in $(x,y)$.

We analyze $J_{w,i}(x,\xi,y)$ by breaking it up into two terms. The first term is $L^\natural$-concave in $y$ according to Assumption~\ref{assumption: properties of u}.
The second term $U_{w,i+1}(X_{i+1},\Xi_{i+1})$ is $L^\natural$-concave in $X_{i+1}$ according to Lemma 2 of \citet{chen2014coordinating}. 
Since $X_{i+1} = x-y$, $U_{w,i+1}(X_{i+1},\Xi_{i+1})$ is $L^\natural$-concave in $(x,y)$ by Lemma~2 in \citet{zipkin2008structure}. This concludes the proof.

\subsection{Proof of Proposition~\ref{prop: V_tilde_separate is concave in z}} 
\label{appendix: proof property outer_separate}

First, we prove part 1. Let us define the \emph{state-action value function} (or the $Q$-value). The terminal value is defined as $Q_T^\textnormal{rep}(r,w;z^\textnormal{rep}) = -b\,r$. For $t<T$, replenish-up-to decision $z^\textnormal{rep}\in \bar{\mZ}(r)$, and dispense-down-to decision $z^\textnormal{dis}\in \underline{\mZ}(z^\textnormal{rep})$, 
\begin{equation} \label{eq.Q_rep}
	Q_t^\textnormal{rep}(r,w;z^\textnormal{rep}) = (c_w-h)\,r - c_w z^\textnormal{rep} + V_t^\textnormal{dis}(z^\textnormal{rep},w),
\end{equation}
\begin{equation} \label{eq.Q_dis}
	Q_t^\textnormal{dis}(z^\textnormal{rep},w;z^\textnormal{dis}) = \E_w\bigl[ U_{w,0}\bigl(z^\textnormal{rep}-z^\textnormal{dis},\Xi_{t,0}\bigr) + V_{t+1}^\textnormal{rep} (R_{t+1},W_{t+1}) \bigr],
\end{equation}
where $R_{t+1}=z^\textnormal{dis}$. 
We now prove the $L^\natural$-concavity of $Q$-value by backward induction. Note that if this is true, then the $L^\natural$-concavity of $\tilde{V}_t^\textnormal{rep}$ and $\tilde{V}_t^\textnormal{dis}$ follows. 
The base case is $Q_T^\textnormal{rep}(r,w;z^\textnormal{rep}) = -b\,r$, which is $L^\natural$-concave in $(r,z^\textnormal{rep})$, and the induction hypothesis is the same property for $Q_{t+1}^\textnormal{rep}(r,w;z^\textnormal{rep})$.

We first analyze \eqref{eq.Q_dis} by breaking it up into two terms. The first term $\E_w\bigl[ U_{w,0}\bigl(z^\textnormal{rep}-z^\textnormal{dis},\Xi_{t,0}\bigr) \bigr]$ is discretely concave in $(z^\textnormal{rep}, z^\textnormal{dis})$ according to Proposition~\ref{prop: U is concave in r} and Lemma~2 in \citet{zipkin2008structure}. In the second term, $V_{t+1}^\textnormal{rep}(r,w) = \max_{z^\textnormal{rep}\in\bar{\mZ}(r)} Q_{t+1}^\textnormal{rep}(r,w;z^\textnormal{rep})$. Lemma 2 of \citet{chen2014coordinating} shows that $V_{t+1}^\textnormal{rep}(r,w)$ is $L^\natural$ concave in $r$. Since $R_{t+1}=z^\textnormal{dis}$, the term $V_{t+1}^\textnormal{rep}(R_{t+1},W_{t+1})$ is $L^\natural$-concave in $z^\textnormal{dis}$. $L^\natural$-concavity is preserved under expectations, so $Q_t^\textnormal{dis}(z^\textnormal{rep},w;z^\textnormal{dis})$ is $L^\natural$-concave in $(z^\textnormal{rep},z^\textnormal{dis})$. 

Next, we analyze \eqref{eq.Q_rep} by breaking it up into two terms. The first term $(c_w-h)\,r - c_w z^\textnormal{rep}$ is clearly $L^\natural$-concave in $(r,z^\textnormal{rep})$. The second term $V_t^\textnormal{dis}(z^\textnormal{rep},w) = \max_{z^\textnormal{dis}\in \underline{\mZ}(z^\textnormal{rep})} Q_t^\textnormal{dis}(z^\textnormal{rep},w;z^\textnormal{dis})$. Lemma 2 of \citet{chen2014coordinating} shows that $V_t^\textnormal{dis}(z^\textnormal{rep},w)$ is $L^\natural$ concave in $z^\textnormal{rep}$. This concludes Part 1.

\subsection{Proof of Lemma~\ref{lemma: convergence of vt}}
\label{appendix. proof of convergence of vt}

Let us show part (1), the convergence of $\bar{v}_t^{\textnormal{rep},k}(z^{\textnormal{rep}},w)$. The convergence of $\bar{v}_t^{\textnormal{dis},k}(z^{\textnormal{dis}},w)$ in part (2) of the lemma is similar.
Since the demand $D_{t,i}$ is bounded by $D_{\textnormal{max}}$, there exists a $v_\text{max}^{\textnormal{rep}} > 0$ such that $|v_t^{\textnormal{rep}}(z^{\textnormal{rep}},w)| \le v_\text{max}^{\textnormal{rep}}$ for all $t$, $z^{\textnormal{rep}}$, and $w$. We first construct two deterministic sequences $\{G^m\}$ and $\{I^m\}$ such that $G^0 = v^{\textnormal{rep}} + v_\text{max}^{\textnormal{rep}}$ and $I^0 = v^{\textnormal{rep}} - v_\text{max}^{\textnormal{rep}}$ with
\begin{equation} \label{eq.G}
	G^{m+1} = \frac{G^m + v^{\textnormal{rep}}}{2}\quad \text{and} \quad I^{m+1} = \frac{I^m + v^{\textnormal{rep}}}{2} .
\end{equation}
It is easy to show that
\begin{equation} \label{eq.giconv}
	G^m \rightarrow v^{\textnormal{rep}} \quad  \text{and}  \quad I^m \rightarrow v^{\textnormal{rep}}.
\end{equation}
Our goal in this proof is to show that for any $m$ and sufficiently large $k$, 
\begin{equation} \label{eq.IvG}
	I_t^m(z^{\textnormal{rep}},w) \leq \bar{v}_t^{\textnormal{rep},k-1}(z^{\textnormal{rep}},w) \leq G_t^m(z^{\textnormal{rep}},w).
\end{equation} 
If \eqref{eq.IvG} is true, then we can conclude the result of Lemma~\ref{lemma: convergence of vt} by \eqref{eq.giconv}.

We now introduce a random set of states $ \mS_t^- $ that are increased by the projection operator \eqref{eq.concavity-projection} on finitely many iterations $k$. Formally, let
\begin{align*}
    \mS_t^- = \bigl \{(z^{\textnormal{rep}},w) \in \mS: \tilde{v}_t^{\textnormal{rep},k}(z^{\textnormal{rep}},w) < \bar{v}_t^{\textnormal{rep},k}(z^{\textnormal{rep}},w) \; \; \text{finitely often} \bigr\}.
\end{align*}	 
Let $\bar{K}$ be the random variable that describes the iteration number after which states in $\mS_t^-$ are no longer increased by the projection step; i.e., for all $(z^{\textnormal{rep}},w) \in \mS_t^-$, it holds that $\tilde{v}_t^{\textnormal{rep},k}(z^{\textnormal{rep}},w) \ge \bar{v}_t^{\textnormal{rep},k}(z^{\textnormal{rep}},w)$ for all $k \ge \bar{K}$. We break apart \eqref{eq.IvG} into two separate inequalities; this proof will focus on showing that for a fixed $m$, there exists a finite random index $\hat{K}_t^{m}$ such that for all $k \geq \hat{K}_t^{m}$, 
\begin{equation} \label{eq.v-and-G}
	\bar{v}_t^{\textnormal{rep},k-1}(z^{\textnormal{rep}},w) \leq G_t^m(z^{\textnormal{rep}},w).
\end{equation} 
The state space $\mS$ can be partitioned into two parts: (1) states $(z^{\textnormal{rep}},w) \in \mS^-_t$ and (2) states $(z^{\textnormal{rep}},w)\in \mS \setminus \mS_t^-$. The proof of \eqref{eq.v-and-G} will consider each partition separately. We now define some noise terms and stochastic sequences. Recall from \eqref{eq.V-and-f-rep} and \eqref{eq.v_rep_hat} that $\hat{v}_t^{\textnormal{rep},k} = \hat{V}_t^{\textnormal{rep},k}(z_t^{\textnormal{rep},k}, w_t^k) - \hat{V}_t^{\textnormal{rep},k}(z_t^{\textnormal{rep},k}-1, w_t^k)$, where
\begin{equation*}
	\begin{aligned}
	\hat{V}_t^{\textnormal{rep},k}(z^\textnormal{rep},w_t^k) =& -c_{w_t^k} z^\textnormal{rep} + U_{w_t^k,0}^{\mu^*} \bigl( z^\textnormal{rep} - \tilde{\pi}_t^{\textnormal{dis},k-1}(z^\textnormal{rep},w_t^k), \check{\xi}_{t,0}^k\bigr) \\&+ f_{t}^{\textnormal{dis}} \bigl(\tilde{\pi}^{\textnormal{rep},k-1}, \tilde{\pi}^{\textnormal{dis},k-1}; \mathbf{Z}^k_{t}(w_{t}), z^\textnormal{rep}\bigr),
	\end{aligned}
\end{equation*}
By our assumption that $\overbar{l}_{\tau}^{\textnormal{rep},k}(w) \rightarrow l_\tau^{\textnormal{rep}}(w)$ almost surely for $\tau \geq t+1$, and $\overbar{l}_{\tau}^{\textnormal{dis},k}(w) \rightarrow l_\tau^{\textnormal{dis}}(w)$ almost surely for $\tau \geq t$, and the fact that $f_{t}^{\textnormal{dis}} \bigl(\tilde{\pi}^{\textnormal{rep},k-1}, \tilde{\pi}^{\textnormal{dis},k-1}; \mathbf{Z}^k_{t}(w_{t}), z^\textnormal{rep}\bigr)$ depends only on the replenish-up-to thresholds for periods $t+1$ onward and the dispense-down-to thresholds for periods $t$ onward, it follows that the simulated value of $\tilde{\pi}^{\textnormal{rep},k-1}$ and $\tilde{\pi}^{\textnormal{dis},k-1}$ becomes unbiased asymptotically:
\begin{equation} \label{eq.Ef-converge}
	\mathbf{E}_w \bigl[f_{t}^{\textnormal{dis}} \bigl(\tilde{\pi}^{\textnormal{rep},k-1}, \tilde{\pi}^{\textnormal{dis},k-1}; \mathbf{Z}^k_{t}(w), z^\textnormal{rep}\bigr)\bigr] \rightarrow \tilde{V}_{t}^{\textnormal{dis}}\bigl(\pi_{t}^{\textnormal{dis},*}(z^\textnormal{rep},w),w\bigr) \quad \text{a.s.}
\end{equation}
We define the noise term $\epsilon_t^{k}(z_t^{\textnormal{rep},k},w_t^k)$ such that
\begin{align} \label{eqilon1}
	\epsilon_t^{k}(z_t^{\textnormal{rep},k},w_t^k) = \E \bigl[\hat{v}_t^{\textnormal{rep},k}\bigr] - v_t^{\textnormal{rep}}(z_t^{\textnormal{rep},k}, w_t^k).
\end{align}
Note that we can conclude from \eqref{eq.Ef-converge} that $\epsilon_t^{k}(z_t^{\textnormal{rep},k},w_t^k) \rightarrow 0$ almost surely. We define another noise term $\varepsilon_t^{k}(z_t^{\textnormal{rep},k}, w_t^k)$ such that $\varepsilon_t^{k}(z_t^{\textnormal{rep},k}, w_t^k) = \hat{v}_t^{\textnormal{rep},k} - \E\bigl[\hat{v}_t^{\textnormal{rep},k}\bigr]$. Thus, we can see that 
\begin{equation} \label{eq.vhat noise}
	\hat{v}_t^{\textnormal{rep},k} = v_t^{\textnormal{rep}}(z_t^{\textnormal{rep},k}, w_t^k) + \epsilon_t^{k}(z_t^{\textnormal{rep},k},w_t^k) + \varepsilon_t^{k}(z_t^{\textnormal{rep},k}, w_t^k)
\end{equation}
Next, we need to define some stochastic sequences related to these noise terms. Let $\{\bar{s}_t^k\}$ be defined such that for $k<\bar{K}$, $\bar{s}_t^k(z^{\textnormal{rep}},w) = 0$, and for $k \geq \bar{K}$,
\begin{equation} \label{eq.s_bar}
	\bar{s}_t^k(z^{\textnormal{rep}},w) = \bigl(1-\alpha_t^k(z^{\textnormal{rep}},w)\bigr) \, \bar{s}_t^{k-1}(z^{\textnormal{rep}},w) + \alpha_t^k(z^{\textnormal{rep}},w) \bigl[ \epsilon_t^k(z_t^{\textnormal{rep},k},w_t^k) + \varepsilon_t^k(z_t^{\textnormal{rep},k}, w_t^k) \bigr]. 
\end{equation}
This sequence averages both of the noise terms. Since $\epsilon_t^k$ is unbiased and $\varepsilon_t^k$ converges to zero, we can apply Theorem 2.4 of \cite{kushner2003stochastic}, a standard stochastic approximation convergence result, to conclude that $\bar{s}_t^k(z^{\textnormal{rep}},w) \rightarrow 0$ almost surely. We then define a stochastic bounding sequence $\{\bar{g}_t\}$ such that for $k<\bar{K}$, $\bar{g}_t^k(z^{\textnormal{rep}},w) = G_t^k(z^{\textnormal{rep}},w)$ and for $k\geq\bar{K}$,
\begin{equation} \label{eq.g_bar}
	\bar{g}_t^k(z^{\textnormal{rep}},w) = \bigl(1-\alpha_t^k(z^{\textnormal{rep}},w)\bigr) \, \bar{g}_t^{k-1}(z^{\textnormal{rep}},w) + \alpha_t^k(z^{\textnormal{rep}},w) \,v_t^{\textnormal{rep}}(z^{\textnormal{rep}},w). 
\end{equation}

\paragraph{Part (1). }
As in \cite{nascimento2009optimal}, we provide an $\omega$-wise argument, meaning that we consider a fixed $\omega \in \Omega$ (although the dependence of random variables on $\omega$ is omitted for notational simplicity). Here, we show the existence of a finite index $\tilde{K}_t^m$ such that for all states $(z^{\textnormal{rep}},w) \in \mS^-_t$, it holds that for all iterations $k \geq \tilde{K}_t^m$, $\bar{v}_t^{\textnormal{rep},k-1}(z^{\textnormal{rep}},w) \leq G_t^m(z^{\textnormal{rep}},w)$. The proof is a forward induction on $m$ where the base case is $m=0$. The base case can be easily proved by applying the definition of $ G^0 $ (note that we can select $\tilde{K}_t^m \ge \bar{K}$. The induction hypothesis is that there exists an integer $ \tilde{K}_t^m \ge \bar{K}$ such that for all $ k \geq \tilde{K}_t^m $, the inequality \eqref{eq.v-and-G} is true. The next step is $ m+1 $: we must show the existence of an integer $\tilde{K}_t^{m+1} \ge \bar{K}$ such that for all states $(z^{\textnormal{rep}},w) \in \mS^-_t$, it holds that
\begin{equation} \label{eq.vbar and G}
	\bar{v}_t^{\textnormal{rep},k-1}(z^{\textnormal{rep}},w) \leq G_t^{m+1}(z^{\textnormal{rep}},w)  
\end{equation}
for all iterations $k \geq \tilde{K}_t^{m+1}$. We require the following lemma.

\begin{lemma}\label{lem:vgs}
The inequality
\begin{equation} \label{eq.vgs}
	\bar{v}_t^{\textnormal{rep},k-1}(z^{\textnormal{rep}},w) \leq \bar{g}_t^{k-1}(z^{\textnormal{rep}},w) + \bar{s}_t^{k-1}(z^{\textnormal{rep}},w)
\end{equation}
holds almost everywhere on $\{k \geq \tilde{K}_t^m,\; (z^{\textnormal{rep}},w)\in \mS^-_t\}$.
\end{lemma}

\begin{proof}
When $k = \tilde{K}_t^m$, the relationship \eqref{eq.vgs} can be shown using the definitions of $\bar{g}_t^{k-1}(z^{\textnormal{rep}},w)$ and $\bar{s}_t^{k-1}(z^{\textnormal{rep}},w)$, along with the induction hypothesis \eqref{eq.v-and-G}. We now induct on $k$. Suppose that \eqref{eq.vgs} is true for a given $k \geq \tilde{K}_t^m$. The inductive step is to show $ \bar{v}_t^{\textnormal{rep},k}(z^{\textnormal{rep}},w) \leq \bar{g}_t^k(z^{\textnormal{rep}},w) + \bar{s}_t^k(z^{\textnormal{rep}},w) $. To simplify notation, let $\check{\alpha}_t^k$, $\check{v}_t^k$, $\check{s}_t^k$, and $\check{g}_t^k$ respectively denote $\alpha_t^k(z^{\textnormal{rep}},w)$, $\bar{v}_t^{\textnormal{rep},k}(z^{\textnormal{rep}},w)$, $\bar{s}_t^k(z^{\textnormal{rep}},w)$ and $\bar{g}_t^k(z^{\textnormal{rep}},w)$. For state $(z^{\textnormal{rep}},w) = (z_t^{\textnormal{rep},k},w_t^k)$, we have
\begin{align}
	\check{v}_t^k  &= \tilde{v}_t^{\textnormal{rep},k}(z^{\textnormal{rep}},w) = (1-\check{\alpha}_t^k) \check{v}_t^{k-1} + \check{\alpha}_t^k \hat{v}_t^{\textnormal{rep},k} \nonumber \\
	& \leq (1-\check{\alpha}_t^k) \bigl( \check{g}_t^{k-1} + \check{s}_t^{k-1} \bigr) + \check{\alpha}_t^k \hat{v}_t^{\textnormal{rep},k} - \check{\alpha}_t^k v_t^{\textnormal{rep}}(z_t^{\textnormal{rep},k},w_t^k) + \check{\alpha}_t^k v_t^{\textnormal{rep}}(z_t^{\textnormal{rep},k},w_t^k) \nonumber\\
	& = (1-\check{\alpha}_t^k) \bigl( \check{g}_t^{k-1} + \check{s}_t^{k-1} \bigr) + \check{\alpha}_t^k \bigl[ \epsilon_t^k(z_t^{\textnormal{rep},k},w_t^k) + \varepsilon_t^k(z_t^{\textnormal{rep},k}, w_t^k) \bigr] + \check{\alpha}_t^k v_t^{\textnormal{rep}}(z_t^{\textnormal{rep},k},w_t^k) \nonumber\\
	& = (1-\check{\alpha}_t^k)\check{g}_t^{k-1} + \check{s}_t^k + \check{\alpha}_t^k v_t^{\textnormal{rep}}(z_t^{\textnormal{rep},k},w_t^k) \nonumber \\
	& = \check{g}_t^k + \check{s}_t^k \nonumber.
\end{align}
The first equality is due to the fact that $(z^{\textnormal{rep}},w) = (z_t^{\textnormal{rep},k}, w_t^k)$, which is unaltered by the projection operator \eqref{eq.concavity-projection}. The second inequality follows from the induction hypothesis \eqref{eq.vgs}. The last three steps follow by \eqref{eq.vhat noise}, \eqref{eq.s_bar} and \eqref{eq.g_bar} respectively. 

For $(z^{\textnormal{rep}},w) \neq (z_t^{\textnormal{rep},k},w_t^k)$, which are the states that are not updated by a direct observation of the sample slope at iteration $k$, period $t$, the stepsize $\check{\alpha}_t^k = 0$. Then, we have
\begin{equation*}
	\check{s}_t^k = \check{s}_t^{k-1} \quad \text{and } \quad \check{g}_t^k = \check{g}_t^{k-1}. 
\end{equation*}
Therefore, from the definition of set $\mS^-_t $, the fact that $\tilde{K}_t^m \ge \bar{K}$, and the induction hypothesis, we have
\begin{equation*}
	\check{v}_t^k \leq \tilde{v}_t^{\textnormal{rep},k}(z^{\textnormal{rep}},w) = \check{v}_t^{k-1} \leq \check{g}_t^{k-1} + \check{s}_t^{k-1} = \check{g}_t^k + \check{s}_t^k,
\end{equation*}
which concludes the proof of \eqref{eq.vgs}. 
\end{proof}

Since $G^m \geq G^{m+1} \geq v^{\textnormal{rep}}$ for all $m$, when $G_t^m(z^{\textnormal{rep}},w) = v_t^{\textnormal{rep}}(z^{\textnormal{rep}},w) = G_t^{m+1}(z^{\textnormal{rep}},w)$, the inequality $\bar{v}_t^{\textnormal{rep},k-1}(z^{\textnormal{rep}},w) \leq G_t^m(z^{\textnormal{rep}},w)$ implies that $\bar{v}_t^{\textnormal{rep},k-1}(z^{\textnormal{rep}},w) \leq G_t^{m+1}(z^{\textnormal{rep}},w)$. Thus, the only remaining states to consider are the ones where $G_t^m(z^{\textnormal{rep}},w) > v_t^{\textnormal{rep}}(z^{\textnormal{rep}},w)$. Let $\delta^m$ be the minimum of the quantity $[ G_t^k(z^{\textnormal{rep}},w) - v_t^{\textnormal{rep}}(z^{\textnormal{rep}},w) ] / 4$ over states $(z^{\textnormal{rep}},w) \in \mS_t^-$ with $G_t^m(z^{\textnormal{rep}},w) > v_t^{\textnormal{rep}}(z^{\textnormal{rep}},w)$. Define an integer $K^G \geq \tilde{K}_t^m$ such that for all states $(z^{\textnormal{rep}},w) \in \mS^-_t$,
\begin{equation*}
	\textstyle \prod_{k=\tilde{K}_t^m}^{K^G-1} \bigl( 1 - \alpha_t^k(z^{\textnormal{rep}},w) \bigr) \leq 1/4 \quad \text{and} \quad \bar{s}_t^k(z^{\textnormal{rep}},w) \leq \delta^m.
\end{equation*}
for every iteration $k \geq K^G$. We can find such a $K^G$ because the stepsize conditions of Assumption~\ref{assumption: stepsize} imply that 
\begin{equation*}
	\textstyle \prod_{k=\tilde{K}_t^m}^{\infty} \bigl( 1 - \alpha_t^k(z^{\textnormal{rep}},w) \bigr) = 0, 
\end{equation*}
and because $\bar{s}_t^k(z^{\textnormal{rep}},w)$ converges to zero.

Now we are ready to show \eqref{eq.vbar and G}. 
The definition of the sequence $\{\bar{g}_t^k\}$ implies that $\bar{g}_t^k(z^{\textnormal{rep}},w)$ is a convex combination of $G_t^k(z^{\textnormal{rep}},w)$ and $v_t^{\textnormal{rep}}(z^{\textnormal{rep}},w)$, of the form
\begin{equation*}
	\bar{g}_t^k(z^{\textnormal{rep}},w) = \hat{\alpha}_t^{k}(z^{\textnormal{rep}},w) \, G_t^k(z^{\textnormal{rep}},w) + \bigl( 1 - \hat{\alpha}_t^{k}(z^{\textnormal{rep}},w) \bigr) \, v_t^{\textnormal{rep}}(z^{\textnormal{rep}},w),
\end{equation*}
\looseness-1 where $\hat{\alpha}_t^{k}(z^{\textnormal{rep}},w) = \prod_{k=\tilde{K}_t^m}^{K-1} \bigl(1 - \alpha_t^k(z^{\textnormal{rep}},w)\bigr) \leq 1/4$ for $k \geq K^G$. Because $G^m \geq v^{\textnormal{rep}}$ for any $m$, it follows that 
\begin{align*}
	\bar{g}_t^k(z^{\textnormal{rep}},w) & \leq \textstyle \frac{1}{4} \, G_t^k(z^{\textnormal{rep}},w) + \frac{3}{4} \, v_t^{\textnormal{rep}}(z^{\textnormal{rep}},w) \\
	& = \textstyle \frac{1}{2} \, G_t^k(z^{\textnormal{rep}},w) + \frac{1}{2} \, v_t^{\textnormal{rep}}(z^{\textnormal{rep}},w) - \frac{1}{4} \, \bigl(G_t^k(z^{\textnormal{rep}},w) - v_t^{\textnormal{rep}}(z^{\textnormal{rep}},w)\bigr) \\
	& \leq G_t^{k+1}(z^{\textnormal{rep}},w) - \delta^m, 
\end{align*}
where the second inequality follows from \eqref{eq.G} and the definition of $\delta^m$. Recall that we are concentrating on the case where $G_t^m(z^{\textnormal{rep}},w) > v_t^{\textnormal{rep}}(z^{\textnormal{rep}},w)$, so $\delta^m$ is well-defined and positive. This inequality, together with Lemma~\ref{lem:vgs} and $\bar{s}_t^k(z^{\textnormal{rep}},w) \leq \delta^m$, imply that for all $k \geq K^G$, 
\begin{align*}
	\bar{g}_t^k(z^{\textnormal{rep}},w) \leq G_t^{k+1}(z^{\textnormal{rep}},w) - \delta^m + \bar{s}_t^k(z^{\textnormal{rep}},w) 
	\leq G_t^{k+1}(z^{\textnormal{rep}},w) - \delta^m + \delta^m 
	\leq G_t^{k+1}(z^{\textnormal{rep}},w). 
\end{align*}
We conclude Part (1) of the proof by letting $\tilde{K}_t^{m+1} = K^G$.

\paragraph{Part (2). } We now focus on the states $(z^{\textnormal{rep}},w) \in \mS \setminus \mS^-_t$ that are increased infinitely often. For a fixed $m$ and state $(z^{\textnormal{rep}},w) \in \mS \setminus \mS^-_t$, we wish to prove the existence of a random index $\hat{K}_t^{m}(z^{\textnormal{rep}},w)$ such that for all $k \ge \hat{K}_t^{m}(z^{\textnormal{rep}},w)$, it holds that $\bar{v}_t^{\textnormal{rep},k-1}(z^{\textnormal{rep}},w) \leq G_t^m(z^{\textnormal{rep}},w)$. Note that $\hat{K}_t^m(z^{\textnormal{rep}},w)$ differs from $\tilde{K}_t^m$ in that it depends on a specific $(z^{\textnormal{rep}},w) \in \mS \setminus \mS^-_t$ (while we $\tilde{K}_t^m$ is chosen uniformly for all states in $\mS_t^-$). The crux of the proof depends on the following lemma.

\begin{lemma} \label{lemma. G}
	Fix $m\geq 0$ and consider a state $(z^{\textnormal{rep}}-1,w)\in \mS \setminus \mS_t^-$ and suppose that there exists a random index $\hat{K}_t^{m}(z^{\textnormal{rep}},w)$ such that the required condition $\bar{v}_t^{\textnormal{rep},k-1}(z^{\textnormal{rep}},w) \leq G_t^m(z^{\textnormal{rep}},w)$ is true, then there exists another random index $\hat{K}_t^{m}(z^{\textnormal{rep}}-1,w)$ such that $\bar{v}_t^{\textnormal{rep},k-1}(z^{\textnormal{rep}}-1,w) \leq G_t^m(z^{\textnormal{rep}}-1,w)$ for all iterations $k \geq \hat{K}_t^{m}(z^{\textnormal{rep}}-1,w)$.
\end{lemma}
\begin{proof}
See the proof of Lemma 6.4 of \cite{nascimento2009optimal}. The only modification that needs to be made is to redefine the Bellman operator `$H$' from \cite{nascimento2009optimal} so that it maps to the optimal value function slopes $v$ for any argument (we no longer interpret $H$ as a Bellman operator as our algorithm is not based on value iteration). 
\end{proof}

Consider some $m\geq 0$ and a state $(z^{\textnormal{rep}},w)\in \mS \setminus \mS_t^-$. Now, let state $(z_\textnormal{min}^{\textnormal{rep}},w)$ where $z_\textnormal{min}^{\textnormal{rep}}$ is the minimum replenish-up-to postdecision resource level such that $z_\textnormal{min}^{\textnormal{rep}} > z^{\textnormal{rep}}$ and $(z_\textnormal{min}^{\textnormal{rep}},w) \in \mS^-_t$. We note that such a state certainly exists because $(R_\textnormal{max} ,w) \in \mS_t^-$. The state $(z_\textnormal{min}^{\textnormal{rep}},w)$ satisfies the condition of Lemma~\ref{lemma. G} with $ \hat{K}_t^{m}(z_\textnormal{min}^{\textnormal{rep}},w) = K_t^m $, so we may conclude that there is an index $\hat{K}_t^{m}(z_\textnormal{min}^{\textnormal{rep}} - 1, w)$ associated with state $(z_\textnormal{min}^{\textnormal{rep}} - 1, w)$ such that for all $k \geq \hat{K}_t^{m}(z_\textnormal{min}^{\textnormal{rep}} - 1, w)$, the required condition $ \bar{v}_t^{\textnormal{rep},k-1}(z_\textnormal{min}^{\textnormal{rep}} - 1,w) \leq G_t^m(z_\textnormal{min}^{\textnormal{rep}} - 1,w)$ holds. This process can be repeated until we reach the state of interest $(z^{\textnormal{rep}},w)$, which provides the required $\hat{K}_t^{m}(z^{\textnormal{rep}},w)$. Finally, if we choose an iteration large enough, i.e., 
\begin{equation*}
	K_t^{m} =\textstyle \max\bigl\{ \tilde{K}_t^m, \max_{(z^{\textnormal{rep}},w) \,\in\, \mS \setminus \mS_t^-} \hat{K}_t^{m}(z^{\textnormal{rep}},w) \bigr\},
\end{equation*}
then \eqref{eq.v-and-G} is true for all $k\geq \hat{K}_t^{m}$ and states $(z^{\textnormal{rep}},w) \in \mS$. A symmetric proof can be given to verify that the other half of the inequality \eqref{eq.IvG}, $\bar{v}_t^{\textnormal{rep},k-1}(z^{\textnormal{rep}},w) \geq I_t^m(z^{\textnormal{rep}},w)$, holds for sufficiently large $k$, which completes the proof.

\subsection{Proof of Theorem~\ref{thm: v and z converge}}
\label{appendix: proof theorem}
The proof of Theorem~\ref{thm: v and z converge} is a backward induction over time periods $t$. For the replenish-up-to value function and threshold, the base case is $t=T$, where the convergence of $\bar{v}_T^{\textnormal{rep},k}(z^{\textnormal{rep}},w)$ and $\bar{l}_T^{\textnormal{rep},k}(w)$ to their optimal counterparts (both equal to zero) are trivial by assumption (see Section~\ref{sec:convanalysis}). The induction hypothesis is that $\overbar{l}_{\tau}^{\textnormal{rep},k}(w) \rightarrow l_\tau^{\textnormal{rep}}(w)$ almost surely for $\tau \geq t+1$, and $\bar{\pi}_{\tau}^{\textnormal{dis},k}(z^{\textnormal{rep}},w) \rightarrow \pi_\tau^{\textnormal{dis}}(z^{\textnormal{rep}},w)$ almost surely for $\tau \geq t$. Now, consider period $t$. The almost sure convergence of $\bar{v}_t^{\textnormal{rep},k}(z^{\textnormal{rep}},w)$ to $v_t^{\textnormal{rep}}(z^{\textnormal{rep}},w)$ follows by Lemma~\ref{lemma: convergence of vt}. Therefore, by Assumption~\ref{assumption:unique basestock}, we can conclude that
\begin{align*}
	\hat{l}_t^{\textnormal{rep},k} = \textstyle\argmax_{z^{\textnormal{rep}} \in \mZ(0)} \sum_{j=0}^{z^{\textnormal{rep}}} \bar{v}_t^{\textnormal{rep},k}(j,w_t^k) \rightarrow l_t^{\textnormal{rep}}(w) \quad \text{a.s.}
\end{align*}
Combining this with the update formula for $\bar{l}_t^{\textnormal{rep},k}(w)$, the stepsize properties of Assumption~\ref{assumption: stepsize}, and Theorem 2.4 of \citet{kushner2003stochastic}, we see that $\bar{l}_t^{\textnormal{rep},k}(w)$ converges to $l_t^{\textnormal{rep}}(w)$ almost surely.

For the dispense-down-to value function and policy, the proof is similar. We only need to notice that the dispense-down-to decision is made after the replenish-up-to decision, and the induction hypothesis for it is that $\overbar{l}_{\tau}^{\textnormal{rep},k}(w) \rightarrow l_\tau^{\textnormal{rep}}(w)$ and $\bar{\pi}_{\tau}^{\textnormal{dis},k}(z^{\textnormal{rep}},w) \rightarrow \pi_\tau^{\textnormal{dis}}(z^{\textnormal{rep}},w)$ almost surely for $\tau \geq t+1$.

\clearpage

\section{Actor-Critic Method} \label{sec:alg_ac}
The actor-critic method is shown in Algorithm~\ref{alg:ac}.

\IncMargin{1em}
\begin{algorithm}
    \small
	\SetKwInput{Input}{Input}\SetKwInput{Output}{Output}
	\Input{RBFs $\boldsymbol{\psi}(r,w)$ for the state value, and $\boldsymbol{\phi}(r,w; z^{\textnormal{rep}}, z^{\textnormal{dis}})$ for the policy.\\
		Initial parameter estimate $\boldsymbol{\eta}^0$ and $\boldsymbol{\theta}^0$. \\
		Stepsize rules $\tilde{\alpha}_t^k$ and $\tilde{\beta}_t^k$ for all $t,k$.}
	\medskip
	\Output{Parameters $\boldsymbol{\eta}^k$ and $\boldsymbol{\theta}^k$.}
	\BlankLine
	\For{$k=1,2,\ldots,K$}{
		\vspace{0.5em}
		Sample an initial state $s_0^k$.\\
		\vspace{0.5em}
		\For{$t=0,1,\ldots,T-1$}{
			\vspace{0.5em}
			Observe $w_t^k$ and $\xi_{t,1}^k$.\\
			\vspace{0.5em}
			Take action $ \bigl(z_t^{k,\textnormal{rep}}, z_t^{k,\textnormal{dis}}\bigr) \sim \pi_t^{k-1}(z^{\textnormal{rep}}, z^{\textnormal{dis}}| r_t^k,w_t^k; \boldsymbol{\theta}^{k-1}) $, observe the next state $(r_{t+1}^k,w_{t+1}^k)$ and the immediate reward $C_t = (c_{w_t^k} - h)\,r - c_{w_t^k} z_t^{k,\textnormal{rep}} + U_{w_t^k,0} \bigl(z_t^{k,\textnormal{rep}} - z_t^{k,\textnormal{dis}}, \xi_{t,0}^k\bigr)$. \\
			\vspace{0.5em}
			Calculate the temperal difference $ \delta_t \leftarrow C_t + \boldsymbol{\psi}(r_{t+1}^k,w_{t+1}^k)^T \boldsymbol{\eta}_{t+1}^k - \boldsymbol{\psi}(r_t^k,w_t^k)^T \boldsymbol{\eta}_t^k $. \\
			\vspace{0.5em}
			Critic update: $ \boldsymbol{\eta}_t^k = \boldsymbol{\eta}_t^{k-1} + \alpha_t^k(r,w) \delta_t \boldsymbol{\psi}(r_t^k,w_t^k) $, where $ \alpha_t^k(r,w) = \tilde{\alpha}_t^k \1\{(r,w) = (r_t^k,w_t^k)\} $.\\
			\vspace{0.5em}
			Actor update: $ \boldsymbol{\theta}_t^k = \boldsymbol{\theta}_t^{k-1} + \beta_t^k(r,w; z^{\textnormal{rep}}, z^{\textnormal{dis}}) \delta_t \Delta_{\boldsymbol{\theta}_t^{k-1}} \ln \dot{\pi}_t^{k-1}(z^{\textnormal{rep}}, z^{\textnormal{dis}}| r_t^k,w_t^k; \boldsymbol{\theta}^{k-1}) $, where $ \beta_t^k(r,w; z^{\textnormal{rep}}, z^{\textnormal{dis}}) = \tilde{\beta}_t^k \1\{(r,w; z^{\textnormal{rep}}, z^{\textnormal{dis}}) = (r_t^k,w_t^k; z_t^{k,\textnormal{rep}}, z_t^{k,\textnormal{dis}})\} $. \\
		}
	}
	\caption{Actor-Critic Method}
	\label{alg:ac}
\end{algorithm}\DecMargin{1em}

\clearpage

\section{A Practical, Aggregation-based Version of S-AC} \label{sec:aggregation}

To deal with potentially continuous information states $W_t \in \mathcal W$, we now introduce a practical version of our algorithm that utilizes aggregation in the information state. The essential idea is that the structural results from Section~\ref{sec:structural} continue to hold when we perform aggregation, so the S-AC idea can be applied almost directly. We partition the exogenous information space $\mW$ into $J$ sets, i.e., let 
\begin{equation*}
    \mW = \mW_1 \cup \mW_2 \cup \ldots \cup \mW_J \quad \text{with} \quad \mW_i \cap \mW_j = \emptyset \quad \text{if} \quad i \neq j.
\end{equation*}
Note that we do not aggregate in the inventory state and only do so in the information state. Each partition $\mathcal W_j$ contains a representative state, denoted $\dot{w}_j \in \mathcal W_j$, similar to what is done in \cite{tsitsiklis1996feature}. We also assign a distribution over each partition and we suppose that the distribution is described with a density function $p^j(w)$, with $w \in \mathcal W_j$. This allows us to map the original MDP to an aggregate version by integrating with respect to this distribution (which should be thought of as a design choice). For the remainder of the paper, we assume that $p^j(\cdot)$ is a uniform density function, but remark that the algorithm can easily accommodate other aggregation distributions by including a likelihood ratio factor.

We use ``dot'' notation to denote variables related to state aggregation. For example, $\dot{W}_t$ denotes the aggregate exogenous information at period $t$. Further, let $\dot{V}_t^{\textnormal{rep}}(r,\dot{w}_j)$ and $\dot{V}_t^{\textnormal{dis}}(z^{\textnormal{rep}},\dot{w}_j)$ respectively denote the \textit{optimal aggregate value functions} for the replenish-up-to decision and the dispense-down-to decision, let $\dot{\tilde{V}}_t^{\textnormal{rep}}(z,\dot{w}_j)$ and $\dot{\tilde{V}}_t^{\textnormal{dis}}(z^{\textnormal{rep}},\dot{w}_j)$ respectively denote their corresponding \textit{aggregate postdecision value function}, let $\dot{\tilde{\pi}}^{\textnormal{rep}}$ and $\dot{\tilde{\pi}}^{\textnormal{dis}}$ be the rounded policies under state aggregation. 
The terminal aggregate value function is $\dot{V}_T^{\textnormal{rep}}(r,\dot{w}_j) = -b\,r$ and for $t<T$, we have 
\begin{equation*}
	\dot{V}_t^{\textnormal{rep}}(r,\dot{w}_j) = \max_{z^{\textnormal{rep}} \in \bar{\mZ}(r)} \mathlarger{\int}_{w\in\mW_j} p^j(w)\, \bigl\{ (c_w-h) r - c_w z^{\textnormal{rep}} + \dot{V}_t^{\textnormal{dis}}(z^{\textnormal{rep}},\dot{w}_j) \bigr\} dw,
\end{equation*}
\begin{equation*}
	\dot{V}_t^{\textnormal{dis}}(z^{\textnormal{rep}},\dot{w}_j) = \max_{z^{\textnormal{dis}} \in \underline{\mZ}(z^\textnormal{rep})} \mathlarger{\int}_{w\in\mW_j} p^j(w)\, \bigl\{ \E_w\bigl[ U_{w,0}^{\bar{\mu}}(z^\textnormal{rep}-z^\textnormal{dis},\Xi_{t,0}) + \dot{V}_{t+1}^\textnormal{rep} \bigl(z^\textnormal{dis}, \dot{W}_{t+1}\bigr) \bigr] \bigr\} dw,
\end{equation*}
where the transition to $\dot{W}_{t+1}$ satisfies $\dot{W}_{t+1} = \sum_{j=1}^k \dot{W}_j \1\{W_{t+1}\in\dot{W}_j\}$, and $\bar{\mu}$ is the approximate policy for the lower-level. 
For the lower-level dispensing problem, similar the the discrete state space version, we solve the optimal policy $\mu^*$ for each aggregate state. Then the policy is extrapolated to the continuous state space by linear regression. 
Similar to the definition of postdecision value functions \eqref{eq.postdec_rep} and \eqref{eq.postdec_dis}, define 
\begin{equation*}
	\dot{\tilde{V}}_t^{\textnormal{rep}}(z,\dot{w}_j) = \int_{w\in\mW_j} p^j(w) \, \bigl\{ - c_w z^{\textnormal{rep}} + \dot{V}_t^{\textnormal{dis}}(z^{\textnormal{rep}},\dot{w}_j) \bigr\} dw,
\end{equation*}
\begin{equation*}
	\dot{\tilde{V}}_t^{\textnormal{dis}}(z^{\textnormal{rep}},\dot{w}_j) = \int_{w\in\mW_j} p^j(w) \,\E_w\bigl[\dot{V}_{t+1}^\textnormal{rep} \bigl(z^\textnormal{dis}, \dot{W}_{t+1}\bigr)\bigr] dw.
\end{equation*}
The optimal replenish-up-to and dispense-down-to policies under state aggregation can be written as 
\begin{equation*}
    \textstyle \dot{\pi}_t^{\textnormal{rep},*}(r,\dot{w}_j) \in \argmax_{z^{\textnormal{rep}} \in\bar{\mZ}(r)} \dot{\tilde{V}}_t^{\textnormal{rep}}(z,\dot{w}_j),
\end{equation*}
\begin{equation*}
	\textstyle \dot{\pi}_t^{\textnormal{dis},*}(z^{\textnormal{rep}},\dot{w}_j) \in \argmax_{z^\textnormal{dis}\in\underline{\mZ}(z^{\textnormal{rep}})} \; \dot{\tilde{V}}_t^\textnormal{dis}(z^\textnormal{dis},\dot{w}_j),
\end{equation*}
The postdecision Bellman equation under state aggregation is $\dot{\tilde{V}}_{T-1}^\textnormal{dis}(z^\textnormal{dis},\dot{w}_j) = -b\, z^\textnormal{dis}$, and for any $t<T-1$, 
\begin{equation*}
	\begin{aligned}
	\dot{\tilde{V}}_t^{\textnormal{rep}}& (z^{\textnormal{rep}},\dot{w}_j) \\
	&= \int_{w\in\mW_j} p^j(w) \, \bigl\{ - c_w z^{\textnormal{rep}} + \E_w \bigl[ U_{w,0}^{\bar{\mu}}\bigl(z^{\textnormal{rep}} - \dot{\pi}_t^{\textnormal{dis},*}(z^{\textnormal{rep}},\dot{w}_j), \Xi_{t,0}\bigr)\bigr] + \dot{\tilde{V}}_t^{\textnormal{dis}}\bigl(\dot{\pi}_t^{\textnormal{dis},*}(z^{\textnormal{rep}},\dot{w}_j), \dot{w}_j\bigr) \bigr\} dw,
	\end{aligned} 
\end{equation*}
\begin{equation*}
	\dot{\tilde{V}}_t^{\textnormal{dis}}(z^{\textnormal{dis}},\dot{w}_j) = \int_{w\in\mW_j} p^j(w) \, \bigl\{\E_w \bigl[ (c_{\dot{W}_{t+1}} - h) z^{\textnormal{dis}} + \dot{\tilde{V}}_{t+1}^{\textnormal{rep}}\bigl(\dot{\pi}_t^{\textnormal{rep},*}(z^{\textnormal{dis}},\dot{W}_{t+1}), \dot{W}_{t+1}\bigr) \bigr]\bigr\} dw.
\end{equation*}
The properties of the aggregate problem are stated in Proposition~\ref{prop: dot_V_tilde is concave in z}. The result follows from the proof of Proposition~\ref{prop: V_tilde_separate is concave in z} and the fact that $L^\natural$-concavity is preserved under expectations.

\begin{restatable}{proposition}{propositiondotVtildeIsConcave} 
    \label{prop: dot_V_tilde is concave in z}
	Suppose Assumption~\ref{assumption: properties of u} is satisfied. Then, the structural properties in Proposition~\ref{prop: V_tilde_separate is concave in z} hold for the aggregate postdecision value functions $\dot{\tilde{V}}_t^{\textnormal{rep}}(z^{\textnormal{rep}},\dot{w}_j)$ and $\dot{\tilde{V}}_t^{\textnormal{dis}}(z^{\textnormal{dis}},\dot{w}_j)$ as well as the thresholds $\dot{l}_t^{\textnormal{rep}}(\dot{w}_j)$ and $\dot{l}_t^{\textnormal{dis}}(\dot{w}_j)$.
\end{restatable}
Proposition~\ref{prop: dot_V_tilde is concave in z} is the theoretical basis of the algorithm for the aggregate problem. At each iteration and each period in the algorithm, we sample/observe the true exogenous information process as in Algorithm~\ref{alg:sac}, while using the corresponding aggregate exogenous information states to update values and thresholds. The details are in Appendix~\ref{sec:alg_aggr}.

\subsection{Algorithm for the Aggregate Problem} \label{sec:alg_aggr}

We define some other notations. At iteration $k$ and period $t$, we use the same notations as in Section~\ref{sec:alg} to represent the observation of the exogenous information and the attribute, which are $w_t^k$ and $\xi_{t,1}^k$ respectively. The corresponding information partition and the aggregate exogenous information are $\mW_t^k$ and $\dot{w}_t^k$ respectively. 
For the process $\mathbf{Z}^k_t(w) = \bigl\{ (\check{w}_\tau^k, \check{\xi}_{\tau,1}^k) : \tau = t,\ldots,T-1 \bigr\}$, denote $\check{\mW}_t^k$ and $\dot{\check{w}}_t^k$ the corresponding information partition and the aggregate exogenous information at period $\tau$ respectively, and we have $\check{w}_t^k \in \check{\mW}_t^k$.
Let $\dot{f}_t^{\textnormal{rep}} \bigl(\dot{\tilde{\pi}}^{\textnormal{rep},k-1}, \dot{\tilde{\pi}}^{\textnormal{dis},k-1}; \mathbf{Z}^k_t(w_t), r_t\bigr)$ be the Monte Carlo estimates of the replenish-up-to postdecision value starting in period $t$ under the current aggregate policy approximations and an initial state $(r_t, w_t)$:
\begin{equation*} \label{eq.aggr_ft}
	\begin{aligned}
	\dot{f}_t^{\textnormal{rep}} \bigl(\dot{\tilde{\pi}}^{\textnormal{rep},k-1}, \dot{\tilde{\pi}}^{\textnormal{dis},k-1}; \mathbf{Z}^k_t(w_t), r_t\bigr) = 
	& \textstyle \sum_{\tau=t}^{T-2} \bigl[ -c_{\check{w}_\tau^k} \dot{\tilde{z}}_\tau^\textnormal{rep} + U_{\check{w}_{\tau}^k}^{\bar{\mu}} \bigl( \dot{\tilde{z}}_\tau^\textnormal{rep} - \dot{\tilde{z}}_\tau^\textnormal{dis}, \check{\xi}_{\tau,0}^k \bigr) + (c_{\check{w}_{\tau+1}^k} - h) \dot{\tilde{z}}_\tau^\textnormal{dis} \bigr] \\
    &-c_{\check{w}_{T-1}^k} \dot{\tilde{z}}_{T-1}^\textnormal{rep} + U_{\check{w}_{T-1}^k,0}^{\bar{\mu}} \bigl( \dot{\tilde{z}}_{T-1}^\textnormal{rep} - \dot{\tilde{z}}_{T-1}^\textnormal{dis}, \check{\xi}_{T-1,0}^k\bigr) - b\,\dot{\tilde{z}}_{T-1}^\textnormal{dis},
	\end{aligned}
\end{equation*}
where for all $\tau\geq t$, the aggregate policies are $\dot{\tilde{z}}_\tau^\textnormal{rep} = \dot{\tilde{\pi}}_\tau^{\textnormal{rep},k-1}(r_\tau,\dot{\check{w}}_\tau^k)$, $\dot{\tilde{z}}_\tau^\textnormal{dis} = \dot{\tilde{\pi}}_\tau^{\textnormal{dis},k-1} \bigl( \tilde{\pi}_\tau^{\textnormal{rep},k-1}(r_\tau,\dot{\check{w}}_\tau^k), \dot{\check{w}}_\tau^k \bigr)$.
Let $\dot{f}_t^{\textnormal{dis}} \bigl(\dot{\tilde{\pi}}^{\textnormal{rep},k-1}, \dot{\tilde{\pi}}^{\textnormal{dis},k-1}; \mathbf{Z}^k_t(w_t), z_t^{\textnormal{rep}}\bigr)$ be the Monte Carlo estimates of the dispense-down-to postdecision value starting in period $t$ under the current aggregate policy approximations and an initial state $(z_t^{\textnormal{rep}},w_t)$:
\begin{equation*}
    \begin{aligned}
	\textstyle \dot{f}_t^{\textnormal{dis}} \bigl(\dot{\tilde{\pi}}^{\textnormal{rep},k-1}, \dot{\tilde{\pi}}^{\textnormal{dis},k-1}; &\mathbf{Z}^k_t(w_t), z_t^{\textnormal{rep}}\bigr)\\
	&=\textstyle \sum_{\tau=t}^{T-2} \bigl[ (c_{\check{w}_{\tau+1}^k} - h) \dot{\tilde{z}}_\tau^\textnormal{dis} - c_{\check{w}_{\tau+1}^k} \dot{\tilde{z}}_{\tau+1}^\textnormal{rep} + U_{\check{w}_{\tau+1}^k,0}^{\bar{\mu}} \bigl( \tilde{z}_{\tau+1}^\textnormal{rep} - \dot{\tilde{z}}_{\tau+1}^\textnormal{dis}, \check{\xi}_{\tau+1,0}^k \bigr) \bigr] - b\,\dot{\tilde{z}}_{T-1}^\textnormal{dis},
	\end{aligned}
\end{equation*}
where $\dot{\tilde{z}}_t^\textnormal{dis} = \dot{\tilde{\pi}}_t^{\textnormal{dis},k-1} \bigl( z_t^{\textnormal{rep}}, \dot{\check{w}}_\tau^k \bigr)$, and for all $\tau\geq t+1$, $\dot{\tilde{z}}_\tau^\textnormal{rep} = \dot{\tilde{\pi}}_\tau^{\textnormal{rep},k-1}(r_\tau,\dot{\check{w}}_\tau^k)$, 
$\dot{\tilde{z}}_\tau^\textnormal{dis} = \dot{\tilde{\pi}}_\tau^{\textnormal{dis},k-1} \bigl( \dot{\tilde{z}}_\tau^\textnormal{rep}, \dot{\check{w}}_\tau^k \bigr)$.

At each period $t$, to compute the approximate slopes, we use $\dot{f}_{t}^{\textnormal{dis}}$ to observe values $\dot{\hat{V}}_t^{\textnormal{rep},k} (z_t^{\textnormal{rep},k}, \dot{w}_t^k)$ and $\dot{\hat{V}}_t^{\textnormal{rep},k} (z_t^{\textnormal{rep},k}-1, \dot{w}_t^k)$, and $\dot{f}_{t+1}^{\textnormal{rep}}$ to observe values $\dot{\hat{V}}_t^{\textnormal{dis},k} (z_t^{\textnormal{dis},k}, \dot{w}_t^k)$ and $\dot{\hat{V}}_t^{\textnormal{dis},k} (z_t^{\textnormal{dis},k}-1, \dot{w}_t^k)$, where $\dot{f}_{t}^{\textnormal{dis}}$ and $\dot{f}_{t+1}^{\textnormal{rep}}$ are implied by the current aggregate policies $\dot{\bar{\pi}}^{\textnormal{rep},k-1}$ and $\dot{\bar{\pi}}^{\textnormal{dis},k-1}$; specifically, for $z^\textnormal{rep},z^\textnormal{dis} \ge 0$, the observations $\dot{\hat{V}}_t^{\textnormal{rep},k}(z^\textnormal{rep},\dot{w}_t^k)$ and $\dot{\hat{V}}_t^{\textnormal{dis},k}(z^\textnormal{dis},\dot{w}_t^k)$ are 
\begin{equation*}
	\begin{aligned}
	\dot{\hat{V}}_t^{\textnormal{rep},k}(z^\textnormal{rep},\dot{w}_t^k) =& -c_{\dot{w}_t^k} z^\textnormal{rep} + U_{\check{w}_{t}^k,0}^{\bar{\mu}} \bigl( z^\textnormal{rep} - \dot{\tilde{\pi}}_t^{\textnormal{dis},k-1}(z^\textnormal{rep},\dot{w}_t^k), \check{\xi}_{t,0}^k \bigr) \\&+ \dot{f}_{t}^{\textnormal{dis}} \bigl(\dot{\tilde{\pi}}^{\textnormal{rep},k-1}, \dot{\tilde{\pi}}^{\textnormal{dis},k-1}; \mathbf{Z}^k_{t}(\dot{w}_{t}), z^\textnormal{rep}\bigr),
	\end{aligned}
\end{equation*}
and
\begin{equation*}
	\begin{aligned}
	\hat{V}_t^{\textnormal{dis},k}(z^\textnormal{dis},\dot{w}_t^k) =& (c_{w_{t+1}} - h)z^\textnormal{dis} + \dot{f}_{t+1}^{\textnormal{rep}} \bigl(\dot{\tilde{\pi}}^{\textnormal{rep},k-1}, \dot{\tilde{\pi}}^{\textnormal{dis},k-1}; \mathbf{Z}^k_{t+1}(w_{t+1}), z^\textnormal{dis}\bigr),
	\end{aligned}
\end{equation*}
where $w_{t+1}$ is a realization from the distribution $W_{t+1}\, |\, W_t=\dot{w}_t^k$. 
The approximate slopes $\dot{\hat{v}}_t^{\textnormal{rep},k}$ and $\dot{\hat{v}}_t^{\textnormal{dis},k}$ are given by: 
\begin{equation} \label{eq.aggr_v_hat_rep}
	\dot{\hat{v}}_t^{\textnormal{rep},k} = \dot{\hat{V}}_t^{\textnormal{rep},k}(z_t^{\textnormal{rep},k}, \dot{w}_t^k) - \dot{\hat{V}}_t^{\textnormal{rep},k}(z_t^{\textnormal{rep},k}-1, \dot{w}_t^k),
\end{equation}
\begin{equation} \label{eq.aggr_v_hat_dis}
	\dot{\hat{v}}_t^{\textnormal{dis},k} = \dot{\hat{V}}_t^{\textnormal{dis},k}(z_t^{\textnormal{dis},k}, \dot{w}_t^k) - \dot{\hat{V}}_t^{\textnormal{dis},k}(z_t^{\textnormal{dis},k}-1, \dot{w}_t^k),
\end{equation}
where we define $\dot{\hat{V}}_t^{\textnormal{rep},k}(-1, \dot{w}_t^k) = \dot{\hat{V}}_t^{\textnormal{dis},k}(-1, \dot{w}_t^k) \equiv 0$. 
Under the assumption that $p^j(\cdot)$ is a uniform density function for all $j$, an algorithm for the aggregate problem is given in Algorithm~\ref{alg:aggr-sac}. 

\IncMargin{1em}
\begin{algorithm}[ht]
    \small
	\SetKwInput{Input}{Input}\SetKwInput{Output}{Output}
	\Input{Lower-level approximate policy $\bar{\mu}$ (learned from backward dynamic programming in the aggregate state space and extrapolated to continuous state space by linear regression).
	Initial policy estimates $\dot{\bar{l}}^{\textnormal{rep},0}$ and $\dot{\bar{\pi}}^{\textnormal{dis},0}$, and value estimates $\dot{\bar{v}}^{\textnormal{rep},0}$ and $\dot{\bar{v}}^{\textnormal{dis},0}$ (nonincreasing in $z^\textnormal{rep}$ and $z^\textnormal{dis}$ respectively).
	Stepsize rules $\tilde{\alpha}_t^k$ and $\tilde{\beta}_t^k$ for all $t,k$.}
	\medskip
	\Output{Approximations $\{\dot{\bar{l}}^{\textnormal{rep},k}\}$, $\{\dot{\bar{\pi}}^{\textnormal{dis},k}\}$, $\{\dot{\bar{v}}^{\textnormal{rep},k}\}$, and $\{\dot{\bar{v}}^{\textnormal{dis},k}\}$.}
	\BlankLine
	\For{$k=1,2,\ldots$}{
		\vspace{0.5em}
		Sample initial states $z_0^{\textnormal{rep},k}$ and $z_0^{\textnormal{dis},k}$. \label{alg:aggr-sac0}\\
		\vspace{0.5em}
		\For{$t=0,1,\ldots,T-1$}{\label{alg: Aggr-forins}
			\vspace{0.5em}
			Observe $w_t^k$ and $\xi_{t,1}^k$, then observe $\dot{\hat{v}}_t^{\textnormal{rep},k}$ and $\dot{\hat{v}}_t^{\textnormal{dis},k}$ according to \eqref{eq.aggr_v_hat_rep} and \eqref{eq.aggr_v_hat_dis} respectively. \label{alg:aggr-sac1}\\
			\vspace{0.5em}
			Perform SA step: \label{alg:aggr-sac2}
			\begin{equation*}
			    \begin{aligned}
			    \textstyle \dot{\tilde{v}}_t^{\textnormal{rep},k}(z^\textnormal{rep},\dot{w}) &= \bigl(1-\alpha_t^k(z^\textnormal{rep},\dot{w})\bigr)\, \dot{\bar{v}}_t^{\textnormal{rep},k-1}(z^\textnormal{rep},\dot{w}) + \alpha_t^k(z^\textnormal{rep},\dot{w})\, \dot{\hat{v}}_t^{\textnormal{rep},k}, \\
			    \textstyle \dot{\tilde{v}}_t^{\textnormal{dis},k}(z^\textnormal{dis},\dot{w}) &= \bigl(1-\alpha_t^k(z^\textnormal{dis},\dot{w})\bigr)\, \dot{\bar{v}}_t^{\textnormal{dis},k-1}(z^\textnormal{dis},\dot{w}) + \alpha_t^k(z^\textnormal{dis},\dot{w})\, \dot{\hat{v}}_t^{\textnormal{dis},k}.
			    \end{aligned}
			\end{equation*}
			\\
			\vspace{0.5em}
			Perform the concavity projection operation \eqref{eq.concavity-projection}: \label{alg:aggr-sac3}
			\begin{equation*}
			    \textstyle \dot{\bar{v}}_t^{\textnormal{rep},k} = \Pi_{z_t^{\textnormal{rep},k},\dot{w}_t^k} (\dot{\tilde{v}}_t^{\textnormal{rep},k}), \;\;\;\; \dot{\bar{v}}_t^{\textnormal{dis},k} = \Pi_{z_t^{\textnormal{dis},k},\dot{w}_t^k} (\dot{\tilde{v}}_t^{\textnormal{dis},k}).
			\end{equation*}
			\\
			\vspace{0.5em}
			Observe and update the replenish-up-to threshold: \label{alg:aggr-sac4}
			\begin{equation*}
			    \dot{\hat{l}}_t^{\textnormal{rep},k} =\textstyle \argmax_{z^\textnormal{rep} \in \bar{\mZ}(0)} \sum_{j=0}^{z^\textnormal{rep}} \dot{\bar{v}}_t^{\textnormal{rep},k}\bigl(j,\dot{w}_t^k\bigr), 
			\end{equation*}
			\begin{equation*}
			    \dot{\bar{l}}_t^{\textnormal{rep},k}(\dot{w}) = \bigl(1-\beta_t^k(\dot{w})\bigr)\,\dot{\bar{l}}_t^{\textnormal{rep},k-1}(\dot{w}) + \beta_t^k(\dot{w})\,\dot{\hat{l}}_t^{\textnormal{rep},k}.
			\end{equation*}
			\\
			\vspace{0.5em}
			Observe and update the dispense-down-to policy: \label{alg:sac5}
			\\
			\vspace{0.5em}
			\For{$z_t^{\textnormal{rep}} = 0,1,\ldots, R_{\textnormal{max}}$ \label{alg:agg-sac6}}{
			\begin{equation*}
			    \dot{\hat{\pi}}_t^{\textnormal{dis}} =\textstyle \argmax_{z^\textnormal{dis} \in \underline{\mZ}(z_t^{\textnormal{rep}})} U_{w_{t}^k,0}^{\bar{\mu}} \bigl(z_t^{\textnormal{rep}} - z^\textnormal{dis}, \xi_{t,0}^k\bigr) + \sum_{j=0}^{z^\textnormal{dis}} \dot{\bar{v}}_t^{\textnormal{dis},k}\bigl(j,\dot{w}_t^k\bigr),
			\end{equation*}
			\begin{equation*}
			    \dot{\bar{\pi}}_t^{\textnormal{dis},k} (z^{\textnormal{rep}},\dot{w}) = \bigl(1-\alpha^k(z^{\textnormal{rep}},\dot{w})\bigr)\,\dot{\bar{\pi}}_t^{\textnormal{dis},k-1}(z^{\textnormal{rep}},\dot{w}) + \alpha^k(z^{\textnormal{rep}},\dot{w})\,\dot{\hat{\pi}}_t^{\textnormal{dis}}.
			\end{equation*}
			}
			\vspace{0.5em}
			If $t<T-1$, take $z_{t+1}^{\textnormal{rep},k}$ and $z_{t+1}^{\textnormal{dis},k}$ according to the $\epsilon$-greedy exploration policy. \label{alg:aggr-sac6}
		}
	}
	\caption{Aggregate Structured Actor-Critic Method}
	\label{alg:aggr-sac}
\end{algorithm}\DecMargin{1em}

\clearpage

\end{appendices}

\clearpage

\bibliographystyle{abbrvnat}
{\small 
	\bibliography{bibliography}}

\begin{thebibliography}{105}
\providecommand{\natexlab}[1]{#1}
\providecommand{\url}[1]{\texttt{#1}}
\expandafter\ifx\csname urlstyle\endcsname\relax
  \providecommand{\doi}[1]{doi: #1}\else
  \providecommand{\doi}{doi: \begingroup \urlstyle{rm}\Url}\fi

\bibitem[Ablah et~al.(2010)Ablah, Scanlon, Konda, Tinius, and
  Gebbie]{ablah2010large}
E.~Ablah, E.~Scanlon, K.~Konda, A.~Tinius, and K.~M. Gebbie.
\newblock A large-scale points-of-dispensing exercise for first responders and
  first receivers in {Nassau County}, {New York}.
\newblock \emph{Biosecurity and Bioterrorism: Biodefense Strategy, Practice,
  and Science}, 8\penalty0 (1):\penalty0 25--35, 2010.

\bibitem[Acharya et~al.(2020)Acharya, Chopra, Hayes, Teeter, and
  Martin]{acharya2020cost}
M.~Acharya, D.~Chopra, C.~J. Hayes, B.~Teeter, and B.~C. Martin.
\newblock Cost-effectiveness of intranasal naloxone distribution to high-risk
  prescription opioid users.
\newblock \emph{Value in Health}, 23\penalty0 (4):\penalty0 451--460, 2020.

\bibitem[Bean et~al.(1987)Bean, Birge, and Smith]{bean1987aggregation}
J.~C. Bean, J.~R. Birge, and R.~L. Smith.
\newblock Aggregation in dynamic programming.
\newblock \emph{Operations Research}, 35\penalty0 (2):\penalty0 215--220, 1987.

\bibitem[Bennett et~al.(2018)Bennett, Bell, Doe-Simkins, Elliott, Pouget, and
  Davis]{doi:10.1080/02791072.2018.1430409}
A.~S. Bennett, A.~Bell, M.~Doe-Simkins, L.~Elliott, E.~Pouget, and C.~Davis.
\newblock {From peers to lay bystanders: Findings from a decade of naloxone
  distribution in Pittsburgh, PA}.
\newblock \emph{Journal of Psychoactive Drugs}, 0\penalty0 (0):\penalty0 1--7,
  2018.

\bibitem[Bertsekas(1975)]{bertsekas1975convergence}
D.~Bertsekas.
\newblock Convergence of discretization procedures in dynamic programming.
\newblock \emph{IEEE Transactions on Automatic Control}, 20\penalty0
  (3):\penalty0 415--419, 1975.

\bibitem[Bertsekas and Tsitsiklis(1996)]{bertsekas1996neuro-dynamic}
D.~P. Bertsekas and J.~N. Tsitsiklis.
\newblock \emph{Neuro-Dynamic Programming}, volume~3.
\newblock Athena Scientific, 1996.

\bibitem[Bijvank and Vis(2011)]{bijvank2011lost}
M.~Bijvank and I.~F. Vis.
\newblock Lost-sales inventory theory: A review.
\newblock \emph{European Journal of Operational Research}, 215\penalty0
  (1):\penalty0 1--13, 2011.

\bibitem[{Centers for Disease Control and
  Prevention}(2009{\natexlab{a}})]{cdc2009H1N1}
{Centers for Disease Control and Prevention}.
\newblock 2009 {H1N1} early outbreak and disease characteristics.
\newblock 2009{\natexlab{a}}.
\newblock URL \url{https://www.cdc.gov/h1n1flu/surveillanceqa.htm#7}.

\bibitem[{Centers for Disease Control and
  Prevention}(2009{\natexlab{b}})]{cdc2009H1N1vaccine}
{Centers for Disease Control and Prevention}.
\newblock Vaccine against 2009 {H1N1} influenza virus.
\newblock 2009{\natexlab{b}}.
\newblock URL
  \url{https://www.cdc.gov/h1n1flu/vaccination/public/vaccination_qa_pub.htm}.

\bibitem[{Centers for Disease Control and Prevention}(2012)]{cdc2012VFC}
{Centers for Disease Control and Prevention}.
\newblock {VFC} will benefit your patients and your practice!
\newblock 2012.
\newblock URL
  \url{https://www.cdc.gov/vaccines/programs/vfc/providers/questions/qa-flyer-hcp.html}.

\bibitem[{Centers for Disease Control and Prevention}(2014)]{cdc2014VFC}
{Centers for Disease Control and Prevention}.
\newblock Vaccines for {Children Program (VFC)}.
\newblock 2014.
\newblock URL \url{https://www.cdc.gov/vaccines/programs/vfc/about/index.html}.

\bibitem[{Centers for Disease Control and Prevention}(2019)]{cdc2018hepatitisA}
{Centers for Disease Control and Prevention}.
\newblock Widespread outbreaks of hepatitis {A} across the {United States}.
\newblock 2019.
\newblock URL
  \url{https://www.cdc.gov/hepatitis/outbreaks/2017March-HepatitisA.htm}.

\bibitem[{Centers for Disease Control and
  Prevention}(2021{\natexlab{a}})]{cdc2021how}
{Centers for Disease Control and Prevention}.
\newblock How {CDC} is making {COVID-19} vaccine recommendations,
  2021{\natexlab{a}}.
\newblock URL
  \url{https://www.cdc.gov/coronavirus/2019-ncov/vaccines/recommendations-process.html}.

\bibitem[{Centers for Disease Control and
  Prevention}(2021{\natexlab{b}})]{cdc2021risk}
{Centers for Disease Control and Prevention}.
\newblock Risk for {COVID-19} infection, hospitalization, and death by age
  group, 2021{\natexlab{b}}.
\newblock URL
  \url{cdc.gov/coronavirus/2019-ncov/covid-data/investigations-discovery/hospitalization-death-by-age.html}.

\bibitem[{Centers for Disease Control and
  Prevention}(2021{\natexlab{c}})]{cdc2021trends}
{Centers for Disease Control and Prevention}.
\newblock Trends in number of {COVID-19} cases and deaths in the {US} reported
  to {CDC}, by state/territory, 2021{\natexlab{c}}.
\newblock URL
  \url{https://covid.cdc.gov/covid-data-tracker/#trends_totalandratecasessevendayrate}.

\bibitem[{Centers for Disease Control and
  Prevention}(2021{\natexlab{d}})]{cdc2021understanding}
{Centers for Disease Control and Prevention}.
\newblock Understanding the epidemic, 2021{\natexlab{d}}.
\newblock URL \url{https://www.cdc.gov/drugoverdose/epidemic/index.html}.

\bibitem[Chen and Samroengraja(2000)]{chen2000staggered}
F.~Chen and R.~Samroengraja.
\newblock A staggered ordering policy for one-warehouse, multiretailer systems.
\newblock \emph{Operations Research}, 48\penalty0 (2):\penalty0 281--293, 2000.

\bibitem[Chen(1999)]{chen1999application}
V.~C. Chen.
\newblock Application of orthogonal arrays and {MARS} to inventory forecasting
  stochastic dynamic programs.
\newblock \emph{Computational Statistics \& Data Analysis}, 30\penalty0
  (3):\penalty0 317--341, 1999.

\bibitem[Chen et~al.(1999)Chen, Ruppert, and Shoemaker]{chen1999applying}
V.~C. Chen, D.~Ruppert, and C.~A. Shoemaker.
\newblock Applying experimental design and regression splines to
  high-dimensional continuous-state stochastic dynamic programming.
\newblock \emph{Operations Research}, 47\penalty0 (1):\penalty0 38--53, 1999.

\bibitem[Chen et~al.(2014)Chen, Pang, and Pan]{chen2014coordinating}
X.~Chen, Z.~Pang, and L.~Pan.
\newblock Coordinating inventory control and pricing strategies for perishable
  products.
\newblock \emph{Operations Research}, 62\penalty0 (2):\penalty0 284--300, 2014.

\bibitem[Christie et~al.(2017)Christie, Baker, Cooper, Kennedy, Madras, and
  Bondi]{christie2017president}
C.~Christie, C.~Baker, R.~Cooper, P.~J. Kennedy, B.~Madras, and P.~Bondi.
\newblock The president's commission on combating drug addiction and the opioid
  crisis.
\newblock \emph{WhiteHouse.gov}, 2017.
\newblock URL
  \url{https://www.whitehouse.gov/sites/whitehouse.gov/files/images/Final_Report_Draft_11-1-2017.pdf}.

\bibitem[Clark and Scarf(1960)]{clark1960optimal}
A.~J. Clark and H.~Scarf.
\newblock Optimal policies for a multi-echelon inventory problem.
\newblock \emph{Management Science}, 6\penalty0 (4):\penalty0 475--490, 1960.

\bibitem[Coffin and Sullivan(2013)]{coffin2013cost}
P.~O. Coffin and S.~D. Sullivan.
\newblock Cost-effectiveness of distributing naloxone to heroin users for lay
  overdose reversal.
\newblock \emph{Annals of Internal Medicine}, 158\penalty0 (1):\penalty0 1--9,
  2013.

\bibitem[Cohn(2017)]{cohn2017baltimore}
M.~Cohn.
\newblock Baltimore city running low on opioid overdose remedy.
\newblock 2017.
\newblock URL
  \url{http://www.baltimoresun.com/health/bs-hs-naloxone-shortage-20170614-story.html}.

\bibitem[{Drug Enforcement Administration}(2015{\natexlab{a}})]{DEA2015NDTA}
{Drug Enforcement Administration}.
\newblock 2015 national drug threat assessment summary.
\newblock 2015{\natexlab{a}}.
\newblock URL \url{https://www.dea.gov/docs/2015%20NDTA%20Report.pdf}.

\bibitem[{Drug Enforcement
  Administration}(2015{\natexlab{b}})]{DEA2015NDTAannouncement}
{Drug Enforcement Administration}.
\newblock {DEA releases 2015 national drug threat assessment: Heroin and
  painkiller abuse continue to concern}.
\newblock 2015{\natexlab{b}}.
\newblock URL \url{https://www.dea.gov/divisions/hq/2015/hq110415.shtml}.

\bibitem[Fanshel and Bush(1970)]{fanshel1970health}
S.~Fanshel and J.~W. Bush.
\newblock A health-status index and its application to health-services
  outcomes.
\newblock \emph{Operations Research}, 18\penalty0 (6):\penalty0 1021--1066,
  1970.

\bibitem[{Florida Health}(2019)]{florida2019hepatitis}
{Florida Health}.
\newblock {Hepatitis A in Florida}.
\newblock 2019.
\newblock URL
  \url{http://www.floridahealth.gov/diseases-and-conditions/vaccine-preventable-disease/hepatitis-a/surveillance-data/}.

\bibitem[Fox(1973)]{fox1973discretizing}
B.~L. Fox.
\newblock Discretizing dynamic programs.
\newblock \emph{Journal of Optimization Theory and Applications}, 11\penalty0
  (3):\penalty0 228--234, 1973.

\bibitem[Fujishige and Murota(2000)]{fujishige2000notes}
S.~Fujishige and K.~Murota.
\newblock Notes on {L}-/{M}-convex functions and the separation theorems.
\newblock \emph{Mathematical Programming}, 88\penalty0 (1):\penalty0 129--146,
  2000.

\bibitem[Gong and Chao(2013)]{gong2013optimal}
X.~Gong and X.~Chao.
\newblock Optimal control policy for capacitated inventory systems with
  remanufacturing.
\newblock \emph{Operations Research}, 61\penalty0 (3):\penalty0 603--611, 2013.

\bibitem[Goodloe and Dailey(2014)]{goodloe2014should}
J.~M. Goodloe and M.~W. Dailey.
\newblock Should naloxone be available to all first responders?
\newblock \emph{Journal of Emergency Medical Services}, 2014.

\bibitem[{GoodRx}(2021)]{naloxoneprice}
{GoodRx}.
\newblock Evzio naloxone, 2021.
\newblock URL \url{https://www.goodrx.com/evzio}.

\bibitem[Graves(1996)]{graves1996multiechelon}
S.~C. Graves.
\newblock A multiechelon inventory model with fixed replenishment intervals.
\newblock \emph{Management Science}, 42\penalty0 (1):\penalty0 1--18, 1996.

\bibitem[Grob(2018)]{grob2018inventory}
C.~Grob.
\newblock \emph{Inventory Management in Multi-Echelon Networks: On the
  Optimization of Reorder Points}, volume 128.
\newblock Springer, 2018.

\bibitem[Gupta et~al.(2019)Gupta, Shukla, Marla, Kolbeinsson, and
  Yellepeddi]{gupta2019incorporate}
A.~Gupta, N.~Shukla, L.~Marla, A.~Kolbeinsson, and K.~Yellepeddi.
\newblock How to incorporate monotonicity in deep networks while preserving
  flexibility?
\newblock \emph{arXiv preprint arXiv:1909.10662}, 2019.

\bibitem[Gupta et~al.(2021)Gupta, Marla, Sun, Shukla, and
  Kolbeinsson]{gupta2021pender}
A.~Gupta, L.~Marla, R.~Sun, N.~Shukla, and A.~Kolbeinsson.
\newblock {PenDer}: Incorporating shape constraints via penalized derivatives.
\newblock 2021.

\bibitem[Ha(1997)]{ha1997inventory}
A.~Y. Ha.
\newblock Inventory rationing in a make-to-stock production system with several
  demand classes and lost sales.
\newblock \emph{Management Science}, 43\penalty0 (8):\penalty0 1093--1103,
  1997.

\bibitem[Huh and Janakiraman(2010)]{huh2010optimal}
W.~T. Huh and G.~Janakiraman.
\newblock On the optimal policy structure in serial inventory systems with lost
  sales.
\newblock \emph{Operations Research}, 58\penalty0 (2):\penalty0 486--491, 2010.

\bibitem[Ingraham(2016)]{ingraham2016heroin}
C.~Ingraham.
\newblock {Heroin deaths surpass gun homicides for the first time, CDC data
  shows}.
\newblock \emph{The Washington Post, December}, 8, 2016.

\bibitem[Janssen et~al.(2016)Janssen, Claus, and Sauer]{janssen2016literature}
L.~Janssen, T.~Claus, and J.~Sauer.
\newblock Literature review of deteriorating inventory models by key topics
  from 2012 to 2015.
\newblock \emph{International Journal of Production Economics}, 182:\penalty0
  86--112, 2016.

\bibitem[Kaplan and Bush(1982)]{kaplan1982health}
R.~M. Kaplan and J.~W. Bush.
\newblock Health-related quality of life measurement for evaluation research
  and policy analysis.
\newblock \emph{Health Psychology}, 1\penalty0 (1):\penalty0 61, 1982.

\bibitem[{Kentucky Department of Public Health}(2018)]{kentucky2018Hepatitis}
{Kentucky Department of Public Health}.
\newblock Department for public health.
\newblock 2018.
\newblock URL \url{http://chfs.ky.gov/dph}.

\bibitem[Konda and Tsitsiklis(2000)]{konda2000actor}
V.~R. Konda and J.~N. Tsitsiklis.
\newblock Actor-critic algorithms.
\newblock In \emph{Advances in Neural Information Processing Systems}, pages
  1008--1014, 2000.

\bibitem[Kunnumkal and Topaloglu(2008)]{kunnumkal2008using}
S.~Kunnumkal and H.~Topaloglu.
\newblock Using stochastic approximation methods to compute optimal base-stock
  levels in inventory control problems.
\newblock \emph{Operations Research}, 56\penalty0 (3):\penalty0 646--664, 2008.

\bibitem[Kushner and Yin(2003)]{kushner2003stochastic}
H.~Kushner and G.~Yin.
\newblock \emph{Stochastic Approximation and Recursive Algorithms and
  Applications}, volume~35.
\newblock Springer-Verlag New York, 2003.

\bibitem[Kwan-Gett et~al.(2009)Kwan-Gett, Baer, and Duchin]{kwan2009spring}
T.~S. Kwan-Gett, A.~Baer, and J.~S. Duchin.
\newblock {Spring 2009 H1N1 influenza outbreak in King County, Washington}.
\newblock \emph{Disaster Medicine and Public Health Preparedness}, 3\penalty0
  (S2):\penalty0 S109--S116, 2009.

\bibitem[Langham et~al.(2018)Langham, Wright, Kenworthy, Grieve, and
  Dunlop]{langham2018cost}
S.~Langham, A.~Wright, J.~Kenworthy, R.~Grieve, and W.~C. Dunlop.
\newblock Cost-effectiveness of take-home naloxone for the prevention of
  overdose fatalities among heroin users in the {United Kingdom}.
\newblock \emph{Value in Health}, 21\penalty0 (4):\penalty0 407--415, 2018.

\bibitem[Lee et~al.(2015)Lee, Yuan, Pietz, Benecke, and Burel]{lee2015vaccine}
E.~K. Lee, F.~Yuan, F.~H. Pietz, B.~A. Benecke, and G.~Burel.
\newblock Vaccine prioritization for effective pandemic response.
\newblock \emph{Interfaces}, 45\penalty0 (5):\penalty0 425--443, 2015.

\bibitem[Levinson(2012)]{levinson2012vaccines}
D.~R. Levinson.
\newblock Vaccines for children program: Vulnerabilities in vaccine management,
  2012.

\bibitem[Löhndorf et~al.(2013)Löhndorf, Wozabal, and
  Minner]{nils2013optimizing}
N.~Löhndorf, D.~Wozabal, and S.~Minner.
\newblock Optimizing trading decisions for hydro storage systems using
  approximate dual dynamic programming.
\newblock \emph{Operations Research}, 61\penalty0 (4):\penalty0 810--823, 2013.

\bibitem[Lu and Song(2005)]{lu2005order}
Y.~Lu and J.-S. Song.
\newblock Order-based cost optimization in assemble-to-order systems.
\newblock \emph{Operations Research}, 53\penalty0 (1):\penalty0 151--169, 2005.

\bibitem[{Michigan Department of Health and Human
  Services}(2018)]{michigan2018Hepatitis}
{Michigan Department of Health and Human Services}.
\newblock {Michigan Hepatitis A outbreak}.
\newblock 2018.
\newblock URL
  \url{http://www.michigan.gov/mdhhs/0,5885,7-339-71550_2955_2976_82305_82310-447907--,00.html}.

\bibitem[Mohebbi(2003)]{mohebbi2003supply}
E.~Mohebbi.
\newblock Supply interruptions in a lost-sales inventory system with random
  lead time.
\newblock \emph{Computers \& Operations Research}, 30\penalty0 (3):\penalty0
  411--426, 2003.

\bibitem[Mousavi et~al.(2004)Mousavi, Mahdizadeh, and
  Afshar]{mousavi2004stochastic}
S.~J. Mousavi, K.~Mahdizadeh, and A.~Afshar.
\newblock A stochastic dynamic programming model with fuzzy storage states for
  reservoir operations.
\newblock \emph{Advances in Water Resources}, 27\penalty0 (11):\penalty0
  1105--1110, 2004.

\bibitem[Murota(1998)]{murota1998discrete}
K.~Murota.
\newblock Discrete convex analysis.
\newblock \emph{Mathematical Programming}, 83\penalty0 (1):\penalty0 313--371,
  1998.

\bibitem[Murota and Shioura(2000)]{murota2000extension}
K.~Murota and A.~Shioura.
\newblock {Extension of M-convexity and L-convexity to polyhedral convex
  functions}.
\newblock \emph{Advances in Applied Mathematics}, 25\penalty0 (4):\penalty0
  352--427, 2000.

\bibitem[Nahmias(1979)]{nahmias1979simple}
S.~Nahmias.
\newblock Simple approximations for a variety of dynamic leadtime lost-sales
  inventory models.
\newblock \emph{Operations Research}, 27\penalty0 (5):\penalty0 904--924, 1979.

\bibitem[Nascimento and Powell(2009)]{nascimento2009optimal}
J.~M. Nascimento and W.~B. Powell.
\newblock An optimal approximate dynamic programming algorithm for the lagged
  asset acquisition problem.
\newblock \emph{Mathematics of Operations Research}, 34\penalty0 (1):\penalty0
  210--237, 2009.

\bibitem[Neumann et~al.(2009)Neumann, Noda, and Kawaoka]{neumann2009emergence}
G.~Neumann, T.~Noda, and Y.~Kawaoka.
\newblock {Emergence and pandemic potential of swine-origin H1N1 influenza
  virus}.
\newblock \emph{Nature}, 459\penalty0 (7249):\penalty0 931, 2009.

\bibitem[{Novel Swine-Origin Influenza A (H1N1) Virus Investigation
  Team}(2009)]{novel2009emergence}
{Novel Swine-Origin Influenza A (H1N1) Virus Investigation Team}.
\newblock {Emergence of a novel swine-origin influenza A (H1N1) virus in
  humans}.
\newblock \emph{New England Journal of Medicine}, 360\penalty0 (25):\penalty0
  2605--2615, 2009.

\bibitem[{Ohio Department of Health}(2019)]{ohio2019hepatitis}
{Ohio Department of Health}.
\newblock Hepatitis {A} statewide community outbreak.
\newblock 2019.
\newblock URL
  \url{https://odh.ohio.gov/wps/portal/gov/odh/know-our-programs/outbreak-response-bioterrorism-investigation-team/news-and-events/newsevent1}.

\bibitem[{Open Data Pennsylvania}(2021)]{overdose2021}
{Open Data Pennsylvania}.
\newblock Overdose information network data {CY} {J}anuary 2018 - current
  monthly county state police, 2021.
\newblock URL
  \url{https://data.pa.gov/Opioid-Related/Overdose-Information-Network-Data-CY-January-2018-/hbkk-dwy3}.

\bibitem[Pang et~al.(2012)Pang, Chen, and Feng]{pang2012note}
Z.~Pang, F.~Y. Chen, and Y.~Feng.
\newblock A note on the structure of joint inventory-pricing control with
  leadtimes.
\newblock \emph{Operations Research}, 60\penalty0 (3):\penalty0 581--587, 2012.

\bibitem[Pereira and Pinto(1991)]{pereira1991stochastic}
M.~Pereira and L.~Pinto.
\newblock Stochastic dual dynamic programming.
\newblock \emph{Mathematical Programming}, 52:\penalty0 359--375, 1991.

\bibitem[Philpott and Guan(2008)]{philpott2008convergence}
A.~B. Philpott and Z.~Guan.
\newblock On the convergence of stochastic dual dynamic programming and related
  methods.
\newblock \emph{Operations Research Letters}, 36\penalty0 (4):\penalty0
  450--455, 2008.

\bibitem[Powell et~al.(2004)Powell, Ruszczy{\'n}ski, and
  Topaloglu]{powell2004learning}
W.~Powell, A.~Ruszczy{\'n}ski, and H.~Topaloglu.
\newblock Learning algorithms for separable approximations of discrete
  stochastic optimization problems.
\newblock \emph{Mathematics of Operations Research}, 29\penalty0 (4):\penalty0
  814--836, 2004.

\bibitem[Powell(2007)]{powell2007approximate}
W.~B. Powell.
\newblock \emph{Approximate Dynamic Programming: Solving the Curses of
  Dimensionality}, volume 703.
\newblock John Wiley \& Sons, 2007.

\bibitem[Rambhia et~al.(2010)Rambhia, Watson, Sell, Waldhorn, and
  Toner]{rambhia2010mass}
K.~J. Rambhia, M.~Watson, T.~K. Sell, R.~Waldhorn, and E.~Toner.
\newblock {Mass vaccination for the 2009 H1N1 pandemic: Approaches, challenges,
  and recommendations}.
\newblock \emph{Biosecurity and Bioterrorism: Biodefense Strategy, Practice,
  and Science}, 8\penalty0 (4):\penalty0 321--330, 2010.

\bibitem[Rando et~al.(2015)Rando, Broering, Olson, Marco, and
  Evans]{rando2015intranasal}
J.~Rando, D.~Broering, J.~E. Olson, C.~Marco, and S.~B. Evans.
\newblock Intranasal naloxone administration by police first responders is
  associated with decreased opioid overdose deaths.
\newblock \emph{The American Journal of Emergency Medicine}, 33\penalty0
  (9):\penalty0 1201--1204, 2015.

\bibitem[Ren and Krogh(2002)]{ren2002state}
Z.~Ren and B.~H. Krogh.
\newblock State aggregation in {Markov} decision processes.
\newblock In \emph{Proceedings of the 41st IEEE Conference on Decision and
  Control}, volume~4, pages 3819--3824. IEEE, 2002.

\bibitem[Robbins and Monro(1951)]{robbins1951stochastic}
H.~Robbins and S.~Monro.
\newblock A stochastic approximation method.
\newblock \emph{The Annals of Mathematical Statistics}, pages 400--407, 1951.

\bibitem[Rudd et~al.(2016)Rudd, Aleshire, Zibbell, and
  Matthew~Gladden]{rudd2000increases}
R.~A. Rudd, N.~Aleshire, J.~E. Zibbell, and R.~Matthew~Gladden.
\newblock {Increases in drug and opioid overdose deaths—United States,
  2000--2014}.
\newblock \emph{American Journal of Transplantation}, 16\penalty0 (4):\penalty0
  1323--1327, 2016.

\bibitem[Santoli et~al.(1999)Santoli, Rodewald, Maes, Battaglia, and
  Coronado]{santoli1999vaccines}
J.~M. Santoli, L.~E. Rodewald, E.~F. Maes, M.~P. Battaglia, and V.~G. Coronado.
\newblock {Vaccines for children program, United States, 1997}.
\newblock \emph{Pediatrics}, 104\penalty0 (2):\penalty0 e15, 1999.

\bibitem[Schweitzer et~al.(1985)Schweitzer, Puterman, and
  Kindle]{schweitzer1985iterative}
P.~J. Schweitzer, M.~L. Puterman, and K.~W. Kindle.
\newblock Iterative aggregation-disaggregation procedures for discounted
  semi-{Markov} reward processes.
\newblock \emph{Operations Research}, 33\penalty0 (3):\penalty0 589--605, 1985.

\bibitem[Shapiro(2011)]{shapiro2011analysis}
A.~Shapiro.
\newblock Analysis of stochastic dual dynamic programming method.
\newblock \emph{European Journal of Operational Research}, 209\penalty0
  (1):\penalty0 63--72, 2011.

\bibitem[Singh et~al.(1995)Singh, Jaakkola, and Jordan]{singh1995reinforcement}
S.~P. Singh, T.~Jaakkola, and M.~I. Jordan.
\newblock Reinforcement learning with soft state aggregation.
\newblock In \emph{Advances in Neural Information Processing Systems}, pages
  361--368, 1995.

\bibitem[Smith et~al.(2005)Smith, Santoli, Chu, Ochoa, and
  Rodewald]{smith2005association}
P.~J. Smith, J.~M. Santoli, S.~Y. Chu, D.~Q. Ochoa, and L.~E. Rodewald.
\newblock The association between having a medical home and vaccination
  coverage among children eligible for the vaccines for children program.
\newblock \emph{Pediatrics}, 116\penalty0 (1):\penalty0 130--139, 2005.

\bibitem[{Social Security Online}(2005)]{social2005VaccineforChildren}
{Social Security Online}.
\newblock Compilation of the social security laws: Program for distribution of
  pediatric vaccines: {SEC} 1928 [42 {USC} 1396s].
\newblock 2005.
\newblock URL \url{https://www.ssa.gov/OP_Home/ssact/title19/1928.htm}.

\bibitem[Sutton and Barto(1998)]{sutton1998reinforcement}
R.~S. Sutton and A.~G. Barto.
\newblock \emph{Reinforcement Learning: An Introduction}, volume~1.
\newblock MIT press Cambridge, 1998.

\bibitem[Sutton et~al.(2000)Sutton, McAllester, Singh, and
  Mansour]{sutton2000policy}
R.~S. Sutton, D.~A. McAllester, S.~P. Singh, and Y.~Mansour.
\newblock Policy gradient methods for reinforcement learning with function
  approximation.
\newblock In \emph{Advances in Neural Information Processing Systems}, pages
  1057--1063, 2000.

\bibitem[Tan(1974)]{tan1974optimal}
F.~K. Tan.
\newblock Optimal policies for a multi-echelon inventory problem with periodic
  ordering.
\newblock \emph{Management Science}, 20\penalty0 (7):\penalty0 1104--1111,
  1974.

\bibitem[{Tennessee Department of Health}(2019)]{tennessee2019hepatitis}
{Tennessee Department of Health}.
\newblock {Tennessee} {Hepatitis A} outbreak.
\newblock 2019.
\newblock URL
  \url{https://www.tn.gov/health/cedep/tennessee-hepatitis-a-outbreak.html}.

\bibitem[{The Economist}(2021)]{opioid2021}
{The Economist}.
\newblock Opioid deaths in {America} reached new highs in the pandemic, 2021.
\newblock URL
  \url{https://www.economist.com/graphic-detail/2021/03/30/opioid-deaths-in-america-reached-new-highs-in-the-pandemic}.

\bibitem[Torrance et~al.(1972)Torrance, Thomas, and
  Sackett]{torrance1972utility}
G.~W. Torrance, W.~H. Thomas, and D.~L. Sackett.
\newblock A utility maximization model for evaluation of health care programs.
\newblock \emph{Health Services Research}, 7\penalty0 (2):\penalty0 118, 1972.

\bibitem[Tsitsiklis and Van~Roy(1996)]{tsitsiklis1996feature}
J.~N. Tsitsiklis and B.~Van~Roy.
\newblock Feature-based methods for large scale dynamic programming.
\newblock \emph{Machine Learning}, 22\penalty0 (1-3):\penalty0 59--94, 1996.

\bibitem[{U.S. Department of Health and Human Services and Centers for Disease
  Control and Prevention}(2018)]{cdc2018vaccineStorage}
{U.S. Department of Health and Human Services and Centers for Disease Control
  and Prevention}.
\newblock Vaccine storage and handling toolkit.
\newblock 2018.
\newblock URL
  \url{https://www.cdc.gov/vaccines/hcp/admin/storage/toolkit/storage-handling-toolkit.pdf}.

\bibitem[Van~Houtum et~al.(2007)Van~Houtum, Scheller-Wolf, and
  Yi]{van2007optimal}
G.-J. Van~Houtum, A.~Scheller-Wolf, and J.~Yi.
\newblock Optimal control of serial inventory systems with fixed replenishment
  intervals.
\newblock \emph{Operations Research}, 55\penalty0 (4):\penalty0 674--687, 2007.

\bibitem[Van~Roy(2006)]{van2006performance}
B.~Van~Roy.
\newblock Performance loss bounds for approximate value iteration with state
  aggregation.
\newblock \emph{Mathematics of Operations Research}, 31\penalty0 (2):\penalty0
  234--244, 2006.

\bibitem[Watkins(1989)]{watkins1989learning}
C.~J. C.~H. Watkins.
\newblock \emph{Learning from delayed rewards}.
\newblock PhD thesis, King's College, Cambridge, 1989.

\bibitem[Weinstein and Stason(1977)]{weinstein1977foundations}
M.~C. Weinstein and W.~B. Stason.
\newblock Foundations of cost-effectiveness analysis for health and medical
  practices.
\newblock \emph{New England Journal of Medicine}, 296\penalty0 (13):\penalty0
  716--721, 1977.

\bibitem[Weinstein et~al.(1996)Weinstein, Russell, Gold, Siegel,
  et~al.]{weinstein1996cost}
M.~C. Weinstein, L.~B. Russell, M.~R. Gold, J.~E. Siegel, et~al.
\newblock \emph{Cost-effectiveness in health and medicine}.
\newblock Oxford University Press, 1996.

\bibitem[Werbos(1974)]{werbos1974beyond}
P.~J. Werbos.
\newblock \emph{Beyond regression: New tools for prediction and analysis in the
  behavioral sciences}.
\newblock PhD thesis, Harvard University, 1974.

\bibitem[Werbos(1992)]{werbos1992approximate}
P.~J. Werbos.
\newblock Approximate dynamic programming for real-time control and neural
  modeling.
\newblock \emph{Handbook of Intelligent Control}, pages 493--526, 1992.

\bibitem[{West Virginia Department of Health and
  Human}(2019)]{westvirginia2019hepatitis}
{West Virginia Department of Health and Human}.
\newblock Multistate {Hepatitis A} outbreak.
\newblock 2019.
\newblock URL \url{https://oeps.wv.gov/ob_hav/pages/default.aspx}.

\bibitem[{West Virginia Department of Health and Human
  Resources}(2018)]{dhhr2018}
{West Virginia Department of Health and Human Resources}.
\newblock {DHHR} begins distributing naloxone statewide for first responders,
  2018.
\newblock URL
  \url{https://dhhr.wv.gov/News/2018/Pages/DHHR-Begins-Distributing-Naloxone-Statewide-for-First-Responders---.aspx}.

\bibitem[Williams(1992)]{williams1992simple}
R.~J. Williams.
\newblock Simple statistical gradient-following algorithms for connectionist
  reinforcement learning.
\newblock In \emph{Reinforcement Learning}, pages 5--32. Springer, 1992.

\bibitem[Witten(1977)]{witten1977adaptive}
I.~H. Witten.
\newblock {An adaptive optimal controller for discrete-time Markov
  environments}.
\newblock \emph{Information and Control}, 34\penalty0 (4):\penalty0 286--295,
  1977.

\bibitem[{World Health Organization}(2010)]{world2010h1n1}
{World Health Organization}.
\newblock {H1N1 in post-pandemic period}.
\newblock 2010.
\newblock URL
  \url{http://www.who.int/mediacentre/news/statements/2010/h1n1_vpc_20100810/en/}.

\bibitem[Xin(2017)]{xin2017convexity}
C.~Xin.
\newblock L-natural-convexity and its applications in operations.
\newblock \emph{Frontiers of Engineering Management}, 4\penalty0 (3):\penalty0
  283--294, 2017.

\bibitem[Zaher and Zaki(2014)]{zaher2014optimal}
H.~Zaher and T.~T. Zaki.
\newblock Optimal control theory to solve production inventory system in supply
  chain management.
\newblock \emph{Journal of Mathematics Research}, 6\penalty0 (4):\penalty0 109,
  2014.

\bibitem[Zhou et~al.(2013)Zhou, Chen, and Ge]{zhou2013multi}
W.-Q. Zhou, L.~Chen, and H.-M. Ge.
\newblock A multi-product multi-echelon inventory control model with joint
  replenishment strategy.
\newblock \emph{Applied Mathematical Modelling}, 37\penalty0 (4):\penalty0
  2039--2050, 2013.

\bibitem[Zimmerman et~al.(2001)Zimmerman, Mieczkowski, Mainzer, Medsger,
  Raymund, Ball, and Jewell]{zimmerman2001effect}
R.~K. Zimmerman, T.~A. Mieczkowski, H.~M. Mainzer, A.~R. Medsger, M.~Raymund,
  J.~A. Ball, and I.~K. Jewell.
\newblock Effect of the vaccines for children program on physician referral of
  children to public vaccine clinics: A pre-post comparison.
\newblock \emph{Pediatrics}, 108\penalty0 (2):\penalty0 297--304, 2001.

\bibitem[Zipkin(2008{\natexlab{a}})]{zipkin2008old}
P.~Zipkin.
\newblock Old and new methods for lost-sales inventory systems.
\newblock \emph{Operations Research}, 56\penalty0 (5):\penalty0 1256--1263,
  2008{\natexlab{a}}.

\bibitem[Zipkin(2008{\natexlab{b}})]{zipkin2008structure}
P.~Zipkin.
\newblock On the structure of lost-sales inventory models.
\newblock \emph{Operations Research}, 56\penalty0 (4):\penalty0 937--944,
  2008{\natexlab{b}}.

\end{thebibliography}
\end{document}